%% file: main.tex
\PassOptionsToPackage{dvipsnames}{xcolor}
\documentclass{amsart}

\input{header}

\newcommand{\ndots}{\!\!\!\!\cdots\!\!\!\!}
\newcommand{\bdots}{\cdots\!\!\!\!}
\newcommand{\edots}{\!\!\!\!\cdots}

\renewcommand{\P}{\mathcal P}

\newcommand{\mcolor}{OrangeRed}
\newcommand{\msays}[1]{{\footnotesize\textcolor{\mcolor}{\textbf{M:} #1}}}

\newcommand{\acolor}{ForestGreen}
\newcommand{\asays}[1]{{\footnotesize\textcolor{\acolor}{\textbf{A:} #1}}}

\newcommand{\TODO}{\textcolor{red}{\footnotesize TODO}}

\DeclareMathOperator{\drop}{drop}

\makeatletter
\newcommand{\addresseshere}{%
  \enddoc@text\let\enddoc@text\relax
}
\makeatother

\begin{document}

%%%%%%%%%%%%%%%%%%%%%%%%%%%%%%%%%%%%%%%%%%%%%%%%%%%%%%%%%%%%%%%%%%%%%%%%%%%%%%%%%%%%%%
	
\title[The canonical form and scissors congruence of polytopes]{The canonical form, scissors congruence and adjoint degrees of polytopes}
		
\author{T.\ Baumbach}
\address{Mathematics Institute, Technische Universität Berlin, Straße des 17.\ Juni 135, 10623 Berlin, Germany}
\email{baumbach@math.tu-berlin.de }
% \thanks{Tom Baumbach is supported by the Deutsche Forschungsgemeinschaft (DFG, German Research Foundation) under Germany’s Excellence Strategy -- The Berlin Mathematics Research Center MATH\raisebox{0.25ex}{$+$} (EXC-2046/1, project ID 390685689).}

\author{A.\ Freyer}
\address{Fachbereich Mathematik und Informatik, Freie Universit\"at Berlin, Arnimallee 2, 14195 Berlin, Germany}
\email{a.freyer@fu-berlin.de}
% \thanks{Ansgar Freyer is supported by the SPP 2458 ``Combinatorial Synergies'', funded by the Deutsche Forschungsgemeinschaft (DFG, German Research Foundation) -- project ID 539867386.}

\author{J.\ Weigert}
\address{Max-Planck-Institut für Mathematik in den Naturwissenschaften, Inselstraße 22, 04103 Leipzig
Germany}
\email{julian.weigert@mis.mpg.de}
% \thanks{Julian Weigert is supported by the SPP 2458 ``Combinatorial Synergies'', funded by the Deutsche Forschungsgemeinschaft (DFG, German Research Foundation), project ID: 539677510.}

\author{M.\ Winter}
\address{Mathematics Institute, Technische Universität Berlin, Straße des 17.\ Juni 135, 10623 Berlin, Germany}
\email{winter@math.tu-berlin.de}
% \thanks{Martin Winter is supported as Dirichlet Fellow by the Berlin Mathematics Research Center MATH\raisebox{0.25ex}{$+$} and the Berlin Mathematical School, funded by the Deutsche Forschungsgemeinschaft (DFG, German Research Foundation) under Germany’s Excellence Strategy (EXC-2046/1, project ID 390685689).} %, as well as by the SPP 2458 ``Combinatorial Synergies'' -- project ID  539851419, both funded by the Deutsche Forschungsgemeinschaft (DFG, German Research Foundation)}

\subjclass[2010]{51M20, 52B45, 52B12, 14P05} % maybe 52A39

% 51M20 - Polyhedra and polytopes
% 52B11 - $n$-dimensional polytopes
% 52C25 - Rigidity and flexibility of structures
% 52B45 - Dissections and valuations (Hilbert's third problem, etc.)
%
% 52B05 - Combinatorial properties of polytopes and polyhedra
% 52B12 - Special polytopes
% 52B15 - Symmetry properties of polytopes
% 52A39 - Mixed volumes and related topics in convex geometry
% 05C50 - Graphs and linear algebra
% 05C62 - Graph representations

\keywords{Canonical forms, adjoint polynomials, scissors congruence for polytopes, valuations}
		
\date{\today}
\begin{abstract}
We study the canonical form $\Omega$ as a valuation in the context~of scissors congruence for polytopes.
We identify the degree of its numerator -- the adjoint polynomial  $\adj_P$ -- as an important invariant in this context.\nls
More precisely, for a polytope $P$ we define the \emph{degree drop} that measures~how much smaller than expected the degree of the adjoint polynomial of $P$ is.
We show that this drop behaves well under various operations, such as decompositions, restrictions to faces, projections, products and Minkowski sums.
Next we define the \emph{reduced canonical form} $\Omega_0$ and show that it is a translation-invariant 1-homogeneous valuation on polytopes that vanishes if and only if $P$ has positive degree drop.
Using it we can prove that zonotopes can be characterized as the $d$-polytopes that have maximal possible degree drop $d-1$.
We obtain a decomposition formula for $\Omega_0$ that expresses it as a sum of edge-local quantities of $P$.
Finally, we discuss valuations $\Omega_s$ that can distinguish higher values of the degree drop.
\end{abstract}

% \begin{center}
%     {\color{red}
%     \Large
%     [PRELIMINARY COPY]
%     }
% \end{center}

\maketitle

% \todo{Use \texttt{\textbackslash todo} for comments in the margin. \msays{Use your color to indicate author of comment.}}
% \notodo{If you addressed the comment, change it to \texttt{\textbackslash notodo}, it will turn white; comment can now be deleted by author if they agree it was addressed.}
% For inline comments use \msays{\texttt{\textbackslash msays}}
% \tsays{\texttt{\textbackslash tsays}}
% \jsays{\texttt{\textbackslash jsays}}
% \asays{\texttt{\textbackslash asays}}

\input{sec/introduction}

%\newpage

% \par\bigskip
% \parindent 0pt
% \textbf{Funding.} 
% Martin Winter was supported by the Engineering and Physical Sciences Research Council [EP/V009044/1]

\par\bigskip
\noindent
%\textbf{Acknowledgments.} 
\textbf{Funding.} 
Tom Baumbach is supported by the Deutsche Forschungsgemeinschaft (DFG, German Research Foundation) under Germany’s Excellence Strategy -- The Berlin Mathematics Research Center MATH\raisebox{0.25ex}{$+$} (EXC-2046/1, project ID 390685689).

Ansgar Freyer is supported by the SPP 2458 ``Combinatorial Synergies'', funded by the Deutsche Forschungsgemeinschaft (DFG, German Research Foundation) -- project ID 539867386.

Julian Weigert is supported by the SPP 2458 ``Combinatorial Synergies'', funded by the Deutsche Forschungsgemeinschaft (DFG, German Research Foundation), project ID: 539677510.

Martin Winter is supported as Dirichlet Fellow by the Berlin Mathematics Research Center MATH\raisebox{0.25ex}{$+$} and the Berlin Mathematical School, funded by the Deutsche Forschungsgemeinschaft (DFG, German Research Foundation) under Germany’s Excellence Strategy (EXC-2046/1, project ID 390685689). %, as well as by the SPP 2458 ``Combinatorial Synergies'' -- project ID  539851419, both funded by the Deutsche Forschungsgemeinschaft (DFG, German Research Foundation)}

%%%%%%%%%%%%%%%%%%%%%%%%%%%%%%%%%%%%%%%%%%%%%%%%%%%%%%%%%%%%%%%%%%%%%%%%%%%%%%%%%%

\bibliographystyle{abbrv}
\bibliography{literature}
\addresseshere

\newpage

\input{sec/appendix}

\end{document}

%% file: header.tex
\usepackage[T1]{fontenc}
\usepackage[utf8]{inputenc}
\usepackage[USenglish]{babel}
\usepackage{comment}
\usepackage[
	bookmarks=true,
	plainpages=false,
	linktocpage,
	colorlinks=true,
	citecolor=green!80!black,
	linkcolor=red!70!black,
	filecolor=magenta,
	urlcolor=magenta,
	breaklinks,
	pdfauthor={Martin Winter},
]{hyperref}

\usepackage{amsmath,amsthm}
\usepackage{calc, mathtools} %calc for \widthof, mathtools for \mathclap

\iftrue
\usepackage{amssymb} % incompatible with mathdesign
\else
\usepackage[charter,cal=cmcal]{mathdesign}
\fi

\usepackage[
    colorinlistoftodos,
    backgroundcolor=yellow,
    textsize=footnotesize %, disable
]{todonotes}
\newcommand{\notodo}[1]{\todo[backgroundcolor=white]{#1}}

\usepackage{stmaryrd} % for lightning symbol

\usepackage{color} % colors
\usepackage[dvipsnames]{xcolor}
\usepackage{centernot} % crossing out symbols correctly
\usepackage{array} % arrays
\usepackage{enumitem,moreenum} % numbered lists where each item can be labeled
\usepackage[noadjust]{cite} % for grouping references in the text which are adjacent in the bibliography
\usepackage[nameinlink,capitalize,noabbrev]{cleveref}
\usepackage{nicefrac}

\usepackage{tikz-cd} 

\usepackage[font=small,labelfont=bf]{caption}

\usepackage{blkarray}

\usepackage{inconsolata}

\usepackage{accents}
\newlength{\dhatheight}

\newcommand{\Def}[1]{\textcolor{Periwinkle}{\textit{#1}}}

\newcommand{\ZZ}{\mathbb{Z}}    % integers                             
\newcommand{\RR}{\mathbb{R}}    % real numbers                      
\newcommand{\F}{\mathcal{F}}    % face lattice  
                   
\newcommand{\Sph}{\mathbb{S}}    % sphere

\newcommand{\bs}[1]{\boldsymbol{#1}}

\newcommand{\hatij}{{\ensuremath{\kern1pt\widehat{\kern-1pt\imath\kern-2.3pt\jmath}}}}
\newcommand{\Tsymb}{\top}
\newcommand{\T}{^{\Tsymb}}
\newcommand{\mT}{^{-\Tsymb}}

\newcommand*{\horzbar}{\rule[.5ex]{2.5ex}{0.4pt}}

\def\:{\colon}

\newcommand{\cupdot}{\mathbin{\mathaccent\cdot\cup}}

\newcommand{\labelstyle}[1]{\upshape(\textit{#1})}
\newcommand{\mylabel}{\labelstyle{\roman*}}
\newenvironment{myenumerate}{\begin{enumerate}[label=\mylabel]}{\end{enumerate}}

\def\itm#1{{\labelstyle{\romannumeral#1\relax}}}

\def\nlspace{\nolinebreak\space}
\def\nls{\nlspace}

\newcommand{\freespace}{\kern.07em} 
\newcommand{\free}{\freespace\cdot\freespace} 

\newcommand{\ul}[1]{\underline{\smash{#1}}}

\numberwithin{equation}{section}

\newtheoremstyle{mythmstyle} % name
    {\parsep}                    % Space above
    {\parsep}                    % Space below
    {\itshape}                   % Body font
    {}                           % Indent amount
    {\bfseries\scshape}          % Theorem head font
    {.}                          % Punctuation after theorem head
    {.5em}                       % Space after theorem head
    {}  % Theorem head spec (can be left empty, meaning ‘normal’)
\newtheoremstyle{mydefstyle} % name
    {\parsep}                    % Space above
    {\parsep}                    % Space below
    {}                   % Body font
    {}                           % Indent amount
    {\mdseries\scshape}          % Theorem head font
    {.}                          % Punctuation after theorem head
    {.5em}                       % Space after theorem head
    {}  % Theorem head spec (can be left empty, meaning ‘normal’)

\theoremstyle{theorem}
\newtheorem{theorem}{Theorem}[section]
\newtheorem{corollary}[theorem]{Corollary}
\newtheorem{lemma}[theorem]{Lemma}
\newtheorem{proposition}[theorem]{Proposition}

\theoremstyle{definition}

\newtheorem{example}[theorem]{Example}
\newtheorem{remark}[theorem]{Remark}
\newtheorem{question}[theorem]{Question}
\newtheorem{observation}[theorem]{Observation}

\crefname{theorem}{Theorem}{Theorems}
\crefname{proposition}{Proposition}{Propositions}
\crefname{lemma}{Lemma}{Lemmas}
\crefname{corollary}{Corollary}{Corollaries}
\crefname{remark}{Remark}{Remarks}
\crefname{example}{Example}{Examples}
\crefname{definition}{Definition}{Definitions}
\crefname{problem}{Problem}{Problems}
\crefname{observation}{Observation}{Observation}
\crefname{construction}{Construction}{Construction}

\theoremstyle{theorem}
\newtheorem{innercustomgeneric}{\customgenericname}
\providecommand{\customgenericname}{}
\newcommand{\newcustomtheorem}[2]{%
  \newenvironment{#1}[1]
  {%
   \renewcommand\customgenericname{#2}%
   \renewcommand\theinnercustomgeneric{##1}%
   \innercustomgeneric
  }
  {\endinnercustomgeneric}
}

\newcustomtheorem{theoremX}{Theorem}
\newcustomtheorem{lemmaX}{Lemma}
\newcustomtheorem{conjectureX}{Conjecture}

\DeclareMathOperator{\aff}{aff}
\DeclareMathOperator{\conv}{conv}
\DeclareMathOperator{\cone}{cone}

\DeclareMathOperator{\GL}{GL}

\DeclareMathOperator{\Id}{Id}

\DeclareMathOperator{\vol}{vol}  
\DeclareMathOperator{\len}{len}  	
\DeclareMathOperator{\Int}{int}

\DeclareMathOperator{\adj}{adj}

\let\eset=\varnothing

\let\<=\langle
\let\>=\rangle

\def\...{...}
\newcommand{\shortStyle}{\textit}
\newcommand{\ie}{\shortStyle{i.e.,}}

\newcommand{\eg}{\shortStyle{e.g.}}

\newcommand{\cf}{\shortStyle{cf.}}

\let\angle=\measuredangle
\makeatletter
\renewcommand*{\eqref}[1]{%
  \hyperref[{#1}]{\textup{\tagform@{\ref*{#1}}}}%
}
\makeatother

%% file: sec/introduction.tex
\section{Introduction}

The \emph{canonical form} $\Omega$ of a polytope $P\subset\RR^d$ is a rational function that measures the volume of the polar dual as a function of the polarization point $x\in\Int(P)$: % $x\in\RR^d$:
%
% The \emph{canonical form} $\Omega$ is a valuation defined for convex polytopes $P\subset\RR^d$.
% It measures the volume of the polar dual dependent on the polarization point $x\in\RR^d$:
%
$$\Omega(P;x):=d!\vol(P-x)^\circ=\frac{\adj_P(x)}{\prod_{F\subset P} L_F(x)}.$$
%
%This expression can shown to be a rational function in $x$. 
%\todo{\tsays{Maybe say $F \subset P$ is a facet?} \msays{I wanted to provide here only a bare minimum of info. Details are in \cref{sec:canonical_form}. So it was at least not an oversight, but intentional. Still a matter of taste I guess.}}%
Its numerator $\adj_P$ is known as the \emph{adjoint polynomial} of $P$. % (see \cref{sec:canonical_form} for details). %, and whose denominator is the product of the facet linear forms $L_F$ .
The canonical form~was initially introduced by Arkani-Hamed, Bai and Lam \cite{arkani2017positive} in the study of positive geo\-metries and scattering amplitudes in quantum field theory.
The adjoint polynomial in turn arose in geometric modeling, initially defined for polygons by Wachspress \cite{wachspress1975rational}, and later generalized to general polytopes and polyhedra by Warren \cite{warren1996barycentric}.\nls
Both concepts have since emerged in a number of other geometric and algebraic contexts, including~geo\-metric rigidity theory \cite{winter2024rigidity}, intersection theory \cite{aluffi2013segre,aluffi2016segre}, and algebraic~statistics \cite{kohn2018moment}.
%
%\todo{\msays{Add more other contexts ...}}

\iftrue % switch: full valuation definition

In this article we focus on the non-trivial fact that the canonical form defines~also a \textit{simple valuation} on convex polytopes.
%We in particular studying $\Omega$ as a tool for \textit{polytope decompositions} and \textit{scissors congruence}.
%As a valuation it behaves additively under decomposition:
In the form relevant to us, this means~that whenever $P_1,...,P_n\subset\RR^d$ are convex polytopes with disjoint interiors, so that $P_1\cupdot\cdots\cupdot P_n$ is convex as well, we have
%
%We in particular studying $\Omega$ as a tool for \textit{polytope decompositions} and \textit{scissors congruence}.
%As a valuation it behaves additively under decomposition:
%
$$\Omega(P_1\cupdot\cdots\cupdot P_n)=\Omega(P_1)+\cdots+\Omega(P_n).$$
%
%With this property the canonical form behaves in many ways similar to volume,\nls and we shall

\else

In this article we focus on the non-trivial fact that the canonical form defines~also a \textit{simple valuation} on convex polytopes.
%We in particular studying $\Omega$ as a tool for \textit{polytope decompositions} and \textit{scissors congruence}.
%As a valuation it behaves additively under decomposition:
This means that $\Omega$ vanishes on polytopes with empty interior, and we have
%
%\todo{\msays{@Ansgar: Do you think it is important to give this definition here? We never use it in this form, right? Or do you need it in your section? Another reason to be not too precise about ``simple'' here is that ``vanishing on low-dimensional polytopes'' does not follow from the given definition. Maybe, if you think we need the details, we can include them in \cref{sec:canonical_form}?}}
%
\[
\Omega(P\cup Q) = \Omega(P) + \Omega(Q) - \Omega(P\cap Q)
\]
for any convex polytopes $P,Q\subset\RR^d$ such that $P\cup Q$ is convex.
In particular, it follows that whenever $P_1,...,P_n\subset\RR^d$ are convex polytopes with disjoint interiors, then
%
%We in particular studying $\Omega$ as a tool for \textit{polytope decompositions} and \textit{scissors congruence}.
%As a valuation it behaves additively under decomposition:
%
$$\Omega(P_1\cupdot\cdots\cupdot P_n)=\Omega(P_1)+\cdots+\Omega(P_n).$$
%
%With this property the canonical form behaves in many ways similar to volume,\nls and we shall

\fi

A number of famous results on polytope decompositions and scissors invariance have their roots in valuation theory, perhaps most notably, Max Dehn's resolution of Hilbert's third problem on the equidecomposability of 3-polytopes \cite{boltianski,hadwiger}.
%
%Yet another example is given below.
One objective of this article is to explore what similarly flavored results can be derived using specifically the valuative nature of the canonical form.
%
%
%We shall explore in how far the canonical form can serve as a likewise useful tool for such questions.
%
%where $\cupdot$ denotes that the $P_i$ have pair-wise disjoint interior.
%We in particular explore $\Omega$ as a tool for studying \textit{polytope decompositions} and \textit{scissors congruence}.
%As a valuation it behaves additively under decomposition:
%It thereby behaved~in many ways similar to the volume or other valuations, such as mixed volumes (including mean width and Euler characteristic), the surface area measure and interior~lattice points.
%Over more than a century this property the theory of valuations has been developed to a rich subject and many powerful tools have emerged from it.
%
%A number of famous results on polytope decompositions have their roots in~valuation theory, 
%Perhaps the most notable example is Max Dehn's resolution of Hilbert's third problem on the equidecomposability of 3-polytopes. 
%
%We give~here another example that will recur throughout this article:
%
%The following instructive example will recur throughout the text and should provide more context  
%
% The following is an instructive~example 
%     %will recur throughout the text and 
% that contains the core ideas we are aiming for:
%
The following recurring example should give a good idea of what we are aiming for:
%
%Another example, that is in this spirit, and that will also recur throughout the text, is the following:
%Below we give another instructive example for the kind of result we re looking for, that will also recur throughout the text.
%Below we give another example that is in the same spirit and that will recur throughout the text.
%The following instructive example will recur throughout the article.

\newcommand{\SAM}{\mathrm S}
\newcommand{\OSAM}{\smash{\mathrm{\tilde S}}}

\begin{example}
    \label{ex:cs}
    Let $P_1,...,P_n\subset\RR^d$ be \Def{centrally symmetric} polytopes (that means, $P_i$ and $-P_i$ are translates of each other).
    It is a remarkable and somewhat surprising fact that if a larger \emph{convex} polytope $P=P_1\cupdot \cdots\cupdot P_n$ is built from the $P_i$ without overlapping interiors, then it is itself centrally symmetric.
    
    This can be shown in different ways, but
    one of the most elegant proofs involves valuation theory: let $\SAM(P)$ be the so-called \emph{surface area measure} of $P$. That is,\nls $\SAM(P)$ is the measure on $\Sph^{d-1}$ that to each subset $N\subset\Sph^{d-1}$ assigns the area of the part~of the surface of $P$ that has unit normals that lie in $N$. 
    In our case, where~$P$~is~a~polytope, $\SAM(P)$ is the sum of finitely many Dirac measures, one per facet of $P$.
    %One can show that $\SAM$ is a valuation.
    Based on this we define the \emph{odd surface area measure} $\OSAM(P):=\SAM(P)-\SAM(-P)$.
    One can~show that $\OSAM$ defines a simple valuation.
    Clearly, since all $P_i$ are centrally symmetric,\nls we have $\OSAM(P_i)=0$.
    Less obvious however, a converse holds as well: if $\OSAM(P)=0$, then $P$ is centrally symmetric.
    This can be proven using Minkowski's uniqueness theorem.
    
    The fact that $P=P_1\cupdot \cdots\cupdot P_n$ inherits its central symmetry from the pieces $P_i$ is now reduced to a one-line argument:
    $$\OSAM(P)=\OSAM(P_1\cupdot\cdots\cupdot P_n) = \sum_i \OSAM(P_i) = 0.$$

    This technique can actually show something stronger: suppose that $P$ and $Q$ are \Def{translation scissors congruent}, that is, $P=P_1\cupdot\cdots\cupdot P_n$ and $Q=Q_1\cupdot \dots\cupdot Q_n$ so that $P_i=Q_i+t_i$ are related through translations along vectors $t_i\in\RR^d$.
    Clearly, $\OSAM$ is \Def{translation-invariant}, that is $\OSAM(P+t)=\OSAM(P)$ for all $t\in\RR^d$.
    Through
    $$\OSAM(P)=\sum_i \OSAM(P_i) = \sum_i \OSAM(Q_i+t_i) = \sum_i\OSAM(Q_i)=\OSAM(Q)$$
    %coordinates
    we therefore find that $P$ is centrally symmetric if and only if $Q$ is. 
    More remarkably, we made no assumptions about whether the $P_i$ and $Q_i$ are centrally symmetric.
    We say that central symmetry is a \Def{translation scissors invariant}.
\end{example}

%In this article we shall explore in how far the canonical form can serve as a tool for studying \textit{polytope decompositions} and \textit{scissors congruence}.
%
%Part of the motivation behind the research for this article was to explore whether similar decomposition results can be obtained specifically using the canonical form. 
%
% valuative nature of the canonical form.

%Initially two obstacles are encountered:
%
When attempting to prove similar decomposition results using the canonical~form, 
initially two obstacles are encountered:
\begin{myenumerate}
    \item A typical way to obtain translation scissors invariants is by identifying them as the vanishing of a valuation (such as central symmetry is characterized as the vanishing of the symmetrized surface area measures in \cref{ex:cs}). 
    However, $\Omega$ only vanishes on polytopes with empty interior.
    %However, $\Omega$ never vanishes (on polytopes with non-empty interior).
    \item $\Omega$ is \emph{not} translation-invariant, and hence one might not expect new results about translation scissors congruence from its study.
\end{myenumerate}
In this article we shall overcome both of these obstacles and thereby discover~interesting new mathematics.

For \itm1 we show that instead of the vanishing of $\Omega$, the degree of the numerator (\ie\ the degree of the adjoint polynomial $\adj_P$) can be controlled in decompositions and presents itself as a new and interesting polytope invariant.
From this we derive the quantity $\drop(P)$, a purely affine notion that measures how much~``smaller~than expected'' the adjoint degree is.
It turns out that this \emph{adjoint degree drop} is subtly related to parallelisms in the face structure of $P$ and in many ways behaves better than the adjoint degree itself. 
Lastly, we note that the adjoint degree comes up independently as a relevant quantity also in other contexts, such as in the construction of polynomial barycentric coordinates \cite{baumbachwinter2025rational}.
%
%\todo{\msays{Is it okay for everyone if we cite Tom's and my upcoming paper here as ``in preparation''?}}
%
%\notodo{\tsays{I’d suggest avoiding the word “previously” here, as Warren didn’t explicitly address this point, and our paper will appear at a later date. Perhaps simply state that it is also relevant for GBCs.} \msays{Agreed, will do ...}} 
%
%or barycentric coordinates for zonotopes.
%Here we observe the adjoint degree (and its drop) as a new interesting invariant for polytopes that seems non-trivially linked to geometric and combinatorial properties.\nls
%We prove some characterizations and hint to interesting connections. %, but much will be left open for future research.

To overcome \itm2 we construct from $\Omega$ a new truly translations-invariant valuation $\Omega_0$.
It will serve us as a powerful tool in the study of adjoint degrees and will also yield some results on translation scissors congruence.
For example, we will show that being a zonotope in dimension three is translations scissors invariant.
We will also prove a decomposition formula for $\Omega_0$ that allows us to compute its value from data that is only local to the edges of $P$ (that is, edge lengths and tangent cones).
%This in turn yields a curious determinantal identities for quadratic forms.\nls
%{\color{lightgray}Lastly,\nls we generalize $\Omega_0$ to valuations $\Omega_s,s\ge 1$ and explore their properties.}\todo{\msays{Ansgar's section ready for this?}}
%whose precise properties certainly pose content for future research.

%Many new questions are opened up by our investigations and so a foundation for future research is established.

\subsection{Overview of the paper} 

In \cref{sec:canonical_form} we recall the \textit{canonical form} $\Omega$ and its essential properties in so far as they are relevant for this article.

Since the vanishing of $\Omega$ does not encode useful information about the polytope, we instead study the degree of its numerator -- the \textit{adjoint polynomial}. 
In \cref{sec:degree_drops} we recall what is known about this degree and introduce the \textit{degree drop}, denoted $\drop(P)$, that measures how much smaller than expected the adjoint degree is.
We show that the degree drop behaves well under polytope decomposition (\cref{res:drop_tiling}) and with respect to several other polytope operations (\cref{res:drop_properties}). We compute the drop for some classes of polytope.

In \cref{sec:Omega0} we introduce the \textit{reduced canonical form} $\Omega_0$, which is a translation-invariant valuation derived from the homogenization of $\Omega$.
By construction, $\Omega_0(P)$ vanishes if and only if $\drop(P)>0$, which shows that $\drop(P)>0$ is a translation scissors invariant.
The main result of \cref{sec:Omega0} is the characterization of zonotopes as those polytopes that have the maximal possible degree drop (\cref{res:zonotopes_iff_drop_d_1}).
As a consequence we find that the exact value of $\drop(P)$ (not only whether it is zero or non-zero) is a translation scissors invariant in dimension $d\le 3$ (\cref{res:d_le_3_preserves_drop}).

In \cref{sec:homogeneous} we prove that $\Omega_0$ is a 1-homogeneous valuation.
This opens up~its study to the powerful tools of valuation theory.
For example, we obtain that $\Omega_0$~is Minkowski additive, which allows us to construct many more polytopes~with non-zero drop.
The main work of this section is the discussion and derivation~of~a~remar\-kable decomposition formula for $\Omega_0$ (\cref{res:edge_decomposition}).
Applied to~simplices,\nls it~yields a curious identity involving quadratic forms and determinants (\cref{res:simplex_qudratic_forms}).

In \cref{sec:central_inversion} we study how $\Omega_0$ behaves under central inversion of the polytope.
Based on this, in dimension $d=3$ we achieve a full characterization of the polytope classes defined by their value of $\drop(P)$ (\cref{res:characterize_drop_3D}).

A disadvantage of $\Omega_0$ is that it does not ``see'' the exact value of the degree drop, but only whether it is zero or not.
In \cref{sec:higher_valuations} we introduce a natural generalization $\Omega_s$ that too is a valuation, and vanish on polytopes precisely if $\drop(P)\ge s$.\nls
These valuations are no longer translation-invariant, but exhibit a polynomial~behavior under translation in the sense of Khonvaskii and Pukhlikov \cite{khovanskiipukhlikov}.

Lastly, some more elaborate, yet elementary or well-known computations that~we felt would distract from our main points have been moved to \cref{sec:appendix_simplices,sec:appendix_facet_restriction,sec:appendix_adjoint_product,sec:appendix_homogenized_Omega,sec:appendix_lij_adj_Tij,sec:appendix_Dij,sec:appendix_triangle_adjoint}.

%\newpage

\section{Dual volume and the canonical form $\Omega$}
\label{sec:canonical_form}

For a convex $d$-dimensional polyhedron $P\subseteq\RR^d$
$$
P^\circ:=\{x\in\RR^d\mid \text{$\<x,y\>\le 1$ for all $y\in P$}\}.
$$
denotes the \emph{polar dual} with \textit{polarization center} at the origin.
In this sense, $(P-x)^\circ$ is the polar dual with polarization center at $x\in\RR^d$.
For $x\in\Int(P)$, the polar dual is a convex polytope (in particular, bounded) and its volume is a rational~function in $x$.
Based on this, we define the \Def{canonical form $\Omega(P;x)$} as the unique~rational~func\-tion that evaluates to
\begin{equation}
%\label{eq:omega_def}
\Omega(P;x) = d! \vol(P-x)^\circ,
\end{equation}
for $x\in \Int(P)$, and is constant $\Omega(P)\equiv0$ if $\Int(P)=\eset$.
Here, $\vol$ denotes the~Euclidean volume, and the factor $d!$ is included to align with the definition given in \cite[Definition 2.5]{gao2024dual} where the measure is normalized to evaluate to 1 on the standard simplex.
This is moreover consistent with the definition of canonical form~in~Posi\-tive Geometry~\cite{arkani2017positive} where it is more commonly defined as a rational $d$-form. 
As~we~are solely concerned with its algebraic properties, we prefer to work with~$\Omega$~as~a~rational function.
The precise relation is clarified by \cite[Theorem 5.3]{gao2024dual}.

If $\Int(P)\not=\eset$ and $x$ approaches the boundary of $P$, the volume of the polar dual diverges.
Hence $\Omega(P;x)$ has poles along the facet hyperplanes.
One can show that these are the only poles, and by factoring them out we obtain a polynomial known as the \Def{adjoint~polynomial} or just \Def{adjoint $\adj_P$} of $P$:
\begin{equation}
%\label{eq:canonical_form}
\label{eq:Omega_represenation}
\Omega(P;x)=\frac{\adj_P(x)}{\prod_{F\subset P} L_F(x)}.    
\end{equation}
Here the product in the denominator iterates over all facets $F$ of $P$, and $L_F(x):=h_F-\<u_F,x\>$ is the facet-defining linear form, with $u_F\in\Sph^{d-1}$ being the facet's unit normal vector (pointing outwards), and $h_F\in\RR$ the height of $F$ over the origin.
%
% Note that the precise value of the adjoint depends on the choice of lengths for the normal vectors.
% %
% In other literature it is common to define the adjoint only up to a scalar factor.
% For this article we shall fix a scale by requiring that our normals are normalized to length one whenever the exact value of the adjoint is needed.
%
%We emphasize that~our normal vectors are normalized to length one, which is important to define the~ad\-joint precisely, not only up to a scalar factor, as is typical in other literature.

%\todo{\msays{The volume definition of $\Omega$ actually makes sense also for polyhedra since their duals are compact.}}

Computing the canonical forms and adjoints for simplices and simplicial cones~is a standard exercise. 
Here we state the result (see \cref{sec:appendix_simplices} for derivations):

\begin{example}
    \label{ex:simplices}
    For a simplex $\Delta:=\{x\in\RR^d\mid \<x,u_i\>\le h_i\}$ we have %\todo{\msays{I am thinking to remove the absolute value and instead assume that we choose an appropriate order of the columns. The reason is both for visual reasons, and because we never ever use it again with the absolute value because later we always either assume a good order, or find that it is not necessary.}}
    \begin{equation}
    \label{eq:adjoint_simplex}
    \Omega(P;x) = \frac{
        % \left|\det\begin{pmatrix}
        %     \horzbar \;u_0\; \horzbar & \!\!h_0
        %     \\
        %     \vdots & \!\!\vdots
        %     \\
        %     \horzbar \;u_d\; \horzbar & \!\!h_d
        % \end{pmatrix}\right|
        \adj_\Delta
    }{\;\;\,\prod_i (h_i-\<x,u_i\>)},
    \quad\text{with }\!
    \adj_\Delta=\left|\,
        % \det\!\begin{pmatrix}
        %     \horzbar \;u_0\; \horzbar & \!\!h_0
        %     \\
        %     \vdots & \!\!\vdots
        %     \\
        %     \horzbar \;u_d\; \horzbar & \!\!h_d
        % \end{pmatrix}
        \det\!\begin{pmatrix}
            | & & | \\[-0.5ex]
            u_0 & \ndots & u_d \\
            | & & | \\[1ex]
            h_0 & \ndots & h_d
        \end{pmatrix}
    \right|.
    \end{equation}
    Similarly, for a simplicial cone $C:=\{x\in\RR^d\mid \<x,u_i\>\le 0\}$ we have
    \begin{equation}
    \label{eq:adjoint_simplicial_cone}
    \Omega(C;x) = \frac{
        % \left|\det\begin{pmatrix}
        %     \horzbar \;u_1\; \horzbar
        %     \\
        %     \vdots
        %     \\
        %     \horzbar \;u_d\; \horzbar
        % \end{pmatrix}\right|
        \adj_C
    }{(-1)^{d}\prod_i \<x,u_i\>},
    \quad\text{with }\!
    \adj_C=\left|\,
        % \det\!\begin{pmatrix}
        %     \horzbar \;u_1\; \horzbar
        %     \\
        %     \vdots
        %     \\
        %     \horzbar \;u_d\; \horzbar
        % \end{pmatrix}
        \det\!\begin{pmatrix}
            | & & | \\[-0.5ex]
            u_1 & \ndots & u_d \\
            | & & | \\[1ex]
        \end{pmatrix}
    \right|.
    \end{equation}
\end{example}
    
Note that according to \cref{ex:simplices} the canonical form of a simplex or simplicial cone is independent of the length of the normals.
The adjoint however is a number and scales with the lengths of the normals.
In the following, whenever we refer~to~the precise value of an adjoint without specifying normals, we shall assume that they are normalized to length one.
This is relevant, for example, when we discuss~geometric interpretations for the adjoint (see \cref{sec:decomposition_on_simplices} and \cref{sec:appendix_triangle_adjoint}).

\iftrue % switch: full definition of valuation

A most remarkable fact about canonical forms, and core motivation for this~article, is that $\Omega$ defines a \Def{simple valuation} on convex polyhedra. %This means, whenever $P,Q,P\cap Q,P\cup Q\in\mathcal P^d$, then
%
%$$\Omega(P\cup Q) + \Omega(P\cap Q) = \Omega(P)+\Omega(Q).$$
%
That it is a \emph{valuation} means that whenever $P,Q,P\cap Q$ and $P\cup Q$ are convex polyhedra, then
$$\Omega(P)+\Omega(Q)=\Omega(P\cap Q)+\Omega(P\cup Q).$$
It is \emph{simple} because it vanishes on polytopes with empty interior.
That dual volumes yield such simple valuations had been observed before (\eg\ proven by Ludwig \cite{ludwig2002valuations}),
but was rediscovered in the context of Positive Geometry by Arkani-Hamed, Bai~and Lam \cite[Equation 3.6]{arkani2017positive}.
%This valuation property will be relevant to us in the following form:
%
For us, the relevant consequence of these properties (which one may as well take as the definition of simple valuation) is the following:

% \todo{\msays{Add full definition of ``valuation'' and ``simple''}}%
% {\color{lightgray}
% This means that $\Omega$ vanishes on polytopes with empty interior, and we have
% \[
% \Omega(P\cup Q) = \Omega(P) + \Omega(Q) - \Omega(P\cap Q)
% \]
% for any convex polytopes $P,Q\subset\RR^d$ such that $P\cup Q$ is convex.
% In particular, it follows that whenever $P_1,...,P_n\subset\RR^d$ are convex polytopes with disjoint interiors, then
% }

\begin{theorem}%[{Arkani-Hamed, Bai, Lam; 2017 \cite{arkani2017positive}}]
    \label{res:Omega_valuation}
    %Let \Def{$\mathcal P^d$} be the set of convex polytopes in $d$-dimensional Euclidean space.
    Whenever $P_1,...,P_n\subseteq\RR^d$ and $P_1\cupdot\cdots\cupdot P_n$ are convex polyhedra (where $\cupdot$ denotes that the $P_i$ have pair-wise disjoint interior), then
    \begin{align}
        \label{eq:valuation}
        \Omega(P_1\cupdot\cdots\cupdot P_n) = \Omega(P_1)+\cdots+\Omega(P_n).
    \end{align}    
    %
    %where $\cupdot$ denotes that the $P_i$ have pair-wise disjoint interior.
    %
    %\notodo{\asays{$\Int P_i\neq \eset$ ?\msays{Not relevant anymore, right?}}}
\end{theorem}

\else

\todo{\msays{Add full definition of ``valuation'' and ``simple''}}%
A most remarkable fact about the canonical form, and core motivation for this article, is that it defines a valuation on convex polyhedra. %This means, whenever $P,Q,P\cap Q,P\cup Q\in\mathcal P^d$, then
%
%$$\Omega(P\cup Q) + \Omega(P\cap Q) = \Omega(P)+\Omega(Q).$$
%
That dual volumes~yield valuations had been known before (\eg\ proven by Ludwig \cite{ludwig2002valuations}),
but was rediscovered in the context of Positive Geometry by Arkani-Hamed, Bai and Lam \cite{arkani2017positive}:
%This valuation property will be relevant to us in the following form:

\begin{theorem}%[{Arkani-Hamed, Bai, Lam; 2017 \cite{arkani2017positive}}]
    \label{res:Omega_valuation}
    %Let \Def{$\mathcal P^d$} be the set of convex polytopes in $d$-dimensional Euclidean space.
    Whenever $P_1,...,P_n\subseteq\RR^d$ and $P_1\cupdot\cdots\cupdot P_n$ are convex polyhedra (where $\cupdot$ denotes that the $P_i$ have pair-wise disjoint interior), then
    \begin{align}
        \label{eq:valuation}
        \Omega(P_1\cupdot\cdots\cupdot P_n) = \Omega(P_1)+\cdots+\Omega(P_n).
    \end{align}    
    %
    %where $\cupdot$ denotes that the $P_i$ have pair-wise disjoint interior.
    %
    %\notodo{\asays{$\Int P_i\neq \eset$ ?\msays{Not relevant anymore, right?}}}
\end{theorem}

\fi

In particular, the canonical form (and adjoint) of any polytope or polyhedron~can be computed by decomposing it into simplices and simplicial cones, computing $\Omega$ on the pieces using \cref{ex:simplices}, and adding up the results according to \eqref{eq:valuation}.

The canonical form behaves furthermore well under restriction to facets.
%Important for us (and crucial in Positive Geometry) is the fact that the canonical form of a 
If~$F\subset P$ is a facet of $P$, then the canonical form of $F$ (considered here as a full-dimensional polytope in the affine span of $F$) can be obtained as the residue of $\Omega(P)$ at $F$.
That means, we drop the factor $L_F$ from the denominator and restrict to $F$:
\iffalse % switch: do we need a factor?
\begin{equation}
\label{eq:Omega_on_face}
\Omega^{d-1}(F;x) 
%= \operatorname{Res}_F(\Omega^d(P)) 
= \frac1d \frac{\adj_P(x)|_F}{\prod_{\substack{G\subset P\\G\not=F}} L_G|_F}.    
\end{equation}
where the factor of $1/d$ is necessary due to the volume normalization in \eqref{eq:omega_def}.%
\todo{\msays{Is the factor $1/d$ correct in \eqref{eq:Omega_on_face}?}}%
\nls
Also, we include here a superscript $d-1$ to indicate that $\Omega^{d-1}(F)$ is computed within~the ambient space $\aff(F)\simeq\RR^{d-1}$\!\!.
%as a rational function on $\aff(F)\simeq\RR^{d-1}$, the affine span of $F$, rather than on $\RR^d$. 
This distinction is important since $\Omega^d(F)=0$.\nls
%Note that we do not define $\adj_P$ for anything else but full-dimensional polytopes.
We~will usually not include these superscripts when they are clear from context.
Recall also that the adjoint is only defined if $\Int(P)\not=\eset$, and so an analogous indication~of~dimension is not necessary.
%
%\notodo{\asays{I think we need to signify in which ambient space we take the adjoint. Otherwise, one might think that this equation is a contradiction to Lemma 2.2. This would also make the formular in Theorem 5.5 a bit easier to parse}\msays{Fixed}}
%
\else
\begin{equation}
\label{eq:Omega_on_face}
\Omega^{d-1}(F) 
%= \operatorname{Res}_F(\Omega^d(P)) 
%= L_F\Omega^d(P)|_F 
= \frac{\adj_P|_F}{\prod_{\substack{G\subset P\\G\not=F}} L_G|_F}.    
\end{equation}
%
%\todo{\msays{Maybe \eqref{eq:Omega_on_face} needs a factor?}}%
We include here superscripts to indicate that $\Omega^{d-1}(F)$ is computed within the~ambient space $\aff(F)\simeq\RR^{d-1}$\!.
%as a rational function on $\aff(F)\simeq\RR^{d-1}$, the affine span of $F$, rather than on $\RR^d$. 
This distinction is important since \mbox{$\Omega^d(F)=0$.
%Note that we do not define $\adj_P$ for anything else but full-dimensional polytopes.
We} will usually not include these superscripts when they are clear from context.
Recall also that the adjoint is only defined if $\Int(P)\not=\eset$, and so an analogous indication~of~dimension is not necessary.
Finally, we highlight the subtlety that \eqref{eq:Omega_on_face} holds without a scaling factor only because we assume the normal $u_F$ to be of unit~length~(since we found our setting to be not quite standard, a proof of this is given in \cref{sec:appendix_facet_restriction}).
\fi

Finally, for later use we also record here the behavior of $\Omega$ under affine transformations.

\begin{lemma}
    \label{res:Omega_trafo}
    %$\Omega$ is affine-covariant.
    %That is, 
    For an invertible linear transformation $S\in\GL(\RR^d)$ and $t\in \RR^d$ holds
    $$\Omega(SP+t;Sx+t)=|\!\det(S)^{-1}|\, \Omega(P;x).$$
    %
    %$$\Omega(SP+t,x)=\det(S)^{-1}\,\cdot \Omega(P;S^{-1}(x-t))$$
\end{lemma}
\begin{proof}
    This follows from a straightforward computation:
    \begin{align*}
        \Omega(SP+t;Sx+t) 
        &= d! \vol((SP+t)-(Sx+t))^\circ
        \\&= d! \vol(SP-Sx)^\circ
        \\&= d! \vol(S(P-x))^\circ
        \\&= d! \vol(S\mT(P-x)^\circ)
        \\&= d!\, |\!\det(S\mT)| \vol(P-x)^\circ
        \\&= |\!\det(S)^{-1}|\, \Omega(P;x).
    \end{align*}
    We used the well-known fact that $(SP)^\circ=S\mT P^\circ$.
\end{proof}

In order to use $\Omega$ for matters of scissors congruence, we would need to identify certain special values or characteristics of $\Omega$ that are preserved under composition.
As mentioned before, $\Omega$ does not vanish on full-dimensional convex polytopes.
Thus, even though $\Omega(P)=0$ is preserved under composition, this yields no interesting~insights at this point.
In the next section we shall explore how instead the degree of the adjoint polynomial can serve as a useful characteristic.

\section{Adjoint degrees and degree drop}
\label{sec:adjoint_degrees}
\label{sec:degree_drops}

From here on, let $P\subset\RR^d$ be a convex $d$-dimensional polytope with $m$ facets.\nls % $F_1,...,F_m\subset P$.
It is known that ``generically'' (or projectively) the degree of the adjoint of a polytope is \textit{precisely} $m-d-1$ (see 
\cite[Theorem 1]{warren1996barycentric} or also \cite{kohn2020projective}).
%\cite{kohn2020projective}.
%In the right setting, that is, projectively, the adjoint is always of this degree.
This holds, for example,\nls if~$P$ is a simplex (\cf\ \cref{ex:simplices}).
For general polytopes this \Def{expected degree} is still an upper bound on $\deg(\adj_P)$, the actual degree might however be much smaller. % than this.

\begin{example}\label{ex:cube}
    If $\square_d:=[-1,1]^d$ is the $d$-dimensional cube, then one can show that
    $$
    \Omega(\square_d;x) = 
        %\frac{1}{\prod_i (1-x_i)^2}.
        %\frac{2^d}{\prod_i (1-x_i)^2}.
        \frac{\vol(\square_d)}{\prod_i (1-x_i)^2}.
    $$
    The expected adjoint degree is $2d-d-1=d-1$, whereas the actual degree is $0$.
\end{example}

This deficiency in degree turns out to be an interesting and well-behaved invariant that we shall call the \Def{adjoint degree drop} of $P$, or just \Def{drop} for short. We~write
$$\drop(P):= \overbrace{(m-d-1)}^{\mathclap{\text{expected degree}}} - \overbrace{\deg(\adj_P)}^{\mathclap{\text{actual degree}}} \ge 0.$$
%
%Throughout we shall discuss the adjoint degree drop only for polytopes
%
With the degree of a rational function defined as $\deg(p/q):=\deg(p)-\deg(q)$, one can express the drop also as
\begin{equation}
    \label{eq:drop_Omega}
    \drop(P) = -d-1-\deg(\Omega(P)).    
\end{equation}
Note that the drop, like the adjoint polynomial itself, is defined only for polytopes with non-empty interior. We shall restrict to the affine span if necessary.

\iffalse % remark on other interpretations of the drop (projectively and via Wachspress coordss) - 27-06-2025

\begin{remark}
    %Other interpretations for the drop can be given in homogenized coordinates:
    For the homogenized adjoint the drop corresponds to the highest power of $x_0$ that divides $\adj_P(x_0,x)$. %, that is, if $s:=\drop(P)$, then
    %
    %$$\adj_P(x_0,x)=x_0^s\, p(x_0,x),$$
    %
    %where $p$ is a polynomial that is not divisible by $x_0$.
    %Equivalently, since none of the linear forms $L_F$ is divisible by $x_0$, $s$ is maximal so that we can write
    %
    %$$\Omega(P;x_0,x)=x_0^s \frac{p(x_0,x)}{\prod_F L_F},$$
    %
    %for some polynomial $p\in\RR[x_0,x]$.
    %where $p$ is a polynomial that is not divisible by $x_0$.
    %
    %Yet another way to describe 
    Equivalently, the drop is as the multiplicity of the hyperplane at infinity as a component of the zero locus of $\adj_P$ (also known as \emph{adjoint hypersurface} of $P$).
\end{remark}

\begin{remark}
    If $\alpha$ are \emph{Wachspress coordinates} for $P$, then $\drop(P)=\deg(\alpha)-1$.
\end{remark}

\fi

% \begin{example}\label{ex:cube}
%     The adjoint of a  simplex $\Delta$ is a constant polynomial, hence $\drop(\Delta)=0$.
%     A $d$-dimensional cube $\square_d$ has 
%     %
%     $$\Omega(\square_d;x) = \frac{\mathrm{const}}{\prod_i (1-x_i)^2},$$
%     %
%     hence $\drop(\square_d)=d-1$.
% \end{example}

The following is an immediate consequence of the valuative nature of the canonical form and motivates the study of the degree drop as a tool for investigating~polytope decompositions.

\begin{lemma}
    \label{res:drop_tiling}
    $\drop(P_1\cupdot\cdots\cupdot P_n) \ge \min_i \drop(P_i)$.
\end{lemma}
\begin{proof}
    Using \cref{res:Omega_valuation}, we have
    $$\deg\Omega(P_1\cupdot\cdots\cupdot P_n) \overset{\smash{\text{\ref{res:Omega_valuation}}}}= \deg \Big(\sum_i \Omega(P_i)\Big) \le \max_i \deg\Omega(P_i),$$
    where the inequality is the usual behavior of rational degrees under addition.
    The claim then follows by substituting \eqref{eq:drop_Omega}.
\end{proof}
%
\iffalse % obsolete proof using homogenization - 22/06/2025
%
\begin{proof}
    %\msays{Write a better proof.}
    Since the linear forms $L_F$ are not divisible by $x_0$, $s_i:=\drop(P_i)$ is maximal so that we can write
    %
    $$\Omega(P_i;x_0,x)=x_0^{s_i} \frac{p_i(x_0,x)}{\prod_{F\subset P_i} L_F}.$$
    %$$\Omega(P_i;x_0,x)=x_0^{s_i} \frac{p_i(x_0,x)}{q(x_0,x)}.$$
    %
    %where $x_0$ does not divide $p_i$.
    for polynomials $p_i\in\RR[x_0,x]$.
    If $s:=\min_i s_i$ then
    %
    $$
    \Omega(P_1\cupdot\cdots\cupdot P_n)
        = \sum_i \Omega(P_i;x_0,x)
        = x_0^s \sum_i \frac{x_0^{s_i-s}p_i(x_0,x)}{\prod_{F\subset P_i} L_F}
        = x_0^s \frac{p(x_0,x)}{\prod_{P\subset F_1} L_F\cdots \prod_{P\subset F_m} L_F}.
    $$
    %
    We see that $\Omega(P_1\cupdot\cdots\cupdot P_n)$ can be written as the product of $x_0^s$ with a rational function whose denominator does not divide $x_0$, hence $\drop(P_1\cupdot\cdots\cupdot P_n)\ge s$.
\end{proof}
%
\fi

This is already conceptionally similar to \cref{ex:cs}: like central symmetry,\nls a lower bound on the drop, such as $\drop(P)\ge s$, is a property that a composed~polytope inherits from its pieces.
To fully appreciate or make use of this fact one needs to understand what geometric and combinatorial properties of $P$ are encoded~by~the degree drop.
Thus, we consider the following as the central question in the study~of adjoint degrees:
%From the perspective of polytope composition, this alone justifies the study of the adjoint degree as a polytope invariant.
%The central questions are of course:

\begin{question}
    \label{q:drop}
    What geometric/combinatorial properties characterize polytopes with a particular adjoint degree drop?
    %How to characterize/classify polytopes $P$ for which $\drop(P)=s$ with a fixed $s$?
\end{question}

\begin{example}\label{ex:cube_drop}
    According to \cref{ex:cube} the $d$-dimensional cube has $\drop(\square_d)=d-1$.
    In fact, this holds for every parallelepiped (see \cref{res:drop_properties} \ref{it:trafo} below).
    By \cref{res:drop_tiling} every polytope composed from parallelepipeds has drop at least $d-1$. 
    This includes all zonotopes. %In $d=2$ this includes all centrally symmetric 
\end{example}

% \begin{lemma}
%     $\drop(P_1\cupdot\cdots\cupdot P_n) \ge \min_i \drop(P_i)$.
%     %If $P=P_1\cupdot\cdots\cupdot P_n$ and $P_i$ is of drop $s_i$, then $P$ is of drop $s\ge \min_i s_i$.
% \end{lemma}
% %
% \begin{proof}
    
% \end{proof}

We collect some elementary properties of the adjoint degree drop.

\begin{theorem}
\label{res:drop_properties}
Let $P,P_1,...,P_n$ be polytopes of dimensions $d,d_1,...,d_n\ge 1$.
Then
\begin{myenumerate}
    \setlength{\itemsep}{0.5ex}
    % \item \label{it:scissor}
    % $\drop(P_1\cupdot\cdots\cupdot P_n) \ge \min_i \drop(P_i)$.
    \item \label{it:product}
    $\drop(P_1\times\cdots\times P_n) = \sum_i \drop(P_i) + n - 1$.
    \item \label{it:face}
    if $F\subset P$ is a facet, then % of codimension $k$, then 
    $$\drop(F)\ge \drop(P)-1.$$
    If equality hold, then $P$ has a facet $G\not= F$ parallel to $F$. 
    Moreover, if $f\subset P$ is a face of dimension $d-k$, then $\drop(f)\ge \drop(P)-k$.
    %If $F$ is a facet then $\drop(F)\ge \drop(P)- 1$ with equality only if $P$ has a facet parallel to $F$.
    \item \label{it:max_drop}
    $\drop(P)\le d-1$.
    \item \label{it:invariant}\label{it:trafo}
    the degree drop is affinely invariant, that is, if $S$ is an invertible linear~transformation and $t\in\RR^d$, then $$\drop(SP+t)=\drop(P).$$
    \item \label{it:projection}
    if $\pi$ is the projection onto a $(d-k)$-dimensional subspace, then $$\drop(\pi P)\ge \drop(P)-k.$$
    \item \label{it:Minkowski} 
    $\drop(P_1+\cdots+ P_n) \ge (d-1) - \sum_i (d_i-1) + \sum_i \drop(P_i)$, where $+$ denotes the Minkowski sum of polytopes.
    \item \label{it:cs}
    if $P$ is centrally symmetric then %, then $\deg(\adj_P)$ is even, and hence
    $$\drop(P) \text{ is }\begin{cases}
        \text{even} & \text{if $d$ is odd,} \\
        \text{odd} & \text{if $d$ is even.}
    \end{cases}$$
    %\item If $P$ is a zonotope, then $\drop(P)=d-1$.
\end{myenumerate}
\end{theorem}
\begin{proof}
    % \ref{it:scissor} is easiest to see using the homogenized canonical form $\Omega(P; x_0,x)$.
    % We have $\Omega(P_i;x_0,x)=x_0^{\drop(P_i)} r(x_0,x)$ where the denominator of $r$ has no factor $x_0$. 
    % Hence the sum of the $\Omega(P_i;x_0,x)$ has $x_0^{\min_i\drop(P_i)}$ as a factor as well. \TODO

    For \ref{it:product} set $P:=P_1\times\cdots\times P_n$.
    It is a standard exercise to show that~$\deg(\adj_P)$ $=\sum_i\deg(\adj_{P_i})$ (we included a proof in \cref{sec:appendix_adjoint_product}).
    We likewise have $d=\sum_i d_{P_i}$ and $m=\sum_i m_{P_i}$ (where $m$ and $m_{P_i}$ denote the number of facets of $P$ and $P_i$~res\-pectively). The claim follows via
    \begin{align*}
        \drop(P) &= m - d - 1 - \deg \adj_P
            \\[1ex] &= \sum_i m_{P_i} - \sum_i d_{P_i} - 1 - \sum_i \deg\adj_{P_i}
            \\ &= \sum_i (\underbrace{m_{P_i} -  d_{P_i} - 1 - \deg\adj_{P_i}}_{=\,\drop(P_i)}) + n - 1.
    \end{align*}
    % \begin{equation}
    %     \drop(P\times Q) 
    %         &= m_{P\times Q} - d_{P\times Q} - 1 - \deg \adj_{P\times Q} 
    %         \\&= (m_P+m_Q) - (d_P+d_Q) - 1 - (\deg\adj_P + \deg\adj_Q)    
    % \end{equation}
    
    % To prove \ref{it:face} recall that $\Omega(F;x)$ is obtained as the residue of $\Omega(P;x)$ on $F$, that is, we remove the factor $L_F$ from the denominator and restrict everything to $\aff(F)$:
    % $$
    % \Omega(F;x) = \frac{\adj_P(x)|_F}{\prod_{G\not= F} L_G|_F}.
    % $$
    % %
    For \ref{it:face} we use the residue formula \eqref{eq:Omega_on_face} to obtain
    \begin{align*}
    %-\drop(P)-d
    %&=
    \deg \Omega^{d-1}(F) 
        &= \deg \left(\frac{\adj_P(x)|_F}{\prod_{G\not= F} L_G|_F}\right)
        = \deg(\adj_P(x)|_F) - \deg\Big(\prod_{G\not=F} L_G|_F\Big)
    \end{align*}
    The product on the right has $m-1$ factors.
    Each $L_G$ has degree one, but becomes constant on restriction to $F$ if and only if $G$ and $F$ are parallel.
    Only one facet can be parallel to $F$, so this introduces a case distinction:
    \begin{align*}
    \deg \Omega^{d-1}(F) 
        &= \deg(\adj_P(x)|_F) - \begin{cases}
            m-1 & \text{if $P$ has no facet parallel to $F$}
            \\[0.5ex]
            m-2 & \text{if $P$ has a facet parallel to $F$}
        \end{cases}
    \\  &\le \deg(\adj_P(x)) - m + 2
    \\  &= (m-d-1-\drop(P))- m + 2
        =- \drop(P)-d+1
    \end{align*}
    Comparing with $\deg \Omega^{d-1}(F) =-\drop(F)-(d-1)-1$ (from \eqref{eq:drop_Omega}) results in the claimed inequality.
    For equality the case distinction must evaluate to $m-2$, that is, there must be a parallel facet.
    %Observe that for equality the one inequality sign in the second calculation needs to become an equality, which in turn requires a parallel facet.
    The general statement with $f$ a face of dimension $d-k$ follows by induction.
    %
    % \begin{align*}
    % -\drop(F)-d
    % &=m_F-(d-1)-1-\drop(F)-m_F
    % \\&=\deg \frac{\adj_F(x)}{\prod_{G\subset F} L_G}
    % \\&=\deg \Omega(F;x) 
    % \\&= \deg \frac{\adj_P(x)|_F}{\prod_{G\not= F} L_G|_F}
    % \\[-1ex]&= \deg(\adj_P(x)|_F)) - \begin{cases}
    %     m-1 & \text{if $F$ has no parallel facet}
    %     \\[0.5ex]
    %     m-2 & \text{if $F$ has a parallel facet}
    % \end{cases}
    % \\&\le \deg(\adj_P(x)) - m + 2
    % \\&= \deg(\Omega(P;x)) + 2
    % \\&= m-d-1-\drop(P)-m+2=-\drop(P)-d+1
    % \end{align*}

    For \ref{it:max_drop} let $e\subset P$ be an edge.
    Then \ref{it:face} gives $0=\drop(e)\ge \drop(P)-(d-1)$, which proves the claim.

    Item \ref{it:invariant} is an application of \cref{res:Omega_trafo}:
    \begin{align*}
        \drop(SP + t) 
            &= -d-1-\deg(\Omega(SP+t;x))
        \\  &\overset{\smash{\mathclap{(*)}}}= -d-1-\deg(\Omega(SP+t;Sx+t))
        %\\  &= -d-1-\deg(\det(S)^{-1}\Omega(P;S^{-1}(x-t)))
        \\  &\overset{\mathclap{\smash{\ref{res:Omega_trafo}}}}= -d-1-\deg(\Omega(P;x))
        \\  &= \drop(P),
    \end{align*}
    where in $(*)$ we used that substituting an invertible affine transformation into a rational function does not change the degree.

    For \ref{it:projection} suppose first that $\pi$ projects onto a $(d-1)$-dimensional subspace orthogonal to a vector $c\in\RR^d$.
    Let $F_1,...,F_\mu$ be the facets of $P$ whose outwards normals satisfies $\<u_{F_i},c\>>0$.    
    Observe that $\pi F_i$ is an \emph{invertible} affine transformation of $F_i$, hence $(*)\,\drop(\pi F_i)=\drop(F_i)\ge \drop(P)-1$ (where we used both \ref{it:trafo} and \ref{it:face}).
    Observe further that the $\pi F_1,...,\pi F_\mu$ form a tiling of $\pi P$.
    Thus
    $$
        \drop(\pi P) 
            = \drop(\pi F_1\cupdot\cdots\cupdot \pi F_\mu) 
            \overset{\mathclap{\smash{\ref{res:drop_tiling}}}}
                \ge \min_i \drop(\pi F_i) 
            \overset{\mathclap{\smash{(*)}}}
                \ge \drop(P)-1.
    $$
    The statement for general projections follows via induction on $k$.

    For \ref{it:Minkowski} recall that the Minkowski sum $P_1+\cdots+ P_n$ is a projection of $P_1\times \cdots \times P_n$. 
    Hence we obtain the statement from a combination of \ref{it:product} and \ref{it:projection}:
    \begin{align*}
        \drop(P_1+\cdots+ P_n) 
            &= \drop(\pi(P_1\times\cdots\times P_n))    
        \\  &\overset{\mathclap{\smash{\text{\ref{it:projection}}}}}%
                \ge \drop(P_1\times\cdots\times P_n) - \Big(\sum_i d_i - d\Big)
        \\  &\overset{\mathclap{\smash{\text{\ref{it:product}}}}}%
                = \sum_i \drop(P_i) + n - 1 - \Big(\sum_i d_i - d\Big).
    \end{align*}
    The last expression rearranges to the claimed statement.

    For \ref{it:cs} we may assume that $P$ is origin symmetric (os), that is, $P=-P$~(other\-wise translate, which does not change the drop by \ref{it:trafo}).
    Using \cref{res:Omega_trafo} with~$S=-\Id$ yields
    $$
    \Omega(P;-x) 
        \overset{\mathclap{\smash{\text{os}}}}
            =\Omega(-P;-x) 
        \overset{\mathclap{\smash{\ref{res:Omega_trafo}}}}
            =\Omega(P;x).
    $$
    Hence $\Omega(P;x)$ is an even function in $x$.
    That means, its numerator degree (which is $\deg\adj_P$) and its denominator degree (which is $m$, the number of facets) have the same parity.
    %Recall the, that the numerator degree is $\deg\adj_P$, and the denominator degree is $m$, the number of facets.
    %Since $P$ is centrally symmetric, $m$ is even.
    Therefore, $\drop(P)=m-d-1-\deg \adj_P(x)\equiv -d-1\pmod 2$, and the parity of $\drop(P)$ is as claimed.
    %Also, for centrally symmetric polytopes $m$ is even, and so must be $\deg\adj_P$.
\end{proof}

In general, characterizing polytopes of a particular drop seems hard.
Some statements can be made about the extreme ends of the spectrum.

\begin{lemma}
    \label{res:pyramid}
    If $P$ is a pyramid, then $\drop(P)=0$.
\end{lemma}
\begin{proof}
    This is true in dimension $d=1$.
    In dimension $d\ge 2$, $P$ has a facet $F$~that is itself a pyramid and therefore has $\drop(F)=0$ by induction hypothesis. 
    Hence $\drop(P)\le \drop(F)+1=1$ by \cref{res:drop_properties} \ref{it:face}.
    But since $P$ has no facet parallel to $F$, this inequality is strict.
\end{proof}

% \begin{lemma}
%     If $P$ is a simplicial polytope, then $\drop(P)\le 1$. If $P$ is in addition centrally symmetric (\eg\ like a crosspolytope), then
%     %
%     $$\drop(C_d) = \begin{cases}
%         0 & \text{if $d$ is odd}
%         \\
%         1 & \text{if $d$ is even}
%     \end{cases}$$
% \end{lemma}

\begin{lemma}
    \label{res:crosspolytope}
    If $C_d$ is the $d$-dimensional crosspolytope, then
    $$\drop(C_d) = \begin{cases}
        0 & \text{if $d$ is odd,}
        \\
        1 & \text{if $d$ is even.}
    \end{cases}$$
\end{lemma}
\begin{proof}
    If $F$ is a facet of $C_d$, then $F$ is a simplex and has drop zero.
    By \cref{res:drop_properties} \ref{it:face} it follows $\drop(C_d)\le \drop(F)+1=1$.
    Hence $\drop(C_d)\in\{0,1\}$, and since $C_d$ is centrally symmetric the actual drop follows from \cref{res:drop_properties} \ref{it:cs}.
\end{proof}

Of course, much stronger statements can be made if required.
For example,\nls the proof of \cref{res:crosspolytope} applies verbatim to any centrally symmetric simplicial polytope, or more generally, to any centrally symmetric polytope with a single facet of drop zero.

%At least~one extreme case does however appear approachable.
On the other end of the spectrum, in \cref{ex:cube_drop} we already mentioned that zonotopes have a drop of at least $d-1$.
Using \cref{res:drop_properties} \ref{it:max_drop} we now see that~$d-1$ is the maximal possible drop and zonotopes attain it.
%By \cref{res:drop_tiling} the same is true for all polytopes that can be~tiled by parallelepipeds.
%This includes in particular all \textit{zonotopes}.
\cref{res:drop_properties} actually gives us two more ways to see this:

\begin{corollary}
    \label{res:zonotopes_max_drop}
    If $P$ is a zonotope, then $\drop(P)=d-1$.
\end{corollary}
\begin{proof}
    Since zonotopes are projections of cubes, the lower bound $d-1$ follows~from \cref{res:drop_properties} \ref{it:projection}.
    Alternatively, since zonotopes are Minkowski sums of line segments, the lower bound also follows from \cref{res:drop_properties} \ref{it:Minkowski} by letting all $P_i$ be line segments:
        $$\drop(P) \ge (d-1)-\sum_i \underbrace{(d_i-1)}_{=0} + \sum_i \underbrace{\drop(P_i)}_{=0} = d-1.$$
    In each case we obtain a lower bound that matches the maximal possible drop from \cref{res:drop_properties} \ref{it:max_drop}.
    We therefore have equality.
\end{proof}

Whether zonotopes are the only polytopes with maximal $\drop(P)=d-1$ we~cannot answer with the tools available at this point.
The answer is provided in \cref{res:zonotopes_iff_drop_d_1}.
For later use we collect further properties of maximal drop polytopes:

\begin{lemma}
    \label{res:max_drop_properties}
    Let $P$ be a polytope of maximal drop $d-1$.
    \begin{myenumerate}
        \item if $f$ is a face of $P$, then it also has maximal drop in its respective dimension.
        \item if $F$ is a facet, then $P$ has a facet $G$ parallel to $F$.
    \end{myenumerate}
\end{lemma}
\begin{proof}
    If $f$ is a face of dimension $\delta$, then by \cref{res:drop_properties} \ref{it:face} $\drop(f)\ge \drop(P)-(d-\delta)=\delta-1$, which is maximal for a polytope of this dimension. 
    If $F$ is a facet, since we attain the bound in \cref{res:drop_properties} \ref{it:face}, $P$ must have a parallel facet.
\end{proof}

%\msays{Maybe: Example, pyramids have no drop.}

% \newpage

% Most polytopes (that is, generic polytopes) has drop $s=0$ (we shall say \Def{no drop}).
% The easiest way to construct polytopes of non-zero drop is via Cartesian products:

% \begin{lemma}
%     If $P=P_1\times\cdots \times P_n$, where $P_i$ is of drop $s_i$, then $P$ has drop
%     %
%     $$d=n-1+\sum_i s_i.$$
% \end{lemma}

% \begin{lemma}
%     If $P$ has drop $s$, then each facet $F$ has drop $s_F\ge s-1$
%     If $s=s_F-1$ then $P$ has a facet parallel to $F$.
% \end{lemma}
% %
% \begin{proof}
%     \TODO
% \end{proof}

% \begin{corollary}
%     The drop of a $d$-polytope is $s\le d-1$.
% \end{corollary}

% We say that a polytope is of \Def{maximal degree drop} if it is of drop $s=d-1$.
% A central question is to characterize the polytopes of maximal degree drop, which we do in \cref{res:max_drop_zonotope}.

% \begin{lemma}
%     If a $d$-polytope $P$ is a linear projection of a $(d+1)$-polytope $Q$, then $s_P\ge s_Q-1$. 
% \end{lemma}
% %
% \begin{proof}
%     \TODO
% \end{proof}

% We can obtain lower bounds on the drop of Minkowski sums.

% \begin{lemma}
%     If $P=P_1+\cdots+ P_n$ is a $d$-polytope and $P_i$ is of dimension $d_i$ and has degree drop $s_i$, then $P$ has degree drop
%     %
%     $$s\ge (d-1)-\sum_i(d_i-1) +\sum_i s_i.$$
% \end{lemma}

% In \cref{ex:...} we shall present a polytope which can be decomposed as a Minkowski sum yet not in a way to make the above lower bound tight.

%\newpage

\section{The reduced canonical form $\Omega_0$ and translation-invariance}
\label{sec:Omega0}

A major disadvantage of $\Omega$ that makes it unsuited for questions of scissors~congruence is that it is not translation-invariant.
In this section we introduce a translation-invariant modification.
In a first step we define the \textit{homogenized canonical form},\nls and subsequently derive from it the \textit{reduced canonical form} $\Omega_0$.

% \todo{\msays{Introduced $\Omega$ for cones somewhere since it will be needed in \cref{sec:homogeneous}.}}%

% \msays{Work in progress: define $\Omega(P;x_0,x)$ as $\Omega$ applied to the homogenized polytope $P$. Prove the other definition as a lemma.}

% \hrulefill

\iffalse % homogenization via cone(P)

\subsection{The homogenized canonical form}

We start from the \Def{homogenized canonical form %, which we define as 
$\Omega(P;x_0,x)$} $:=\Omega(P^{\hom};(x_0,x))$, where
$$P^{\mathrm{hom}}:=\cone\{(1,x)\mid x\in P\}\subset \RR^{d+1},$$ 
denotes the homogenization of $P$.
Then
\begin{align*}
    \Omega(P;x_0,x) &= \Omega(P^{\hom};(x_0,x)) 
    \\&= x_0^{-d-1}\Omega(P^{\hom};(1,x/x_0))
    = x_0^{-d-1}\Omega(P;x/x_0),
\end{align*}
where we used that $\Omega$ on $(d+1)$-dimensional cones is homogeneous in $x$ of degree $-d-1$, and that de-homogenization works as expected.
Both facts are easy to verify for simplices (\cf\ \cref{ex:simplices}), and subsequently generalize to all polytopes using triangulation and the valuation property \eqref{eq:valuation}.
Substituting \eqref{eq:Omega_represenation} we arrive~at
\begin{align*}
\Omega(P;x_0,x) 
    = x_0^{-d-1}\frac{\adj_P(x/x_0)}{\prod_{F\subset P} L_F(x/x_0)}  
    %\\&
    &= \frac{x_0^{m-d-1} \adj(x/x_0)}{x_0^m \prod_F(h_F-\<x/x_0,u_F\>)}.
    %\\&= \frac{\adj_P(x_0,x)}{\prod_F (h_Fx_0-\<u_F,x\>)} 
\end{align*}
We see that the homogenized canonical form can alternatively be obtained by homo\-genizing both numerator and denominator in \eqref{eq:Omega_represenation} to their ``expected degrees'',\nls that is, the numerator (which is the adjoint) to degree $m-d-1$, and the denominator to degree $m$.
If $\adj_P(x_0,x)$ is this \Def{homogenized adjoint}, then we can write
\begin{align*}
\Omega(P;x_0,x) 
    = \frac{\adj_P(x_0,x)}{\prod_F (h_Fx_0-\<u_F,x\>)}.
\end{align*}

\begin{lemma}
    $\Omega(P^{\hom};(x_0,x)) = x^{-d-1}_0\Omega(P;x/x_0).$
\end{lemma}
\begin{proof}
    We triangulate $P$. By the valuation property it suffices to check the identity on a simplex. \TODO
\end{proof}

\iffalse % previous attempt at defining Omega_hom in two ways - 17/07/2025

{\color{lightgray}

\hrulefill

As a first step we introduce the \Def{homogenized canonical form} $\Omega(P;x_0,x)$.
There are two equivalent ways to do so:
\begin{myenumerate}
    \setlength{\itemsep}{1ex}
    
    \item as the canonical form applied to the homogenized polytope
    $$P^{\mathrm{hom}}:=\cone\{(1,x)\mid x\in P\}\subset\RR\times \RR^d.$$ 
    Then $\Omega(P;x_0,x):=\Omega(P^{\mathrm{hom}};(x_0,x))$.
    
    \item by homogenizing the numerator and denominator to their expected degrees. More precisely, we homogenize the numerator (the adjoint) to degree $m-d-1$, and the denominator to degree $m$. If we write $\adj_P(x_0,x):=x_0^{m-d-1} \adj_P(x/x_0)$ for the \Def{homogenized adjoint}, then we obtain
    \begin{align*}
    \quad\qquad\Omega(P;x_0,x) 
        &= \frac{\adj_P(x_0,x)}{\prod_F (h_Fx_0-\<u_F,x\>)} 
        \\&= \frac{x_0^{m-d-1} \adj(x/x_0)}{x_0^m \prod_F(h_F-\<u_F,x/x_0\>)} = x_0^{-d-1}\, \Omega(P;x/x_0).    
    \end{align*}
    
\end{myenumerate}

We quickly verify that these definitions are equivalent:

\begin{lemma}
    $\Omega(P^{\mathrm{hom}};(x_0,x))= x_0^{-d-1}\Omega(P;x/x_0)$.
\end{lemma}
\begin{proof}
    Since both sides valuative, it suffices to triangulate $P$ and verify the claim on simplices.
    If a simplex $\Delta$ has normal vectors $u_0,...,u_d$ and heights $h_0,...,h_d$, then its homogenization has normal vectors $v_i:=(h_i,u_i)$.
    Given \cref{ex:simplices}, clearly $\Omega(\Delta^{\mathrm{hom}})$ is homogeneous of degree $-d-1$ in $(x_0,x)$:
    $\<(1,x_i),(-h_i,u_i)\> = -h_i + \<x,u_i\>$.

    $$\Omega(\Delta^{\hom},(x_0,x))=x_0^{-d-1}\Omega(\Delta^{\hom},(1,x/x0))=x_0^{-d-1}\Omega(\Delta;x/x_0).$$
\end{proof}

\hrulefill

}

\fi

\else % homogenizing numerator and denominator separately

%\subsection{The homogenized canonical form}

\subsection{The homogenized canonical form}
\label{sec:homogenized_Omega}

There are two equivalent ways to define the \Def{homogenized canonical form $\Omega(P;x_0,x)$} for a polytope -- geometrically or more algebraically.
We shall mainly work with the latter.

We construct $\Omega(P;x_0,x)$ by homogenizing both the numerator and denominator of $\Omega$ to their respective ``expected degrees''.
More precisely, we homogenize the~numerator (which is the adjoint) to degree $m-d-1$, and the denominator to degree $m$ (recall that $m$ is the number of facets).
%More precisely, we homogenize the numerator (the adjoint) to degree $m-d-1$, and the denominator to degree $m$.
Formally, this yields
\begin{align*}
\Omega(P;x_0,x) 
&:= \frac{x_0^{m-d-1} \adj(x/x_0)}{x_0^m \prod_F L_F(x/x_0)} = x_0^{-d-1}\, \Omega(P;x/x_0).    
\end{align*}
%
% \hrulefill
%
% %%%%%%%%%%%%
%
% As a first step we need to define the \Def{homogenized canonical form $\Omega(P;x_0,x)$}.
% We construct it by homogenizing both~the numerator and denominator of $\Omega$ to their respective ``expected degrees'', that is, we homogenize the numerator (which is the adjoint) to degree $m-d-1$, and~the~de\-nominator to degree $m$.
% %More precisely, we homogenize the numerator (the adjoint) to degree $m-d-1$, and the denominator to degree $m$.
% Formally, this yields
% %
% \begin{align*}
% \Omega(P;x_0,x) 
% &= \frac{x_0^{m-d-1} \adj(x/x_0)}{x_0^m \prod_F L_F(x/x_0)} = x_0^{-d-1}\, \Omega(P;x/x_0).    
% \end{align*}
%
If $\adj_P(x_0,x):=x_0^{m-d-1} \adj_P(x/x_0)$ denotes the \Def{homogenized adjoint}, then
\begin{equation}
\label{eq:Omegahom}
\Omega(P;x_0,x) 
= \frac{\adj_P(x_0,x)}{\prod_F (h_Fx_0-\<u_F,x\>)}   
\end{equation}
Observe that $\drop(P)$ is the highest power of $x_0$ that divides $\adj_P(x_0,x)$.
%
% \begin{align*}
% \Omega(P;x_0,x) 
% &= \frac{\adj_P(x_0,x)}{\prod_F (h_Fx_0-\<u_F,x\>)} 
% \\&= \frac{x_0^{m-d-1} \adj(x/x_0)}{x_0^m \prod_F(h_F-\<u_F,x/x_0\>)} = x_0^{-d-1}\, \Omega(P;x/x_0).    
% \end{align*}
%
%\subsection{A translation invariant valuation}

%\begin{remark}
    Geometrically, this construction agrees with applying the canonical form to the homogenization of the polytope $P$: %his homogenized canonical form can also be understood as the usual canonical form applied to the homogenized polytope:
    \begin{equation}
        \label{eq:Omegahom_geom}\Omega(P;x_0,x)=\Omega(P^{\hom};(x_0,x)),\quad\text{where } P^{\mathrm{hom}}:=\cone(\{1\}\times P)\subset \RR^{d+1}.
    \end{equation}
    %
    % Then $\Omega(P;x_0,x)=\Omega(P^{\hom};(x_0,x))$.
    % The equivalence to our initial definition is~obtained via
    % %
    % $$\Omega(P^{\hom};(x_0,x))=x^{-d-1}\Omega(P^{\hom};(1,x/x_0))=x^{-d-1}\Omega(P;x/x_0).$$
%\end{remark}
The equivalence can be shown by triangulating $P$ and proving the statement on~simplices (\cf\ \cref{ex:simplices}). For details, see \cref{res:homogenized_Omega} in the appendix.

\fi

\subsection{The reduced canonical form}
\label{sec:reduced_Omega}

The \Def{reduced canonical form $\Omega_0$} is obtained from $\Omega(P;x_0,x)$ by setting $x_0=0$:
\begin{equation}
    \label{eq:Omega0_def}
    \Omega_0(P;x) := \Omega(P;x_0,x)|_{x_0=0}=(-1)^m\frac{\adj_P(x_0,x)|_{x_0=0}}{\prod_F \<x,u_F\>}.
\end{equation}
%
%\begin{remark}
%\todo{\asays{Nothing major, but I believe we're lacking a $(-1)^{d+1}$ in the above equation.}\msays{Don't we actually need a $(-1)^m$? My impression was that the signs come from the denominator. Did you have something else in mind?}}%
%
In other words,
$\Omega_0(P)$ is obtained from $\Omega(P)$ by restricting both numerator and~denominator to the monomials of the expected maximal degree.
In the numerator~this means restricting to monomials of degree $m-d-1$, and in the denominator to~monomials of degree $m$. 

The geometric interpretation of this construction is that $\Omega_0(P)$ is the restriction of the homogeneous rational function $\Omega(P^{\hom};(x_0,x))$ to the ``hyperplane at~infinity'' $H^\infty=\{x_0=0\}$.
%\end{remark}
%There are different ways to describe its construction.
%For example, by first homogenization to degree $m-d-1$ in the numerator and degree $m$ in the denominator, and then setting $x_0=0$; 
%
This restriction clearly preserves the property of being~a~simple valuation, and so $\Omega_0$ too is a simple valuation on convex polytopes.\nls
Its~cru\-cial new property is translation-invariance:

\begin{lemma}
    \label{res:translation_invariant}
    $\Omega_0$ is translation-invariant, that is % valuation on convex polytopes, that is,
    %\todo{\asays{This follows from Theorem \ref{thm:transl_polynomial} now.}}
    %defined
    $$\Omega_0(P+t;x)=\Omega_0(P;x),\quad\text{for all $t\in\RR^d$}.$$
\end{lemma}
\begin{proof}
    % First, we note that the map \[
    % \RR(x)\to\RR(x),~  f\mapsto \big(x_0^{-d-1}f(x/x0)\big)|_{x_0=0}
    % \]
    % is linear on its domain. Hence $\Omega_0$ inherits the valuation property from $\Omega$.

    By \cref{res:Omega_trafo} we have $\Omega(P+t,x)=\Omega(P;x-t)$.
    Observe that substituting $x\mapsto x-t$ into a polynomial does not change its monomials of maximal degree.
    Since $\Omega_0$ only retains the expected top degree terms of the numerator and denominator~of $\Omega$, it is therefore unaffected by this modification.
    %\todo{\msays{Combine both proofs.}}
    %
    \iffalse % Ansgar's old proof - 18-06-2025
    \asays{I think this argument is quite brief. Aren't we looking at
    \[
        \Omega_0(P+t;x) = \Omega(P+t;x_0,x)|_{x_0=0} = [x_0^{-d-1}\Omega(P+t;x/x_0)]_{x_0=0} = [x_0^{-d-1}\Omega(P; x/x_0 - t)]_{x_0=0}?
    \]
    So we're not really substituting $x\mapsto x-t$ but $x/x_0 \mapsto x/x_0 -t$, \ie\ $x\mapsto x-x_0t$. I try to spell this out below, but it's of course the same idea:
    }
    \msays{I can accept if you find it too brief, but I think it is mathematically sound: I use here the ``definition'' of restricting to the ``expected leading monomials'', not even passing through the homogenization. Let's talk about this later, for now I keep both versions.}

     By \cref{res:Omega_trafo} we have $\Omega(P+t,x)=\Omega(P;x-t)$ and therefore 
     \todo{\msays{On a second though: what is written here reads like a more detailed explanation of what I mean by ``restriction to expected leading monomial''. So maybe this should become an explanation before the proof, and then my proof suffices.}}
     %
     \[
        \Omega(P+t;x_0,x) = x_0^{-d-1}\Omega(P; x/x_0 - t) = \frac{x_0^{m-d-1} \adj_P( x/x_0-t)}{\prod x_0L_F(x/x_0-t)}
     \]
      Setting $x_0=0$ in the numerator and denominator yields the degree-$(m-d-1)$ part of $\adj_P(x)$ and the degree-1-part of each $L_F$, \ie\ $\langle u_F,x\rangle$. In particular, after setting $x_0=0$ the terms involving $t$ vanish.
      \fi%
    %
    % Translation invariance follows essentially from the fact that the coefficients of the leading monomials of a polynomial are invariant under translation.
    % In details, we have $\Omega(P+t;x)=\Omega(P;x-t)$ by \cref{res:affine_covariant}.
    % For the numerator of $\Omega(P;x-t)$ holds
    % %
    % $$\adj_P(x)=\sum_k a_k x^k\;\implies\; \adj_P(x-t) = \sum_k a_k(x-t)^k = a_k x^k + o(x^k).$$
    % %
    % The same applies to the denominator.
    % If we then restrict to the leading monomials (by homogenizing and setting $x_0=0$), the remaining terms are identical before and after translation.
\end{proof}

{\color{black}
    
\begin{remark}
In contrast to $\Omega$, the reduced canonical form cannot be extended,\nls as a valuation, to general \emph{unbounded} polyhedra. 
%In fact, it is known that the only~trans\-lation-invariant valuation that is defined on all polyhedra is the \mbox{Euler characteristic} \todo{Reference?}\cite{...}.
The algebraic construction of the~homogenized canonical form $\Omega(P;x_0,x)$ fails because for unbounded polyhedra the adjoint degree can be higher than $m-d-1$ and so cannot be homogenizing to degree $m-d-1$ (\eg\ simplicial cones have adjoint degree $m-d$, see \cref{ex:simplices}).
The geometric construction of $\Omega(P;x_0,x)$ does work, but yields an expression with a non-zero power of $x_0$ in the denominator.
This prevents us from setting $x_0=0$, and so $\Omega_0$ cannot be constructed.
\end{remark}

}

%\todo{\msays{What exactly fails again?}\asays{I think the geometric construction of $\Omega(P;x_0,x)$ does NOT work if $P$ is unbounded: In that case $P^{\hom}$ will not be closed, so the canonical form is not defined.}\msays{$\Omega$ on cones such as $P^{\hom}$ can be defined via an integral over $\exp(\<x,y\>)$, so it should be well-defined on an open set and equals $\Omega$ on the closure, no? My impression was that for cones after canceling everything we are still left with an $x_0$ factor in the denominator and therefore cannot set $x_0=0$. But I am not sure.}}

\subsection{Applying $\boldsymbol{\Omega_0}$}
\label{sec:applying_Omega0}

If the adjoint $\adj_P$ is not of the expected degree $m-d-1$, \nls then restricting to the monomial of this degree yields zero. Hence

\begin{observation}
    \label{res:Omgea0_0_iff_drop_0}
    $\Omega_0(P)=0$ if and only if $\drop(P)>0$.
\end{observation}

%For now we emphasize that \cref{res:Omgea0_0_iff_drop_0} 
This gives us a first translation scissors invariant based on adjoint degrees:

\begin{corollary}
    \label{res:drop_0_invariant}
    $\drop(P)=0$ is a translation scissors invariant.
\end{corollary}

Once again, to make use of this fact directly we would need to understand~what a drop of zero encodes geometrically or combinatorially (see \cref{q:drop}).
The~intuition is that drop zero is the ``generic case'', though a precise alternative formulation is not available in all cases (for $d\le3$ an answer was obtained, see \cref{res:characterize_drop_2D}~and \cref{res:characterize_drop_3D}).

We can now answer the question raised by \cref{res:zonotopes_max_drop}.

\begin{theorem}
    \label{res:zonotopes_iff_drop_d_1}
    $P$ has maximal $\drop(P)=d-1$ if and only if $P$ is a zonotope.
    % Given a $d$-polytope $P$, the following are equivalent:
    % %
    % \begin{myenumerate}
    %     \item
    %     $P$ has maximal adjoint degree drop, that is, $\drop(P)=d-1$,
    %     \item 
    %     $P$ is a zonotope.
    % \end{myenumerate}
\end{theorem}
\begin{proof}
    That zonotopes have maximal drop $d-1$ was proven in \cref{res:zonotopes_max_drop}.
    
    For the converse, assume that $P$ is a polytope of maximal $\drop(P)=d-1$.
    The statement is clear for $d=1$.

    \begin{figure}[h!]
        \centering
        \includegraphics[width=0.8\linewidth]{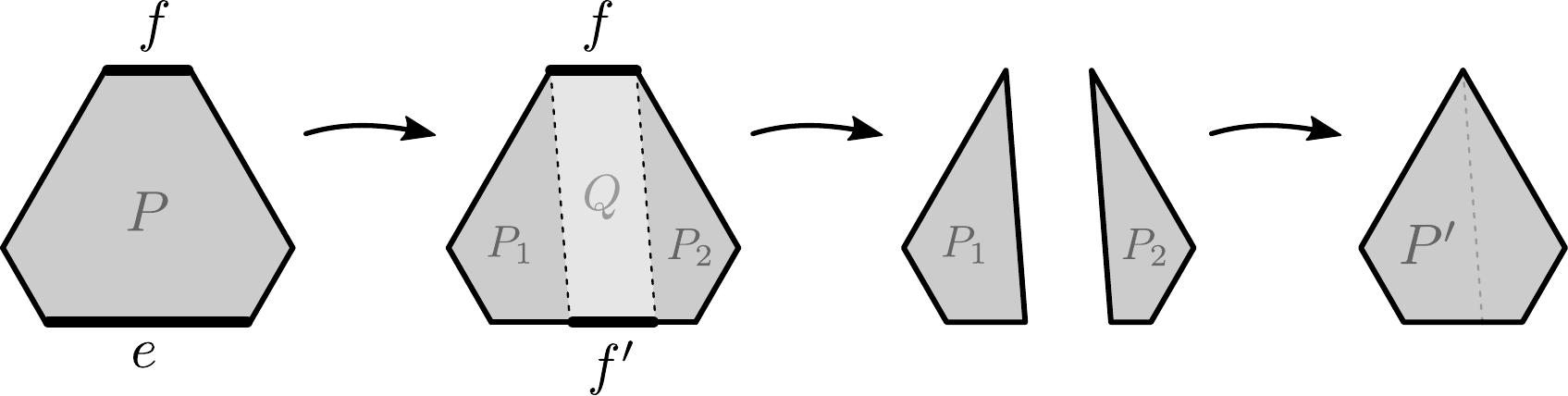}
        \caption{Visualization of the proof that a polygon with non-zero drop must be centrally symmetric.}
        \label{fig:2D_cs_proof}
    \end{figure}
    
    For $d=2$ we proceed as follows: 
    by \cref{res:max_drop_properties} \itm2 each edge of $P$ has a~parallel edge, in particular, the number of edges is even.
    Fix an arbitrary edge~$e$~and~let $f$ be its parallel edge.
    Without loss of generality we may assume that $e$ is not shorter than $f$ and we can therefore choose a sub-segment $f'\subseteq e$ that is a translate of~$f$.
    Then $f$ and $f'$ together form a parallelogram $Q$ that subdivides $P$ into three convex polygons $P=P_1\cupdot Q\cupdot P_2$ as shown in \cref{fig:2D_cs_proof}.
    Recall that both $P$ and $Q$ have non-zero drop (\cf\ \cref{ex:cube}).
    Since $\Omega_0$ is a simple valuation we obtain
    $$\Omega_0(P_1)+\Omega_0(P_2)=\Omega_0(P)-\Omega_0(Q)=0.$$
    Moreover, we can move $P_1$ and $P_2$ together to close the gap in between, and since $\Omega_0$ is translation-invariant, the resulting convex polygon $P'$ still satisfies $\Omega_0(P')=0$.
    As argued above, $P'$ must therefore have an even number of edges.
    But since~deleting $Q$ removed edge $f$ entirely (see \cref{fig:2D_cs_proof}), we can only have an even edge count if also edge $e$ vanished entirely.
    This in turn requires that $e$ and $f$ were of the~same length.
    Since this edge pair was chosen arbitrarily, we conclude that the edges~of~$P$ come in parallel pairs of matching lengths, and hence, that $P$ is centrally symmetric.
    Equivalently, $P$ is a zonotope.

    For $d>2$ we can argue as follows: by \cref{res:max_drop_properties} \itm1 the 2-dimensional faces~of $P$ are themselves of maximal drop in their respective dimension.
    By case $d=2$~they are therefore centrally symmetric.
    But a polytope of dimension $d\ge 3$ all~whose~2-faces are centrally symmetric is a zonotope (see \eg\ \cite[Section 7.3]{ziegler2012lectures}).
\end{proof}

%\msays{Zonotope adjoint is sum of squares?}

As an aside we found a geometric interpretation for the drop in dimension two.

\begin{corollary}
    \label{res:characterize_drop_2D}
    If $P$ is a polygon, then
    $$\drop(P) = \begin{cases}
        0 & \text{if $P$ is not centrally symmetric,}
        \\
        1 & \text{if $P$ is centrally symmetric.}
    \end{cases}$$
\end{corollary}

\begin{observation}
    \label{res:drop_yields_cs_result}
    With \cref{res:characterize_drop_2D} we rediscovered \cref{ex:cs} for $d=2$:\nls that central symmetry is a translation scissors invariant follows here from the fact that central symmetry in dimension $d=2$ is equivalent to $\drop(P)=0$, which is translations scissors invariant by \cref{res:drop_0_invariant}.
\end{observation}

Another way to phrase \cref{res:drop_yields_cs_result} is that the value of $\drop(P)$~is~a~translation scissors invariant in dimension $d=2$.
Is this also true in higher dimensions?
%
%$\Omega_0$ has some advantages and disadvantages compared to $\Omega$. 
%The major advantage is that we can use it for matters of translation scissor congruence.
%Since $\Omega_0(P)=0$ precisely if $P$ has a non-zero drop, we can obtain some results on drops in relation  to translation scissor congruence.
Unfortunately, there is a major shortcoming of the reduced canonical form $\Omega_0$ as compared to $\Omega$:
%Some examples are shown below.
%The major disadvantage is that 
it does not ``see'' the exact drop of $P$, but merely whether the drop is zero or non-zero. 
This captures the full range of possibilities for $d=2$, but not for $d\ge 3$.
One attempt at capturing higher values of the drop using valuations~is discussed in \cref{sec:higher_valuations}.
Here we prove the surprising fact that the drop~is~also~invariant in dimension $d=3$:

% \begin{question}
%     How can one interpret $\Omega_0(P)=0$ for $d>2$?
% \end{question}

\begin{theorem}
    \label{res:d_le_3_preserves_drop}
    If $d\le 3$ then $\drop(P)$ is a translation scissors invariant.
\end{theorem}
\begin{proof}
    For $d=1$ this is trivial; for $d=2$ it follows from \cref{res:characterize_drop_2D}.

    For $d=3$ we have $\drop(P)\in\{0,1,2\}$ and the following facts: centrally symmetric polytopes have an even drop (see \cref{res:drop_properties} \ref{it:cs}), and zonotopes~are~centrally symmetric %and have a drop of exactly two. % (see \cref{res:zonotopes_iff_drop_d_1} for the latter). 
    %Taken together we obtain:
    and characterized by $\drop(P)=2$ (see \cref{res:zonotopes_iff_drop_d_1}). % (see \cref{res:zonotopes_iff_drop_d_1} for the latter). 
    Taken together~we obtain:
    \begin{align*}
        \drop(P) = \begin{cases}
            0 & \text{if $\Omega_0(P)\not=0$,}
            \\
            1 & \text{if $\Omega_0(P)=0$ and $P$ is not centrally symmetric,}
            \\
            2 & \text{if $\Omega_0(P)=0$ and $P$ is centrally symmetric.}
        \end{cases}
    \end{align*}
    % \begin{align*}
    %     \drop(P)=0 &\;\Longleftrightarrow \;\Omega_0\not=0,
    %     \\
    %     \drop(P)=1 &\;\Longleftrightarrow\; \Omega_0=0 \text{ and $P$ is not centrally symmetric},
    %     \\
    %     \drop(P)=2 &\;\Longleftrightarrow\; \Omega_0=0 \text{ and $P$ is centrally symmetric}.
    % \end{align*}
    %
    %Since the statements on the right side are mutually exclusive and cover all cases, the implications are actually equivalences.
    The statement then follows from the fact that both ``$\Omega_0(P)=0$'' and ``$P$ is centrally symmetric'' are translation scissors invariants (\cf\ \cref{res:drop_0_invariant} and \cref{ex:cs}).
    % For $d=1$ this is trivial; for $d=2$ this follows from the face that $\drop(P)\in\{0,1\}$ (\cref{res:drop_properties} \itm4) and that $\drop(P)>0$ is translation scissor invariant (\cref{res:...}).
    %
    % For $d=3$ we argue as follows: $\drop(P)=0$ is translation scissor congruent as argued before.
    % If $\drop(P)=2$, then it is a zonotope, in particular, centrally symmetric.
    % If $Q$ is translation scissors congruent to $P$, then we know two things: first, $\drop(Q)>0$ by \cref{res:...}, and $Q$ is centrally symmetric by \cref{res:...}.
    % But by \cref{res:drop_properties} \itm8, a centrally symmetric 3-polytope has an even drop. Since $\drop(Q)>0$ and $\drop(Q)\in\{0,1,2\}$ only $\drop(Q)=2$ remains.
    % The case $\drop(P)=1$ then follows necessarily.
\end{proof}

Together with \cref{res:zonotopes_iff_drop_d_1} this yields a purely geometry statement about translation scissors invariance that we believe is not in the literature:

\begin{corollary}
    In dimension $d\le 3$ being a zonotope is translation scissors invariant. 
\end{corollary}

\begin{remark}
    It is a classical result that a polytope that tiles Euclidean space by translates -- a so-called \emph{parallelotope} -- is translation scissors congruent to~a~cube~\cite{lagariasmoews}.
    The list of different combinatorial types of parallelotopes is know up to dimension $d=5$ \cite{dutour2016complete,garber2019voronoi}.
    %\notodo{\asays{In there, they only classify Voronoi cells of lattices. That these coincide with parallelotopes was only later shown by Alexey Garber et al., I think. I'll look that up at some point.}\msays{This one? \cite{garber2019voronoi}}\asays{Yup, but I'd cite Achill and friends as well here.}}
    Interestingly, for $d=3$ this list contains only zonotopes. \Cref{res:d_le_3_preserves_drop} provides a quick argument for this fact. %, more specifically, that $\Omega_0(P)=2$ is translation scissors invariant in $d=3$.
    Conversely, in dimension $d\ge 4$~there~are polytopes that are translation scissors congruent to the cube but that are \textit{not} zonotopes, \eg\ the 24-cell, which is a parallelotope for $d=4$.
    We conclude that for $d\ge 4$ the drop is \textit{not} a translation scissors invariant.
\end{remark}

% \msays{Some results that use translation invariance}

% \begin{theorem}
%     Given a $d$-polytope $P$, then the following are equivalent:
%     %
%     \begin{myenumerate}
%         \item $\drop(P)=d-2$ and has a centrally symmetric facet $F$.
%         \item $P$ is half of a (generalized) zonotope.
%     \end{myenumerate}
% \end{theorem}

We close this section with an analogue of \cref{res:Omega_trafo} for $\Omega_0$. %, which is also a direct consequence of it: % recorded for later use: %For later use we also record the following:

\begin{corollary}
    \label{res:Omega0_trafo}
    For $S\in\GL(\RR^d)$ holds
    %\todo{\asays{This is proven for general $\Omega_s$ in Lemma \ref{lemma:equivariance_omegas}.}}
    %
    % $$\Omega_0(SP,x)= \det(S)^{-1}\Omega_0(P;S^{-1}x).$$
    % $$\det(S) \Omega_0(SP,Sx)= \Omega_0(P;x).$$
    $$ \Omega_0(SP,Sx)= |\!\det(S)|^{-1}\,\Omega_0(P;x).$$
\end{corollary}
\begin{proof}
    By the definition of the homogenized canonical form and \cref{res:Omega_trafo}, we~have
    \begin{align*}
        \Omega(SP;x_0,Sx)
        &=x_0^{-d-1}\, \Omega(SP;Sx/x_0)
        = x_0^{-d-1}\, \Omega(SP;S(x/x_0))
        \\&\overset{\mathclap{\text{\ref{res:Omega_trafo}}}}=\, |\!\det(S)|^{-1} x_0^{-d-1} \Omega(P;x/x_0)
        = |\!\det(S)|^{-1}\, \Omega(P;x_0,x).
    \end{align*}
    Setting $x_0=0$ on both sides yields the claim.
\end{proof}

%
% \begin{proof}
%     This follows immediately from \cref{res:Omega_trafo}. % and \cref{res:translation_invariant}.
% \end{proof}

%\newpage

\section{Homogeneity}
\label{sec:homogeneous}

A translation-invariant valuation $\phi$ is \Def{$k$-homogeneous} if for all $\lambda>0$ holds
$$\phi(\lambda P)=\lambda^k \phi(P).$$
In this section we prove that $\Omega_0$ is homogeneous and explore the implications.

% {\color{lightgray}

% \begin{lemma}
%     $\Omega_0$ is 1-homogeneous for all $\lambda\in\RR$.
% \end{lemma}
% %
% \begin{proof}
%     We use \cref{res:Omega0_trafo} and that $\Omega_0$ is a homogeneous rational function of degree $-d-1-s$:
%     %
%     \begin{align*}
%         \Omega_0(\lambda P;x) 
%             \overset{\mathclap{\smash{\ref{res:Omega0_trafo}}}}= |\lambda^{-d}| \Omega_0(P;x/\lambda)
%             = |\lambda^{-d}| \lambda^{d+1+s} \Omega_0(P;x)
%             = \sign(\lambda)^d \lambda^{s+1} \Omega_0(P;x)
%             = \sign(\lambda)^d \lambda^{s+1} \Omega_0(P;x)
%     \end{align*}
% \end{proof}

% }

\iffalse % proof homogeneous and central parity at the same time

\begin{lemma}
    It holds
    \begin{myenumerate}
        \item $\Omega_0$ is 1-homogeneous, that is, $\Omega_0(\lambda P)=\lambda \Omega_0(P)$ for all $\lambda>0$.
        \item $\Omega_0$ has opposite parity to $d$, that is, $\Omega_0(-P)=(-1)^{d+1} \Omega_0(P)$.
    \end{myenumerate}
\end{lemma}
\begin{proof}
    We use \cref{res:Omega0_trafo} and that $\Omega_0$ is a homogeneous rational function of degree $-d-1-\drop(P)$.
    If we set $s:=\drop(P)$, then
    \begin{align*}
        \Omega_0(\lambda P;x) 
            &\overset{\mathclap{\smash{\ref{res:Omega0_trafo}}}}= \lambda^{-d} \Omega_0(P;x/\lambda)
            = \lambda^{-d} \lambda^{d+1+s} \Omega_0(P;x)
            = \lambda^{s+1} \Omega_0(P;x),
    \\&\hspace{-2em}\text{and}
    \\\Omega_0(-P;x)  &\overset{\mathclap{\smash{\ref{res:Omega0_trafo}}}}= \Omega_0(P;-x) = (-1)^{-d-1-s}\, \Omega_0(P;x).
    \end{align*}
    In both cases: if $s>0$ then $\Omega_0(P)=0$ and the result holds trivially; and if $s=0$~we obtain the claimed result as well.
\end{proof}

\else

\begin{lemma}
    \label{res:1_homogeneous}
    $\Omega_0$ is 1-homogeneous.
    %\todo{\asays{This is a special case of Theorem \ref{thm:omegas_homogogeneity}}}%
\end{lemma}
\begin{proof}
    The statement is clear if $\Omega_0(P)=0$.
    Otherwise, $\Omega_0(P)$ is a homogeneous~rational function of degree $-d-1-\drop(P)$.
    Using \cref{res:Omega0_trafo} and $\lambda>0$ we~obtain
    \[
        \Omega_0(\lambda P;x) 
            \,\overset{\mathclap{\smash{\ref{res:Omega0_trafo}}}}=\, \lambda^{-d} \Omega_0(P;x/\lambda)
            = \lambda^{-d} \lambda^{d+1} \Omega_0(P;x)
            = \lambda \Omega_0(P;x).
            \qedhere
    \]
    %
    % \begin{align*}
    %     \Omega_0(\lambda P;x) 
    %         \,\overset{\mathclap{\smash{\ref{res:Omega0_trafo}}}}=\, \lambda^{-d} \Omega_0(P;x/\lambda)
    %         = \lambda^{-d} \lambda^{d+1+s} \Omega_0(P;x)
    %         = \lambda^{s+1} \Omega_0(P;x).
    % \end{align*}
    %
    % First, recall that $\Omega(\lambda P;x) = \lambda^{-d} \,\Omega(P;\lambda^{-1}\,x)$, and hence the same holds for $\Omega_0$.
    % Then if $s:=\drop(P)$
    % %
    % \begin{align*}
    %     \Omega_0(\lambda P;x) 
    %     &= \lambda^{-d}\,\Omega_0(P;x/\lambda)
    %     \\&= \lambda^{-d}\frac{\adj_P(x_0,x/\lambda)|_{x_0=0}}{\prod_F \<u_F,x/\lambda\>}
    %     \\&= \lambda^{-d}\frac{\lambda^{-(m-d-1-s)}\adj_P(x_0,x)|_{x_0=0}}{\lambda^{-m}\prod_F \<u_F,x\>}
    %     \\&= \lambda^{s+1}\frac{\adj_P(x_0,x)|_{x_0=0}}{\prod_F \<u_F,x\>} 
    %     \\&= \lambda^{s+1} \Omega_0(P;x).
    % \end{align*}
    % %
    %If $s>0$ then $\Omega_0(P)=0$ and the result holds trivially. 
    %And if $s=0$ we obtain the claimed result as well.
\end{proof}

\fi

Homogeneity gives us access to a variety of results from valuation theory.
For~example, 1-homogeneous translation-invariant valuations are \Def{Minkowski additive} (see \cite[Remark 6.3.3]{schneider2013convex}):

\begin{corollary}
\label{res:Minkowski_additivity}
$\Omega_0(P_1+\cdots+ P_n)=\Omega_0(P_1) + \cdots + \Omega_0(P_n).$
\end{corollary}

So far, our repertoire of polytopes with a non-zero drop is rather small.
In particular, in dimension three we are limited to zonotopes.
We can use Minkowski~additivity to construct many more examples:

\begin{example}
    \label{ex:Minkowski_sum}
    Let $P_1,...,P_n\subset\RR^d$ be convex polytopes with empty interior. 
    For them $\Omega_0(P_i)=0$, and so $\Omega_0(P_1+\cdots+ P_n)=0$ by Minkowski additivity.

    Consider for example the 3-dimensional polytope $P$ shown in \cref{fig:half_zonotope} (left)~which is constructed as the Minkowski sum of two triangles.
    %We currently have no other way of proving that $P$ has non-zero drop.
    Since $P$ is not a zonotope, it is our first 3-dimensional example for which we can deduce $\drop(P)=1$.
    Interestingly, this~$P$~can also be obtained as ``half of a zonotope''\! by cutting the latter~with a central hyperplane (see \cref{fig:half_zonotope}, right).
    This will later give us a second proof for $\Omega_0(P)=0$ (see \cref{res:half_cs_polytope}).
\end{example}

\begin{figure}[h!]
    \centering
    \includegraphics[width=0.63\linewidth]{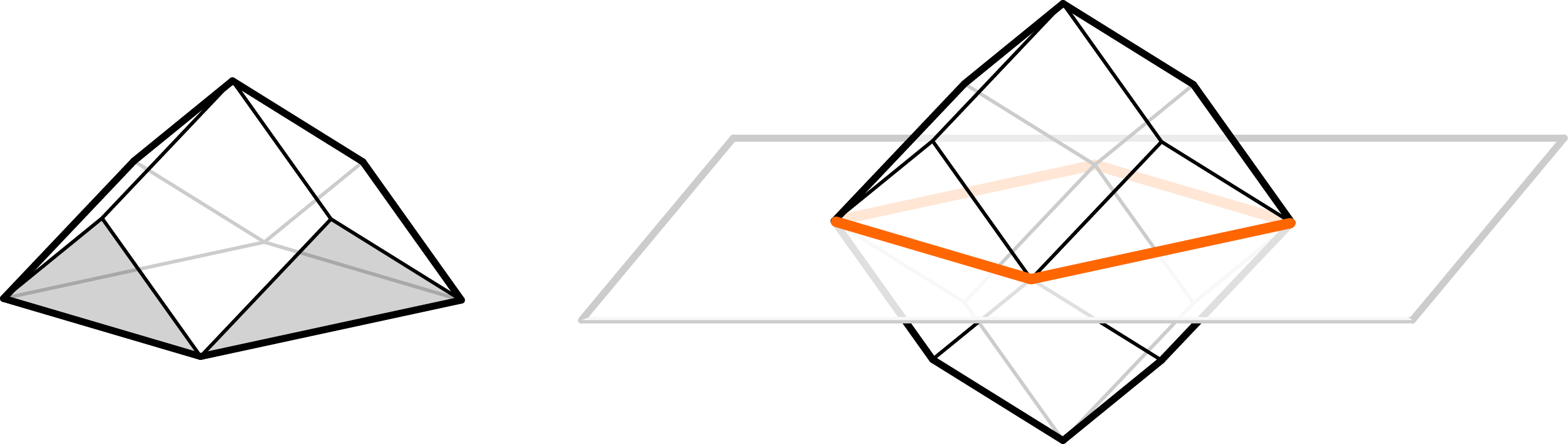}
    \caption{The polytope on the left is the Minkowski sum of the two shaded triangular faces.
    The polytope on the right is a zonotope with~a plane that dissects it into two identical halves, both of which are isometric to the polytope on the left.}
    \label{fig:half_zonotope}
\end{figure}

% As another consequence of Minkowski additivity we obtain a neat characterization of 3-polytopes with drop.

% \begin{theorem}
%     A 3-polytope has $\drop(P)>0$ if and only if $P+(-P)$ is a zonotope. %, the following are equivalent
%     %
%     % \begin{myenumerate}
%     %     \item $\drop(P)>0$.
%     %     \item $P+(-P)$ is a zonotope.
%     % \end{myenumerate}    
% \end{theorem}
% %
% In dimension three, detecting when $P+(-P)$ is relatively easy: this is the case if and only if for each facet $F\subset P$ holds: if $G$ is the antipodal face, then $F+(-G)$ is centrally symmetric.
% %
% \begin{proof}
%     If $\drop(P)>0$ then $\Omega_0(P)=\Omega_0(-P)=0$.
%     By Minkowski additivity we have $\Omega_0(P+(-P))=0$.
%     But $P+(-P)$ is also centrally symmetric, and so $\drop(P)=2$, and hence it is a zonotope.

%     Conversely, suppose that $P+(-P)$ is a zonotope but suppose that $\Omega_0(P)\not=0$.
%     For odd $d$ holds $\Omega_0(P)=\Omega_0(P)$, and hence $\Omega_0(P+(-P))=2\Omega_0(P)\not=0$, in contradiction to the assumption that $P+(-P)$ is a zonotope.

%     The equivalence to \itm3 is elementary.
% \end{proof}

\subsection{A decomposition formula for $\bs{\Omega_0}$}

%For the rest of this section we discuss a decomposition of $\Omega_0$ inspired by a 
The following is a deep theorem due~to McMullen (see \cite[Theorem 6.4.7]{schneider2013convex}): a valuation 
$\phi$ on convex polytopes
%$\phi\:\mathcal P^d\to\RR$ 
%(where $\mathcal P^d$ denotes~the~set of convex polytopes in $\RR^d$) 
is simultane\-ously translation-invariant, 1-homogeneous and weakly continuous if and only if it permits a decomposition
%
% \todo{\msays{I removed $\mathcal P^d$ from the rest of the document because I wasn't sure whether it should refer to polytopes or polyhedra (and it is easy to write around it). Here it is more awkward to avoid, and now I am pondering whether I should reintroduce it}}
%
\begin{equation}
    \label{eq:McMullen_decomposition}
    \phi(P) = \sum_{e\subset P} \ell_e \psi(N_P(e)),
\end{equation}
where the sum is over the edges $e$ of $P$, $\ell_e$ is the length of the edge $e$, $\psi$ is a valuation defined on $(d-1)$-dimensional cones, and $N_P(e)$ is the normal cone at $e$.\nls
\emph{Weak~continuity} means that $\phi$ is continuous under translation of facet hyperplanes.

To apply McMullen's result to $\Omega_0$ we would first need to verify weak continuity, and we would get no information for what valuation $\psi$ to chose.
%This poses technical challenges since $\Omega_0$ take on values not in $\RR$, but in the infinite dimensional vector space $\RR(x)$.
Instead we managed to prove a decomposition as in \eqref{eq:McMullen_decomposition} independently.
%(we could then use McMullen's result in the other direction to conclude weak continuity if desired).

% \begin{theorem}[McMullen]
%     \label{res:McMullen_decomposition}
%     For a valuation $\phi$ the following are equivalent:
%     %
%     \begin{myenumerate}
%         \item $\phi$ is translation-invariant, 1-homogeneous and weakly continuous, and
%         \item there is a valuation $\theta$ defined on $(d-1)$-dimensional cones so that
%         %
%         $$\phi(P) = \sum_{e\subset P} \ell_e \theta(N_P(e)),$$
%         %
%         where the sum is over the edges $e$ of $P$, $\ell_e$ is the length of the edge $e$, and $N_P(e)$ is the normal cone at $e$.
%     \end{myenumerate}
% \end{theorem}

% This result was initially formulated only for valuations $\phi\:\mathcal P^d\to\RR$.
% It easily generalizes to valuations that map to finite dimensional real vector spaces, but certain difficulties arise for infinite dimensional vector spaces such as $\RR[x]$.
% We were however able to derive such a decomposition independently for the valuation $\Omega_0$ thereby also establishing that $\Omega_0$ is weakly continuous.

\begin{theorem}
    \label{res:edge_decomposition}
    For a convex polytopes $P\subset\RR^d$ holds
    \begin{equation}
        \label{eq:edge_decomposition}
        \Omega_0(P;x) = -\frac{1}{\|x\|^2}\sum_{e\subset P} \ell_e\, \Omega^{d-1}(T_e;\pi_e x),
    \end{equation}
    %
    %\todo{\msays{Implement the $\pi_e$ change everywhere.}}%
    where the sum is over the edges $e$ of $P$, $\ell_e$ is the length of edge $e$, 
    %
    %and $T_e:=T_P(e)/\vec e$ is the tangent cone at $e$ projected along the edge direction $\vec e$.
    %
    $\pi_e$ is the orthogonal projection onto $e^\bot$, and $T_e:=\pi_eT_P(e)\subset e^{\bot}$ is the projected tangent cone at $e$.
    %
    %$T_P(e)$ is the~tangent cone at $e$, and $\pi_e$ is the orthogonal projection onto the orthogonal complement $e^\bot$.
\end{theorem}

%This is allowed as a cone can be understood as a polytope with another facet at the hyperplane at infinity (this is easier to understand in the projective picture
We note that \cref{res:edge_decomposition} uses the projected tangent cone $\pi_eT_P(e)$ as opposed~to the normal cone $N_P(e)$ as in \eqref{eq:McMullen_decomposition}.
The theorem can nevertheless be understood~as an instance of \eqref{eq:McMullen_decomposition} because $N_P(e)$ and $T_P(e)$ are related through polar duality,\nls and valuations on cones can be interpreted as valuations on their polar duals.

\iffalse % explanation for Omega on cones - 08/07/2025
\textcolor{lightgray}{Note that in \cref{res:edge_decomposition} the canonical form $\Omega^{d-1}$ applies to $(d-1)$-dimensional cones rather than polytopes.
%\todo{\msays{Tangent cones are always $d$-dim ....}\asays{We could use the notation $T_P(e) / e$ to denote the projection of $T_P(e)$ to the complement of its linearity space.}\msays{I was thinking so as well. Yet I first want to see whether it becomes nicer or worse to actually work with the full tangent cone, see my email.}}%
That this is allowed and yields a meaningful result is easiest to see using the homogenized version of $\Omega$.
Alternatively, a cone can~be~understood as a polytope with one facet hyperplane moved to infinity.}
\fi
%
Alternatively, we can express \eqref{eq:edge_decomposition} as a polynomial identity in terms of adjoints by substituting in \eqref{eq:Omega_represenation} and~\eqref{eq:Omega0_def}.
%Most relevant here, the canonical form of $T_e$ permits a decomposition analogous to \eqref{eq:canonical_form}:
The canonical form of $T_e$ can be written as
\begin{equation}
    \label{eq:canonical_form_on_cones}
    \Omega^{d-1}(T_e;\pi_e x) = (-1)^{m_e}\frac{\adj_{T_e}(\pi_ex)}{\prod_{\substack{F\subset P\\F\supset e}} \<x, u_F\>},
\end{equation}
%\todo{\msays{Do we need a minus sign here as well? This gets out of hand. Do we maybe want to use $\<x,u_F\>-h_F$ instead of $h_F-\<x,u_F\>$ in the definition of $\Omega$? But this would introduce a minus sign in the adjoint ... what is more common?}}
%
where $m_e$ is the number of facets of $T_e$, and we used $\<\pi_ex,u_F\>=\<x,u_F\>$. % since $\<\vec e,u_F\>=0$.
%where $\smash{\adj_{T_e}}$ is the analogously defined adjoint for the cone $T_e$. 
%, and the $u_F\in\Sph^{d-1}$ are the facet unit normal vectors of $C$.
\mbox{Multiplying} \eqref{eq:edge_decomposition} through with $\|x\|^2$ and $\prod_i\<x,u_i\>$ we arrive at a polynomial identity equivalent to \cref{res:edge_decomposition}:
%Substituting this an \eqref{eq:Omega0_def} into \cref{res:edge_decomposition} and multiplying through with $\|x\|^2$ and $\prod_i\<x,u_i\>$ we obtain a polynomial identity that relates the expected leading monomials of $\adj_P$ with the adjoints of the edge tangent cones:
%
\begin{corollary}
    \label{res:edge_polynomial_identity}
    For a convex polytopes $P\subset\RR^d$ holds
    $$\adj_P^0(x) \cdot \|x\|^2 = - \sum_{e\subset P} (-1)^{m-m_e} \ell_e \adj_{T_e}\!(\pi_e x) \!\prod_{\substack{F\subset P\\F\not\supset e}} \!\<x,u_F\>,$$
    where $\adj_P^0(x):=\adj_P(x_0,x)|_{x_0=0}$ is the homogenized adjoint evaluated at $x_0=0$.
    %
    %The product on the right is over all facets $F$ of $P$ that do not contain $e$.
    %
    %where the sum is over the edges $e$ of $P$, $\ell_e$ is the length of edge $e$, %$\hat\Omega$ is the canonical form on $(d-1)$-dimensional cones, 
    %and $T_P(x)$ is the tangent cone at $e$.
\end{corollary}

%Note that \cref{res:edge_polynomial_identity} is equivalent to \cref{res:edge_decomposition}, a fact that we shall use for proving the latter.

% One can show that for a simplicial cone $C$ with normal vectors $u_1,...,u_d\in\Sph^{d-1}$
% %
% \begin{equation}
%     \label{eq:form_on_cones}
%     \Omega^d(C;x) = \frac{\det(u_1,...,u_d)}{\<x,u_1\>\cdots\<x,u_d\>}.    
% \end{equation}
% %
% For general cones the value can be obtained through triangulation.

To prove \cref{res:edge_decomposition} we can employ a standard strategy from valuation theory: since both sides of \eqref{eq:edge_decomposition} are valuative, we can decompose the polytope into simpler pieces (typically simplices) and prove the statement only on these simple pieces.
It does however turn out that already on simplices \cref{res:edge_decomposition} is a rather non-trivial statement that is interesting in its own right.
We therefore first explore the simplex case in \cref{sec:decomposition_on_simplices}, and postpone a proof to \cref{sec:decomposition_proof}.

\subsection{\cref{res:edge_decomposition} for simplices}
\label{sec:decomposition_on_simplices}

For a $d$-dimensional simplex $\Delta\subset\RR^d$ the adjoint is of expected and actual degree $m-d-1=(d+1)-d-1=0$.
That is, $\adj_\Delta$ is a number and is invariant under restriction to monomials of expected maximal~degree.
Likewise, tangent cones at edges of $\Delta$ are simplicial, and their adjoints too are~just numbers.
Lastly, each edge lies in all but two facets.
If we denote by $e_{ij}$ the edge that~does~\emph{not} lie in facets $F_i$ and $F_j$, then
\cref{res:edge_polynomial_identity} simplifies to
%
\iffalse

%We want to apply \cref{res:edge_decomposition} to it.
Recall that $\adj_\Delta$ has degree zero, that is, is just a number.
Hence it is invariant under~re\-striction to the leading monomial, and
%
\begin{equation}
    \label{eq:Omega0_simplex}
    \Omega_0(\Delta;x) = \frac{\adj_\Delta}{\<x,u_0\>\cdots\<x,u_d\>}.
\end{equation}
%where we substituted in the known value of $\adj_P$ for simplices.
%
At the same time, all tangent cones at edges of $\Delta$ are simplicial and
their canonical forms are of an equally simple form:
if $e_{ij}\subset \Delta$ is the edge that is contained in all facets but $F_i$ and $F_j$, then
%For this, let $v_i$ be the vertex opposite to the facet $F_i$, and let $e_{ij}$ be the edge with end vertices $v_i$ and $v_j$.
%Then $e_{ij}$ is contained in every facet but $F_i$ and $F_j$, and so we have
%
\begin{equation}
    \label{eq:Omega_simplicial_cone}
    \Omega^{d-1}(T_\Delta(e_{ij});x) = \frac{\adj_{T_\Delta(e_{ij})}}{\prod_{k\not={i,j}}\<x,u_k\>}.    
\end{equation}
%
Again, $\adj_{T_\Delta(e_{ij})}$ is just a number.

Substituting \eqref{eq:Omega0_simplex} and \eqref{eq:Omega_simplicial_cone} into \cref{res:edge_decomposition} and multiplying both sides with $\<x,u_0\>\cdots\<x,u_d\>$ and $\|x\|^2$, we obtain

\fi
%
\begin{equation}
    \label{eq:adjoint_identitiy_for_simplex}
    \adj_\Delta \|x\|^2 = -\sum_{i<j} \ell_{ij} \adj_{T_{ij}} \<x,u_i\>\<x,u_j\>,
\end{equation}
%\todo{\msays{Something is off with the signs. Adjoints of simplices are always positive numbers (I think), and so the above identity cannot hold in odd dimensions. Maybe \cref{res:edge_decomposition} is currently missing some sign?} \msays{I am 99\% sure that the single minus sign in the adjoint equation is correct. It means something cancels the $(-1)^d$ somewhere else.}}
%
where $\ell_{ij}$ and $T_{ij}$ are the length and projected tangent cone at $e_{ij}$ respectively.\nls
This is a curious identity:
on the right hand side we construct a quadratic form by~summing over the edges of a simplex; but the left hand side shows that this form reduces to a multiple of $\|x\|^2$. 
Alternatively, this can be expressed as a matrix identity:
%Writing quadratic forms as matrices, this is equivalent~to
%\todo{\msays{Include the matrix form?}}
%
$$\adj_\Delta \Id_d = -\tfrac12\sum_{i<j} \ell_{ij}\adj_{T_{ij}} (u_iu_j\T\!+u_ju_i\T).$$

% In case of simplices and simplicial cones, the adjoints can be expressed explicitly.
% % %It becomes even more curious once we substitute in values for the adjoints. 
% % For simplices the exact value is known to be
% %
% %In case of simplices and simplicial cones, the adjoints can be expressed explicitly.
% If the simplex $\Delta$ has normals $u_0,...,u_d\in\Sph^{d-1}\!$ and facet heights $h_0,...,h_d\in\RR$, and if a (full-dimensional) simplicial cone $C$ has normals $u_1,...,u_d\in\Sph^{d-1}\!$, then~their~res\-pective adjoints are
%

Our next goal is to rewrite \eqref{eq:adjoint_identitiy_for_simplex} more explicitly by substituting in the known expressions for adjoints of simplices and simplicial cones (\cf\ \cref{ex:simplices}).\nls
\iftrue % toggle old vs new explanation of T_ij adjoint computation
For~conve\-nience (and later generalization) we shall assume that $u_0,...,u_d$ form a positive affine basis, \ie\ the determinant in \eqref{eq:adjoint_simplex} is non-negative and the absolute value can be dropped:
%
%Let us rewrite \eqref{eq:adjoint_identitiy_for_simplex} more explicitly.
%For a simplex $\Delta$ and a simplicial cone~$C$~the adjoints can be expressed as determinants:
%
%\vspace{-0.99em}
\begin{equation}
    \label{eq:simplicial_adj_pos}
    \adj_\Delta=\det\begin{pmatrix}
        | & & | \\[-0.5ex]
        u_0 & \cdots & u_d \\
        | & & | \\[1ex]
        h_0 & \cdots & h_d
    \end{pmatrix}.
\end{equation}
%
%Here we assume that the normal vectors $u_0,...,u_d$ (\resp\ $u_1,...,u_d$ for the cone) are ordered so that the determinants are positive.
%
%We want to substitute this into \eqref{eq:adjoint_identitiy_for_simplex} to obtain a determinantal identity.
%
%The expression for the simplex we can use right away.
For the simplicial cone $T_{ij}$ $\subset\RR^d$ we cannot use \eqref{eq:adjoint_simplicial_cone} right away, but need to adjust for the fact that $T_{ij}$ is of dimension $d-1$.
%
%where $h_i$ is the height of facet $F_i$ over the origin.
%Likewise, the adjoint of a full-dimensional simplicial cone $C$ can be expressed as
%
%$$\adj_C=\det\begin{pmatrix}
%     | & & | \\[-0.5ex]
%     u_1 & \cdots & u_d \\
%     | & & | \\[1ex]
% \end{pmatrix}.$$
%
%Since $T_{ij}\subset\RR^{d}$ is a cone of dimension $d-1$, in order to express its adjoint as a determi\-nant as in \eqref{eq:simplicial_adj_pos}, we need to introduce an additional column that is~orthogonal to the others.
To express $\smash{\adj_{T_{ij}}}$ as a $d\times d$-determinant as in \eqref{eq:adjoint_simplicial_cone}, we have to introduce an additional column that is orthogonal to the~others.
%
%we can apply \eqref{eq:form_on_cones} but need to account for the fact that our cones ar not full-dimensional.
%That is, in order to express the adjoint of $T_P(e_{ij})$ as a $d\times d$-determinant, we need to include another vector orthogonal to the others.
For~this we can use the (normalized) edge direction $\vec e_{ij}$, oriented from~$i$~to~$j$.\nls
Conve\-niently, this allows us to also include the edge length into the determinantal~expres\-sion.
More precisely, if $\Delta$ has vertices $v_0,...,v_d$ (with $v_i$ being opposite to the facet with normal $u_i$), and $\hat u_i$ denotes the absence of $u_i$ from the list $u_0,...,u_d$, then we can write
\begin{align}
\ell_{ij}\adj_{T_{ij}}\, 
&\equiv\,
\ell_{ij}\cdot \det\!\begin{pmatrix}\notag
    & | &  & | & & |
    \\%[-0.5ex]
    %u_0 & \cdots & u_d & v_j-v_i
    \bdots & \hat u_i & \ndots & \hat u_j & \edots\!\! & \vec e_{ij}
    \\
    & | &  & | & & |
\end{pmatrix}
\\&=\,
\det\!\begin{pmatrix}\label{eq:Dij_form}
    & | &  & | & & |
    \\%[-0.5ex]
    %u_0 & \cdots & u_d & v_j-v_i
    \bdots & \hat u_i & \ndots & \hat u_j & \edots\!\! & v_j-v_i
    \\
    & | &  & | & & |
\end{pmatrix}
\end{align}
where $\equiv$ denotes identity up to sign.
Tracking the signs in these identities is somewhat cumbersome.
In the following we shall instead use the following more elegant determinantal expression, which is equivalent \emph{including signs}:
\begin{align}
\label{eq:Dij_def}
\ell_{ij}\adj_{T_{ij}} = 
-\det\!\underbrace{\begin{pmatrix}
    | &  & | &  & | & & |
    \\[-0.5ex]
    u_0 & \ndots & v_i & \ndots & v_j & \ndots & u_d
    \\
    | &  & | &  & | & & |
    \\[1ex]
    0 & \ndots & 1 & \ndots & 1 & \ndots & 0
\end{pmatrix}}_{=:D_{ij}}.
\end{align}
Here, $D_{ij}$ is the $(d+1)\times(d+1)$-matrix whose columns are $(u_k,0)\T\!$ for $k\not\in\{i,j\}$, and $(v_j,1)\T\!$ otherwise.
The equivalence of these expressions is shown in \cref{res:lij_edj_Tij}.

\else %%%%%%%%%%%%%%%%%%%%%%%%%%%%%%%%%%%%% ELSE
For~conve\-nience we assume a suitable ordering of the normal vectors so that the determinantal expressions for the adjoints are positive. We thne have the expressions:
%
%Let us rewrite \eqref{eq:adjoint_identitiy_for_simplex} more explicitly.
%For a simplex $\Delta$ and a simplicial cone~$C$~the adjoints can be expressed as determinants:
%
\begin{equation*}
    \label{eq:simplicial_adj_pos}
    \adj_\Delta=\det\begin{pmatrix}
        | & & | \\[-0.5ex]
        u_0 & \cdots & u_d \\
        | & & | \\[1ex]
        h_0 & \cdots & h_d
    \end{pmatrix},
    \qquad
    \adj_C=\det\begin{pmatrix}
        | & & | \\[-0.5ex]
        u_1 & \cdots & u_d \\
        | & & | \\[1ex]
    \end{pmatrix}.
\end{equation*}
%
%Here we assume that the normal vectors $u_0,...,u_d$ (\resp\ $u_1,...,u_d$ for the cone) are ordered so that the determinants are positive.
%
%We want to substitute this into \eqref{eq:adjoint_identitiy_for_simplex} to obtain a determinantal identity.

The expression for the simplex we can use right away.
For the simplicial cone~$T_{ij}$ $\subset\RR^d$ we need to adjust for the fact that $T_{ij}$ is of dimension $d-1$.
%
%where $h_i$ is the height of facet $F_i$ over the origin.
%Likewise, the adjoint of a full-dimensional simplicial cone $C$ can be expressed as
%
%$$\adj_C=\det\begin{pmatrix}
%     | & & | \\[-0.5ex]
%     u_1 & \cdots & u_d \\
%     | & & | \\[1ex]
% \end{pmatrix}.$$
%
%Since $T_{ij}\subset\RR^{d}$ is a cone of dimension $d-1$, in order to express its adjoint as a determi\-nant as in \eqref{eq:simplicial_adj_pos}, we need to introduce an additional column that is~orthogonal to the others.
To express $\adj_{T_{ij}}$ as a $d\times d$-determinant, we have to introduce an additional column that is orthogonal to the others.
%
%we can apply \eqref{eq:form_on_cones} but need to account for the fact that our cones ar not full-dimensional.
%That is, in order to express the adjoint of $T_P(e_{ij})$ as a $d\times d$-determinant, we need to include another vector orthogonal to the others.
For this we use the (normalized) edge direction~$\vec e_{ij}$,\nls oriented from $i$ to $j$.
Conveniently, this allows us to include the edge length into the determinantal expression as well.
We found the following to be an elegant way to express this:\nls
for $i<j$ let $D_{ij}$ be the negative of the $(d+1)\times(d+1)$-determinant whose columns are $(u_k,0)\T\!$ for $k\not\in\{i,j\}$, and $(v_j,1)\T\!$ otherwise:
\begin{align}
\label{eq:Dij_def}
D_{ij} := 
-\det\!\begin{pmatrix}
    | &  & | &  & | & & |
    \\[-0.5ex]
    u_0 & \ndots & v_i & \ndots & v_j & \ndots & u_d
    \\
    | &  & | &  & | & & |
    \\[1ex]
    0 & \ndots & 1 & \ndots & 1 & \ndots & 0
\end{pmatrix}.
\end{align}
Let $\hat u_i$ denote the absence of $u_i$ from the list $u_0,...,u_d$.
We indeed find
\todo{\msays{I checked for all signs numerically that they give positive expressions. I did not check this analytically. If someone wants to try ....... It should suffice to check the last identity. It is correct, if we ignore the sign.}}
\begin{align}
%\qquad\qquad\!\!
D_{ij}
&=\notag
(-1)^{i+j}\det\!\begin{pmatrix}
    & | &  & | & & | & |
    \\%[-0.5ex]
    %u_0 & \cdots & u_d & v_j-v_i
    \bdots & \hat u_i & \ndots & \hat u_j & \edots & v_i & v_j
    \\
    & | &  & | & & | & |
    \\[1ex]
    \bdots & 0 & \ndots & 0 & \ndots & 1 & 1
\end{pmatrix}
\\&=\notag
(-1)^{i+j}\det\!\begin{pmatrix}
    & | &  & | & & | & |
    \\%[-0.5ex]
    %u_0 & \cdots & u_d & v_j-v_i
    \bdots & \hat u_i & \ndots & \hat u_j & \edots & v_i & v_j - v_i
    \\
    & | &  & | & & | & |
    \\[1ex]
    \bdots & 0 & \ndots & 0 & \ndots & 1 & 0
\end{pmatrix}
\\&= \label{eq:Dij_form}
(-1)^{i+j+1}\det\!\begin{pmatrix}
    & | &  & | & & |
    \\%[-0.5ex]
    %u_0 & \cdots & u_d & v_j-v_i
    \bdots & \hat u_i & \ndots & \hat u_j & \edots & v_j-v_i
    \\
    & | &  & | & & |
\end{pmatrix}
\\&=\notag
(-1)^{i+j+1}\ell_{ij}\cdot \det\!\begin{pmatrix}
    & | &  & | & & |
    \\%[-0.5ex]
    %u_0 & \cdots & u_d & v_j-v_i
    \bdots & \hat u_i & \ndots & \hat u_j & \edots & \vec e_{ij}
    \\
    & | &  & | & & |
\end{pmatrix}
= 
\ell_{ij} \adj_{T_{ij}} 
\end{align}
The choice of signs guarantees a positive expression throughout, as long as $i<j$.
\fi

Substituting \eqref{eq:simplicial_adj_pos} and \eqref{eq:Dij_def} into \eqref{eq:adjoint_identitiy_for_simplex} we obtain a curious identity for quadratic forms involving determinants:
%
% \newpage
% %
% %
% that is, the columns are $(u_k,0)\T$\!, except for $k\in\{i,j\}$, where the columns is $(v_k,1)\T$\!.
% We further rewrite this determinant. Hereby $\hat u_i$ and $\hat u_j$ denote that $u_i$ and $u_j$ are missing from the list $u_0,...,u_d$.
% %
% \begin{align*}
% D_{ij} &:= 
% \det\begin{pmatrix}
%     & | &  & | & & | & |
%     \\%[-0.5ex]
%     %u_0 & \cdots & u_d & v_j-v_i
%     \bdots & \hat u_i & \ndots & \hat u_j & \edots & v_j & v_i
%     \\
%     & | &  & | & & | & |
%     \\[1ex]
%     \bdots & 0 & \ndots & 0 & \edots & 1 & 1
% \end{pmatrix}
% = 
% \det\begin{pmatrix}
%     & | &  & | & & | & |
%     \\%[-0.5ex]
%     %u_0 & \cdots & u_d & v_j-v_i
%     \bdots & \hat u_i & \ndots & \hat u_j & \edots & v_j-v_i & v_i
%     \\
%     & | &  & | & & | & |
%     \\[1ex]
%     \bdots & 0 & \ndots & 0 & \edots & 0 & 1
% \end{pmatrix}
% \\[1ex] &=
% \det\begin{pmatrix}
%     & | &  & | & & |
%     \\%[-0.5ex]
%     %u_0 & \cdots & u_d & v_j-v_i
%     \bdots & \hat u_i & \ndots & \hat u_j & \edots & v_j-v_i
%     \\
%     & | &  & | & & |
% \end{pmatrix}
% =
% \|v_j-v_i\|\det\begin{pmatrix}
%     & | &  & | & & |
%     \\%[-0.5ex]
%     %u_0 & \cdots & u_d & v_j-v_i
%     \bdots & \hat u_i & \ndots & \hat u_j & \edots & \vec e_{ij}
%     \\
%     & | &  & | & & |
% \end{pmatrix}
% \\[1ex] &= 
% \ell_{ij}\adj_{T_\Delta(e_{ij})}.
% \end{align*}
%
% Substituting $D_{ij}$ and the determinantal expression for $\adj_\Delta$ into \cref{res:edge_decomposition}, we arrive at the following curious matrix identity:
%
\begin{theorem}
    \label{res:simplex_qudratic_forms}
    If $\Delta\subset\RR^d$ is a $d$-dimensional simplex with normal vectors $u_i\in\RR^d\setminus\{0\}$, heights $h_i\in\RR$ and vertices $v_i\in\RR^d$, then
    \begin{equation}
        \label{eq:simplex_qudratic_forms}
        \det\begin{pmatrix}
            | & & | \\[-0.5ex]
            u_0 & \ndots & u_d \\
            | & & | \\[1ex]
            h_0 & \ndots & h_d
        \end{pmatrix}
        \|x\|^2 = \sum_{i<j}
            \det\begin{pmatrix}
                | &  & | &  & | & & |
                \\[-0.5ex]
                u_0 & \ndots & v_i & \ndots & v_j & \ndots & u_d
                \\
                | &  & | &  & | & & |
                \\[1ex]
                0 & \ndots & 1 & \ndots & 1 & \ndots & 0
            \end{pmatrix}
            \<x,u_i\>\<x,u_j\>.
    \end{equation}
\end{theorem}

In \cref{res:simplex_qudratic_forms} we do no longer need that the $u_i$ are normalized.
This is because both sides are linear in $(u_i,h_i)\T\!$. % (when simultaneously scaling the $h_i$).
We also do not need to assume a specific~order~for the indices to guarantee positive determinants.
This is because~both~sides~flip~signs under a swap of indices.
So far our only proof of \cref{res:simplex_qudratic_forms} relies on the valuation properties of $\Omega_0$ (see \cref{sec:decomposition_proof}).
%While our proof of \cref{res:simplex_qudratic_forms} relies on the valuation properties of $\Omega_0$, 
Due to its elementary statement, a purely matrix-theoretic proof would be desirable.

We also remark on special case $d=2$, that is, $\Delta$ is a triangle.
Here the projected edge tangent cones $T_{ij}$ are 1-dimensional and so $\adj_{T_{ij}} = 1$.
Therefore \cref{res:simplex_qudratic_forms} further simplifies to
\begin{equation}
    \label{eq:trangle_adjoint}
    \adj_\Delta \|x\|^2 = \sum_{i<j} \ell_{ij} \<x,u_i\>\<x,u_j\>.
\end{equation}

Note that even though the entries $h_i$ of $\adj_\Delta$ depend on the translation of $\Delta$, the adjoint itself does not.
The adjoint therefore expresses some geometric property of $\Delta$. 
In dimension $d=2$ we identified the identity
$$\adj_\Delta=\frac{\operatorname{Area}(\Delta)}{\operatorname{Circumradius}(\Delta)}.$$
A proof is given in \cref{sec:appendix_triangle_adjoint}.
We are not aware of an analogously geometric~interpretation in dimensions $d\ge 3$.

\begin{question}
    \label{q:adj_simplex_geometrically}
    What does $\adj_\Delta$ express geometrically if $d\ge 3$?
\end{question}

\subsection{Proof of \cref{res:edge_decomposition}}
\label{sec:decomposition_proof}

%The standard strategy for proving statements about valuations on $\mathcal P^d$ is to prove them on some simple building block, such as simplices, and then extend it to all polytopes by invoking some decomposition theorem such as the existence of triangulations.
%In our case the statement for simplices~(see~\cref{res:simplex_qudratic_forms}) is already surprisingly complicated and we are not aware of an independent proof.

As discussed previously, we prove \cref{res:edge_decomposition} by decomposing $P$ into simpler polytopes.
Since the statement is already complicated on simplices (see \cref{res:simplex_qudratic_forms}), we need to choose even simpler pieces. 

An \Def{ortho-simplex} (also known as an \Def{orthoscheme}) is a simplex $\Delta$ (or any simplex isometric to it) whose vertices are of the form
%
%\begin{samepage}
\begin{align*}
    %&(\mathrlap0\phantom{h_1},\mathrlap0\phantom{h_2},\mathrlap0\phantom{h_3},...,\mathrlap0\phantom{h_4}),\\
    %&(h_1,\mathrlap0\phantom{h_2},\mathrlap0\phantom{h_3},...,\mathrlap0\phantom{h_4}),\\
    v_0&=(0,0,0,...,0),\\
    v_1&=(\ell_1,0,0,...,0),\\
    v_2&=(\ell_1,\ell_2,0,...,0),\\
    v_3&=(\ell_1,\ell_2,\ell_3,...,0),\\
    &\;\;\vdots \\
    v_d&=(\ell_1,\ell_2,\ell_3,...,\ell_d),
\end{align*}
for some positive numbers $\ell_1,...,\ell_d>0$ (see \cref{fig:orthoschemes}).

\begin{figure}[h!]
    \centering
    \includegraphics[width=0.65\textwidth]{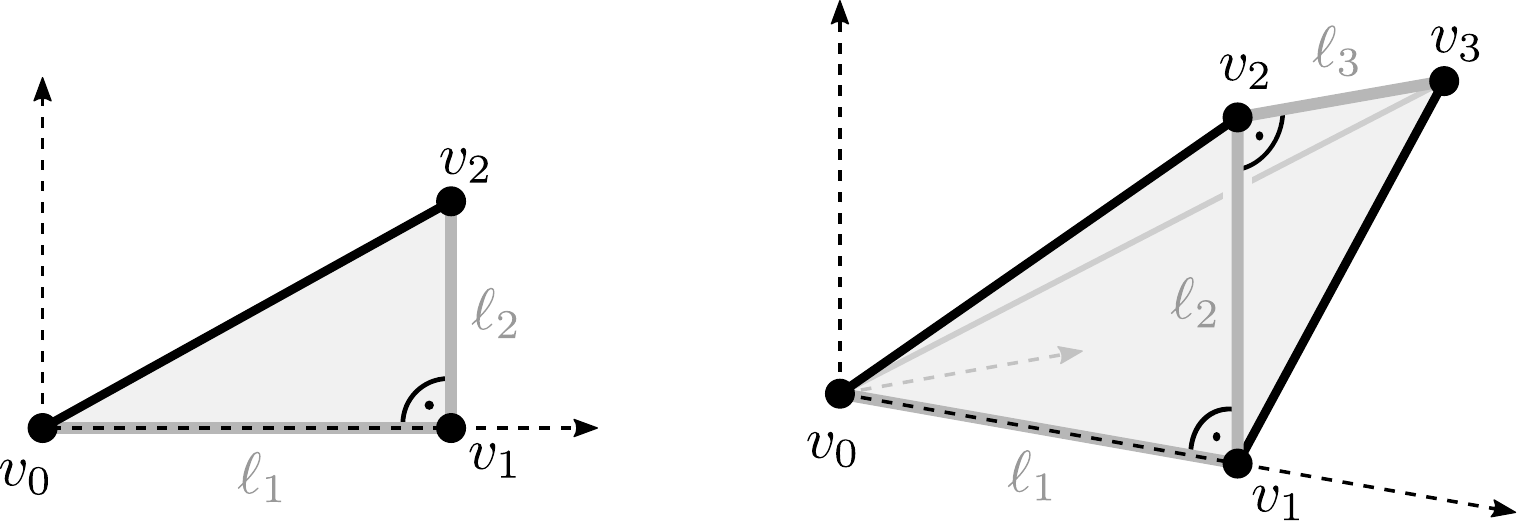}
    \caption{Ortho-simplices of dimension two and three.}
    \label{fig:orthoschemes}
\end{figure}

Crucially, every polytope $P$ has an \Def{ortho-simplex decomposition}, that is, it can be written as a \textit{signed} sum of ortho-simplices.
This is standard, but we give a~quick description of the decomposition.
First, fix a point $x_P\in\RR^d$.
For each face $F\subset P$ let $x_F$ be the orthogonal projection of $x_P$ onto $\aff(F)$.
Then for each flag $\mathcal F=\{P\supset F_{d-1}\supset\cdots\supset F_1\supset F_0\}$ the convex hull of $\{x_{F_i}\mid F_i\in\mathcal F\}$ is an ortho-simplex $\Delta_{\mathcal F}$ (see \cref{fig:orthoschemes_decomposition}).
Let $u_i\in\aff(F_i)$ be the outwards pointing normal vector to $F_{i-1}$ considered as a facet of $F_i$ and let $k(\mathcal F)$ count the number of $i\in\{1,...,d\}$ for which $\<x_{F_{i-1}}-x_{F_i},u_i\><0$. 
One can show that
%\end{samepage}
%
$$\phi(P)=\sum_{\mathcal F} (-1)^{k(\mathcal F)} \phi(\Delta_{\mathcal F}),$$
whenever $\phi$ is a simple valuation, such as $\Omega$ or $\Omega_0$.
The goal here is to prove~\cref{res:simplex_qudratic_forms} for ortho-simplices, which gives us the full statement by the above ortho-simplex decomposition.

\begin{figure}[h!]
    \centering
    \includegraphics[width=0.60\linewidth]{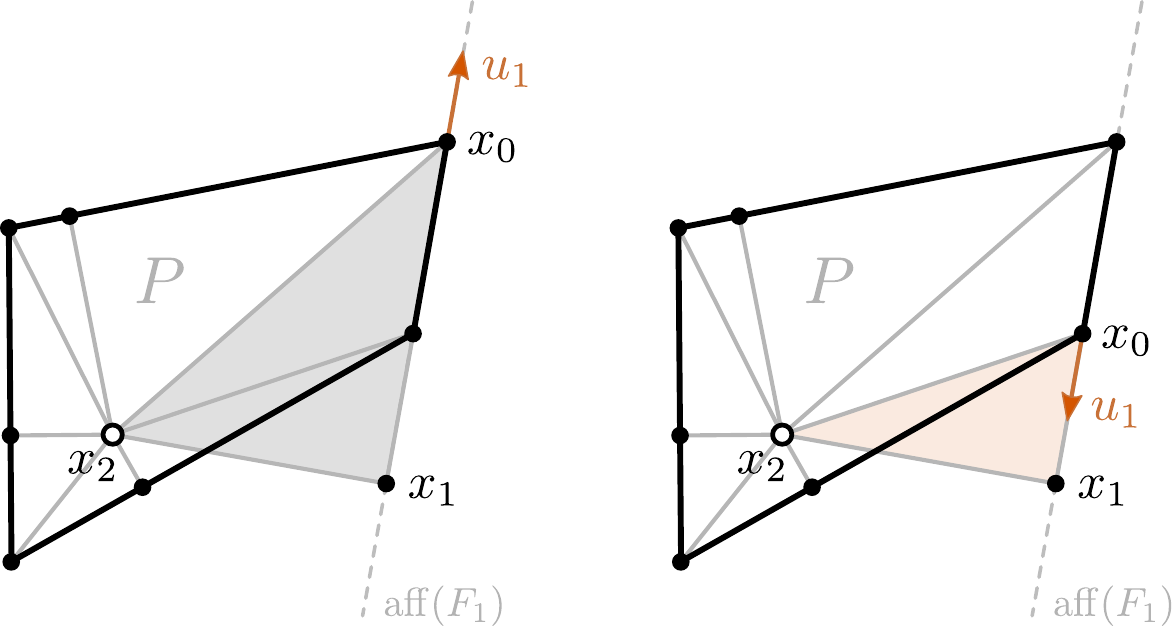}
    \caption{
        An ortho-simplex decomposition of a quadrangle. 
        Left:\nls a~po\-sitively signed ortho-simplex in the decomposition.
        Right: a negatively signed ortho-simplex in the decomposition.
    }
    \label{fig:orthoschemes_decomposition}
\end{figure}

First, one verifies that the (non-normalized) normal vectors of an ortho-simplex $\Delta$ are given by
\begin{align*}
    u_0&=(\ell_0,0,0,...,0,0),\\
    u_1&=(-\ell_2,\ell_1,0,...,0,0),\\
    u_2&=(0,-\ell_3,\ell_2,...,0,0),\\
    u_3&=(0,0,-\ell_4,\ell_3,...,0,0),\\
    &\;\;\vdots \\
    u_{d-1}&=(0,0,0,...,-\ell_d,\ell_{d-1}),\\
    u_d&=(0,0,0,...,0,-\ell_{d+1}),
\end{align*}
where $u_i$ is the normal of facet $F_i$ opposite to vertex $v_i$. We introduced additional parameters $\ell_0,\ell_{d+1}>0$ for notational convenience; one might set $\ell_0=\ell_{d+1}=1$.

The height of each facet over the origin is $0$, except for the height of $F_0$, which is $\ell_0\ell_1$. 
Following \eqref{eq:simplicial_adj_pos} we can express the adjoint of $\Delta$ as
%
% $$
% \adj_\Delta
% =
% \det\begin{pmatrix}
%     | & | &  & | \\[-0.5ex]
%     u_0 & u_1 & \ndots & u_d \\
%     | & | & & | \\[1ex]
%     \ell_0\ell_1 & 0 & \ndots & 0
% \end{pmatrix}
% =
%
% $$
%
\begin{align*}
\adj_\Delta
    &=
    % \det\left(\;\begin{matrix}
    %     \horzbar\, u_0\,\horzbar & \ell_0\ell_1 \\
    %     \horzbar\, u_1\,\horzbar & 0 \\[-0.3ex]
    %     \vdots & \vdots \\
    %     \horzbar\, u_d\,\horzbar & 0 \\
    % \end{matrix}\;\right) 
    \det\begin{pmatrix}
        | & \!\!\!\!| & & | \\[-0.5ex]
        u_0 & \!\!\!\!u_1 & \ndots & u_d \\
        | & \!\!\!\!| & & | \\[1ex]
        \ell_0\ell_1 & \!\!\!\!0 & \ndots & 0
    \end{pmatrix}
    =(-1)^{d+1}\ell_0\ell_1\cdot 
    % \det\left(\;\begin{matrix}
    %     \horzbar\, u_1\,\horzbar \\[-0.3ex]
    %     \vdots \\
    %     \horzbar\, u_d\,\horzbar \\
    % \end{matrix}\;\right) 
    \det\begin{pmatrix}
        | & & | \\[-0.5ex]
        u_1 & \ndots & u_d \\
        | & & | \\[1ex]
    \end{pmatrix}
    \\[1ex]
    &= (-1)^{d+1}\ell_0\ell_1 \cdot
    % \det\left(\;\begin{matrix}
    %     -\ell_2 & \ell_1
    %     \\ & \mathllap-\ell_3 & \ell_2
    %     \\ & & \ddots & \ddots
    %     \\ & & & \mathllap- \ell_{d} &
    %     \ell_{d-1}
    %     \\ & & & & \mathllap- \ell_{d+1}
    % \end{matrix}\;\right) 
    \det\left(\;\begin{matrix}
        -\ell_2
        \\
        \phantom+\ell_1 & \mathllap-\ell_3
        \\[-1.2ex]
        & \ell_2 & \!\!\!\ddots\!\!\!\!
        \\[-1.2ex]
        & & \!\!\!\ddots\!\!\!\! & -\ell_d
        \\
        & & & \ell_{d-1} & \,\mathllap-\ell_{d+1}
    \end{matrix}\;\right) 
    = \ell_0 \cdots \ell_{d+1}. 
\end{align*}
%\notodo{\tsays{ missing factor:  $(-1)^{d-1}?$}\msays{Fixed?}}

A likewise elementary but unlike more involved computation yields an expression for $\det(D_{ij})=\ell_{ij}\adj_{T_{ij}}$ as defined in \eqref{eq:Dij_def}:
%We do so in the form \eqref{eq:Dij_form}.
%
% Let $D_{ij}$ be the matrix whose rows are the $u_k$ for all $k\not\in\{i,j\}$; and an additional row $v_i-v_j$.
% %
% The eventual goal of this section is to show the following identity:
% %
% \begin{equation}
%     \label{eq:id}
%    \adj_\Delta \|x\|^2 = -\sum_{\substack{i,j=0\\i<j}}^d |\!\det(D_{ij})|\<x,u_i\>\<x,u_j\>. 
% \end{equation}
%
% \begin{lemma}
for $i<j$ we obtain
    $$\det(D_{ij}) = \frac{\ell_{i+1}^2+\cdots +\ell_j^2}{\ell_i\ell_{i+1}\ell_j \ell_{j+1}} \cdot \ell_0\cdots \ell_{d+1}.$$
% \end{lemma}
% %
% \begin{proof}
%     %\include{DeterminantFormula}
% \end{proof}
%
The details of this computation can be found in \cref{sec:appendix_Dij}.

Lastly, we observe $\<x,u_i\> = \ell_i x_{i+1} - \ell_{i+1} x_i$, where we set $x_0=x_{d+1}=0$. After substituting all of the above expressions into \eqref{eq:adjoint_identitiy_for_simplex} and canceling $\ell_0\cdots \ell_{d+1}$ on both sides, we are left with the following identity that we need to verify:
\begin{align*}
\sum_{k=1}^d x_k^2 
&= -\!\sum_{\substack{i,j=0\\i<j}}^d \frac{\ell_{i+1}^2+\cdots+ \ell_j^2}{\ell_i\ell_{i+1}\ell_j\ell_{j+1}} (\ell_i x_{i+1} - x_{i+1} x_i)(\ell_j x_{j+1} - x_{j+1} x_j)
\\&= -\!\sum_{\substack{i,j=0\\i<j}}^d (\ell_{i+1}^2+\cdots+ \ell_j^2)\Big(\frac{x_{i+1}}{\ell_{i+1}} - \frac{x_{i}}{\ell_{i}}\Big)\Big(\frac{x_{j+1}}{\ell_{j+1}} - \frac{x_{j}}{\ell_{j}}\Big).
\end{align*}
We can slightly simplify by substituting $y_k:=x_k/\ell_k$:
\begin{align*}
\sum_{k=1}^d \ell_k^2 y_k^2 
&= -\!\sum_{\substack{i,j=0\\i<j}}^d (\ell_{i+1}^2+\cdots+ \ell_j^2)(y_{i+1} - y_i)(y_{j+1} - y_j).
\end{align*}
This identity can now be verified by sorting the terms on the right side according~to $\ell_k^2$, and recognizing two telescoping sums:
\begin{align*}
    -&\!\sum_{\substack{i,j=0\\i<j}}^d (\ell_{i+1}^2+\cdots+ \ell_j^2)(y_{i+1} - y_i)(y_{j+1} - y_j)
    \\[-4ex]&\qquad\qquad=
    -\!\sum_{k=1}^d \ell_k^2 \!\!\sum_{\substack{i,j=0\\i<k\le j}}^d\!\! (y_{i+1}-y_i)(y_{j+1}-y_j)
    \\[-1ex]&\qquad\qquad=
    -\!\sum_{k=1}^d \ell_k^2 \!\!\underbrace{\sum_{0\le i<k}\!\!(y_{i+1}-y_i)}_{=\,y_k-y_0\,=\,y_k} \underbrace{\sum_{k\le j\le d}\!\! (y_{j+1}-y_j)}_{=\,y_{d+1}-y_k\,=\,-y_k}
    =\sum_{k=1}^d \ell_k^2 y_k^2.
\end{align*}
This finishes the proof of \eqref{eq:adjoint_identitiy_for_simplex} for ortho-simplices, and hence also the proof~of~\cref{res:edge_decomposition}.

\section{Central inversion}
\label{sec:central_inversion}

In this section we consider the behavior of $\Omega_0$ under central inversion $P\mapsto -P$. % and explore implications.

\begin{lemma}
    \label{res:central_inversion}
    \label{res:partity}
    $\Omega_0(-P) = (-1)^{d+1} \Omega_0(P)$.
    %\todo{\asays{This follows directly from Lemmas \ref{lemma:equivariance_omegas} and \ref{lemma:hom_rational_functions} now.}}%
\end{lemma}
\begin{proof}
    The proof is analogous to \cref{res:1_homogeneous}.
    The statement is clear if $\Omega_0(P)=0$.
    Otherwise, $\Omega_0(P)$ is a homogeneous rational function of degree $-d-1$.
    Using~\cref{res:Omega0_trafo} we obtain   
    \[
    \Omega_0(-P;x) \,\overset{\mathclap{\smash{\ref{res:Omega0_trafo}}}}=\, \Omega_0(P;-x) = (-1)^{-d-1}\, \Omega_0(P;x).
    \qedhere
    \]
    % The statement is clear if $\Omega_0(P)=0$.
    % Otherwise, $\Omega_0(P)$ is a homogeneous~rational function of degree $-d-1-\drop(P)$.
    % Using \cref{res:Omega0_trafo} and $\lambda>0$ we~obtain
    % we use \cref{res:Omega0_trafo} and the fact that $\Omega_0$ is a homogeneous rational function of degree $-d-1-\drop(P)$. If we~abbreviate $s:=\drop(P)$, then
    %
    % $$\Omega_0(-P;x) \,\overset{\mathclap{\smash{\ref{res:Omega0_trafo}}}}=\, \Omega_0(P;-x) = (-1)^{-d-1-s}\, \Omega_0(P;x).$$
    % %
    % If $s>0$ then $\Omega_0(P)=0$ and the result holds trivially. 
    % And if $s=0$ we obtain the claimed result as well.
\end{proof}

As previously advertised in \cref{ex:Minkowski_sum}, we can now confirm that indeed ``half~of a centrally symmetric polytope'' has a drop  in odd dimensions:

\begin{corollary}
    \label{res:half_cs_polytope}
    Let $d$ be odd and let $P\subset\RR^d$ be a centrally-symmetric polytope with $\drop(P)>0$.
    Let $H$ be a central hyperplane (\ie\ it is passing through the~symmetry center of $P$) with halfspaces $H_+$ and $H_-$.
    Then $$\drop(P\cap H_+)=\drop(P\cap H_-)>0.$$
\end{corollary}
\begin{proof}
    Set $P_\pm:=P\cap H_\pm$.
    Since $d$ is odd, \cref{res:central_inversion} yields $\Omega_0(P_-)=\Omega_0(-P_+)=\Omega_0(P_+)$. Using \cref{res:Minkowski_additivity} we obtain
    $$0=\Omega_0(P)=\Omega_0(P_+\cupdot P_-) \overset{\mathclap{\smash{\ref{res:Minkowski_additivity}}}}= \Omega_0(P_+) + \Omega_0(P_-) = 2\Omega_0(P_+),$$
    and analogously for $P_-$.
\end{proof}

\begin{figure}[h!]
    \centering
    \includegraphics[width=0.65\linewidth]{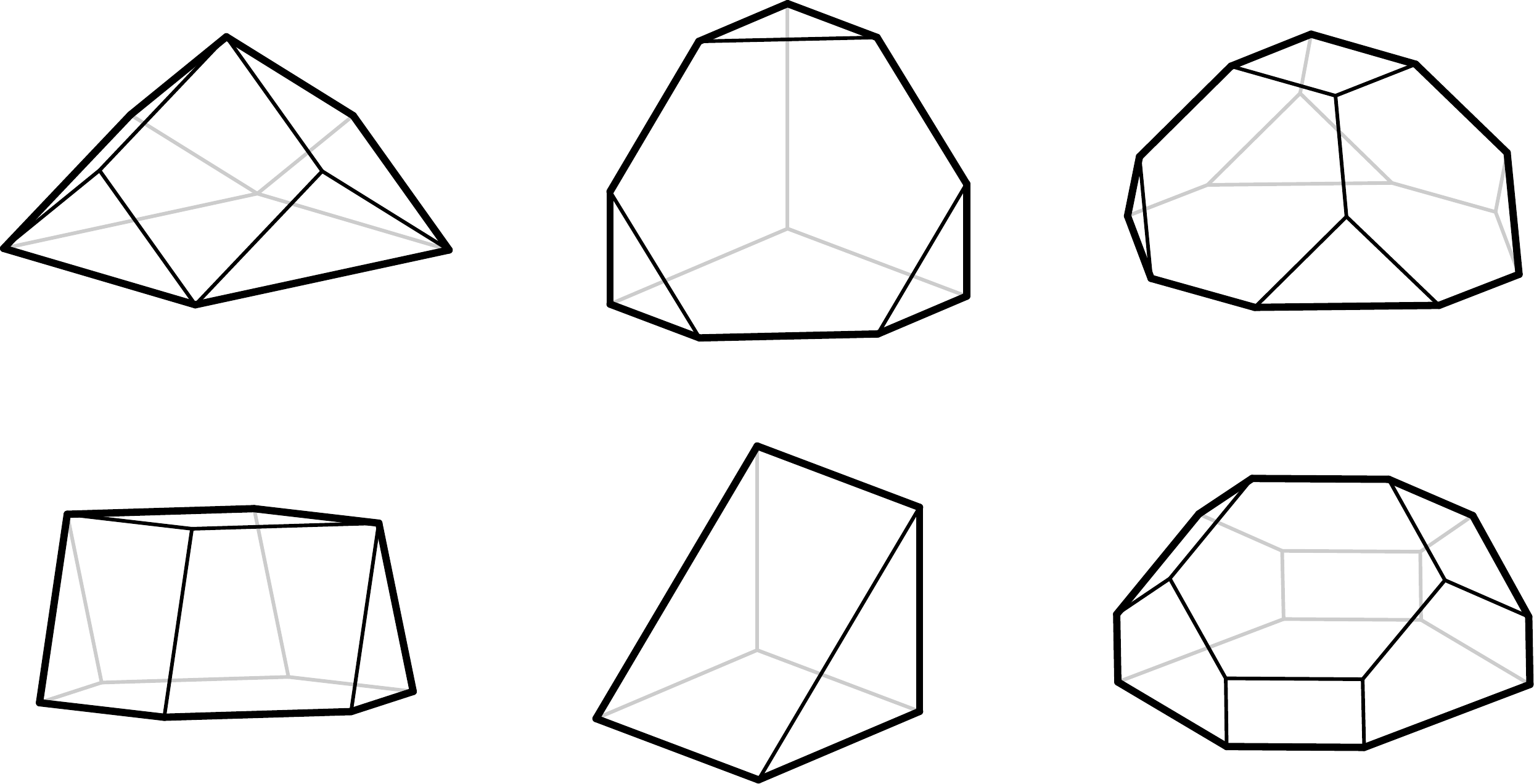}
    \caption{A selection of halves of zonotopes in dimension $d=3$, which have a non-zero drop by \cref{res:half_cs_polytope}. Left: halves of the rhombic~dode\-cahedron. Middle:\nls halves~of~the~cube.\nls Right:\nls halves~of~the~permutahe\-dron.}
    \label{fig:half_zonotopes}
\end{figure}

\begin{corollary}
    \label{res:Z++Z-}
    If $Z=Z_+\cupdot Z_-$ is a central dissection of a 3-dimensional zonotope into two halves (as in \cref{res:half_cs_polytope}), then $Z_++Z_-$ is a zonotope.
\end{corollary}
\begin{proof}
    $Z_++Z_-$ is centrally symmetric, hence $\drop(Z_++Z_-)\in\{0,2\}$ according to \cref{res:drop_properties} \ref{it:cs}. 
    Using \cref{res:half_cs_polytope} we have $\Omega_0(Z_++Z_-)=\Omega_0(Z_+)+\Omega_0(Z_-)=0$.
    Hence $\drop(Z_++Z_-)=2$, and the claim follows from \cref{res:zonotopes_iff_drop_d_1}.    
\end{proof}

There are other ways to arrive at \cref{res:Z++Z-}.
For example, one can show that $Z_++Z_- = Z + Z_0$, where $Z_0=Z\cap H$ is the section of $Z$ with the central dissecting hyperplane $H$.
Since for $d=3$ the section $Z_0$ is a centrally symmetric polygon, it is also a zonotope, and so $Z+Z_0$ must be a zonotope as well.
This also shows that an analogue of \cref{res:Z++Z-} in higher dimensions will fail: in $d\ge 4$ the central~section $Z_0$ of a zonotope needs not be a zonotope, but $Z_0$ will appear as a face of $Z_++Z_-$.

% The following lemma gives us the, so far, most diverse class of polytopes with a non-zero drop.

% \begin{lemma}
%     \label{res:Minkowski_sum_drop_bound}
%     Let $P$ be a centrally symmetric polytope of odd dimension with $\drop(P)>0$ (\eg\ a zonotope).
%     Let $H$ be a hyperplane through its center that dissects $P$ into two identical but opposite halves $P_+$ and $P_-$.
%     Then $\drop(P_+)>0$.
% \end{lemma}
% %
% \begin{proof}
%     \label{ex:half_zonotope}
%     % Let $P$ be a centrally symmetric polytope (centered at the origin) that has $\drop(P)>0$ (\eg\ a zonotope).
%     % Let $H$ be a hyperplane through its center. 
%     % Then $H$ dissects $P$ into two identical halves $P_+$ and $P_-$.
%     % We claim that if $d$ is odd, then $P_+$ has a drop.
%     Suppose that $\drop(P_+)=0\implies\Omega_0(P_+)\not=0$. 
%     Since we are in odd dimension, we have $\Omega_0(P_+)=\Omega_0(P_-)$.
%     Hence $\Omega_0(P)=\Omega_0(P_+)+\Omega_0(P_-) = 2\Omega(P_0)\not=0$, in contradiction to $P$ having a drop. 
% \end{proof}

We are now ready to prove a 3-dimensional analogue of \cref{res:characterize_drop_2D} that provides a geometric interpretation for each value of the degree drop:

\begin{theorem}
    \label{res:characterize_drop_3D}
    If $P$ is a 3-dimensional polytope, then
    $$
    \drop(P) = \begin{cases}
        0 & \text{if $P+(-P)$ is \ul{not} a zonotope,}
            \\
        1 & \text{if $P+(-P)$ is a zonotope, but $P$ itself is \ul{not},}
            \\
        2 & \text{if $P$ is a zonotope.}
    \end{cases}.
    $$
    % In dimension $d=3$ the following are equivalent:
    % \begin{myenumerate}
    %     \item $\drop(P) > 0$
    %     \item $P + (-P)$ is a zonotope
    %     \item for each facet $F$ of $P$ holds: if $G$ is the face opposite to $F$ (in terms of supporting hyperplanes) then $F + (-G)$ is centrally symmetric.
    % \end{myenumerate}
\end{theorem}
Note that $P+(-P)$ is a zonotope precisely if for each facet $F\subset P$ and opposite face $F'\subset P$ (\ie\ $F$ and $F'$ are touched by parallel hyperplanes), $F+(-F')$ is itself a zonotope. For $d=3$ the latter means checking whether $F+(-F')$ is centrally symmetric, which is easy to do. 
Hence, \cref{res:characterize_drop_3D} provides a practical geometric way to determine the drop in dimension three.
%With this characterization, checking whether a 3-polytope has a drop comes down to a simple investigation of its faces.
%
\begin{proof}[Proof of \cref{res:characterize_drop_3D}]
    Using Minkowski additivity and central inversion, we have
    $$
    \Omega(P+(-P)) \,\overset{\mathclap{\smash{\ref{res:Minkowski_additivity}}}}=\, 
    \Omega(P) + \Omega(-P) \,\overset{\mathclap{\smash{\ref{res:central_inversion}}}}=\, 
    2\Omega(P).
    $$
    That is, $\drop(P+(-P))=0$ if and only if $\drop(P)=0$.
    Note also that $P+(-P)$ is centrally symmetric, hence by \cref{res:drop_properties} \ref{it:cs} its drop is $0$ or $2$.
    So 
    \begin{align*}
        \drop(P)=0
        \;&\Longleftrightarrow\;
        \drop(P+(-P))=0
        \\\;&\Longleftrightarrow\;
        \drop(P+(-P))\not=2 
        \;\overset{\smash{\ref{res:zonotopes_iff_drop_d_1}}}{\Longleftrightarrow}\;
        \text{$P+(-P)$ is not a zonotope}.
    \end{align*}
    %$$\drop(P)=0$ $\Longleftrightarrow$ $\drop(P+(-P))=0$ $\Longleftrightarrow$ $\drop(P+(-P))\not=2$ $\Longleftrightarrow$ $P+(-P)$ is not a zonotope (
    %where we used \cref{res:zonotopes_iff_drop_d_1} for the last equivalence.
    The equivalence ``$\drop(P)=2$ $\Longleftrightarrow$ $P$ is a zonotope'' is precisely \cref{res:zonotopes_iff_drop_d_1}.
    The remaining case $\drop(P)=1$ is necessarily capturing what remains, that is, that $P$ is not zonotope, while $P+(-P)$ is a zonotope.
    %So $P+(-P)$ is a zonotope if and only if $\drop(P+(-P))>0$ if and only if $\drop(P)>0$.
\end{proof}

% {\color{lightgray}
% \begin{theorem}
%     If $F\subset P$ is a centrally symmetric facet, then $\<x,u_F\>$ divides $\adj_P^0$.
% \end{theorem}
% }

% \asays{What about Corollary: If $Z\subset\RR^3$ is a zonotope, then $Z_+ + (-Z_+)$ is a zonotope, for any central plane $H$ and $Z_+ = Z\cap H_+$. Is this known/easy to see/true in any dimension?} \msays{I notices this as well. First I also thought this is interesting. Then through visual inspection I convinced myself that $Z_++(-Z_+)=Z+Z_0$, where $Z_0=Z\cap H$ is the section (no formal proof so far). In 3D $Z_0$ is a cs polygon, hence a zonotope, and so the sum is a zonotope again. In 4D you can get the octahedron as a central section of the 4-cube, and so $Z_++(-Z_+)$ will not be a zonotope. What do you think, is it worth to include this discussion?} 

% \todo{\msays{add this as a Corollary.}}

\section{Valuations for higher degree drops}
\label{sec:higher_valuations}

Over the course of the last sections we established the reduced canonical~form $\Omega_0$ as a helpful tool in the study of degree drops. %, but also as an object of intrinsic interest.
Its major drawback is that~it~cannot quantify this drop. 
The goal of this section is to introduce and investigate a new family of valuations $\Omega_s,s\ge 0$, where $\Omega_s$ can testify a degree drop of $s$ or more. %, and study their properties. 
%We write $\mathcal P^d$ to denote the set of all polytopes in $\RR^d$.

From \eqref{eq:Omegahom} we see that the degree drop of $P$ is the order of vanishing of $\Omega(P;x_0,x)$ at a generic point of the form $(0,x)$. Indeed, since the denominator of $\Omega(P;x_0,x)$ is not divisible by $x_0$ we may expand $\Omega(P;x_0,x)$ as a Taylor series at $(0,x)$: % More precisely, we regard $\Omega(P;x_0,x)$ as a rational function in $x_0$ with coefficients in $\RR(x)$ and we write
\begin{equation}
\label{eq:omegasdef}
\Omega(P;x_0,x) = \sum_{s\geq 0} \Omega_s(P;x)x_0^s,
\end{equation}
which shall serve as a definition of $\Omega_s$ as a function mapping convex polytopes~to $\RR(x)$.
%\colon \mathcal P^d \to \RR(x)$. 
This notation is consistent with our previous definition of $\Omega_0$ as $\Omega(P;x_0,x)$ evaluated at $x_0=0$. 
As it turns out, all $\Omega_s$ together completely describe the degree drop of $P$. We obtain the following generalization of \cref{res:Omgea0_0_iff_drop_0}:

\begin{theorem}
    \label{prop:higher_drops}
    \label{res:higher_drops}
    %Let $P\subset\RR^d$ be a convex polytope. 
    %We have $\drop(P)\geq t$ if and only if $\Omega_s(P)=0$ holds for all $s<t$.
    $\Omega_s(P)=0$ for all $s\le s_0$ if and only if $\drop(P)> s_0$.
\end{theorem}
% \begin{theorem}
% \label{prop:higher_drops}
% \label{res:higher_drops}
%     $\Omega_s(P)=0$ if and only if $\drop(P)> s$.
% \end{theorem}

\begin{proof}
    We assume $s_0\ge0$. 
    We can write the homogenized adjoint as a polynomial~in $x_0$ with coefficients $a_k\in \RR[x]$. Then
    \begin{align}
        \label{eq:powerseries_adjoint}
        \begin{split}
            \sum_k a_k(x) x_0^k
            =\adj_P(x_0,x)
            &\,=\;\prod_F (h_Fx_0 - \langle x,u_F\rangle) \cdot \Omega(P;x_0,x)
            \\&\,\overset{\mathclap{\smash{\text{\eqref{eq:omegasdef}}}}}=\;\prod_F (h_Fx_0 - \langle x,u_F\rangle) \cdot \sum_{s\geq 0} \Omega_s(P;x)x_0^s  
        \end{split}
    \end{align}
    The proof now proceeds by comparing coefficients on both sides.
    
    Suppose first that $\drop(P) > s_0$. This is, by definition, equivalent to $a_k\equiv 0$, for all $0\leq k \le s_0$. 
    Hence, the corresponding coefficients of $x_0^s$ on the right hand side of \eqref{eq:powerseries_adjoint} must vanish as well. 
    For $s=0$, this coefficient is $\prod_F (-\langle x,u_F\rangle) \Omega_0(P)$,\nls that is, $\Omega_0(P) = 0$. 
    Proceeding by induction on $s$, we see that for each $s \le s_0$ the~coefficient is $\prod_F (-\langle x,u_F\rangle) \Omega_s(P)$.
    Thus, $\Omega_s(P) = 0$ whenever $0\leq s \le s_0$.

    Conversely, suppose that $\Omega_s(P)=0$ for all $0\leq s < s_0$.
    Then performing~the~previous argument backwards, the corresponding coefficients of $x_0^s$ on the right hand side of \eqref{eq:powerseries_adjoint} vanish, and we conclude $a_k \equiv 0$ for all $0\leq k \le s_0$.
    %\nls Thus~$\drop(P)>s_0$.
    %
    %
    % If $a_k\in\RR[x]$ is the $k$-homogeneous part of the~adjoint of $P$, then
    % %Let us write $\adj_P(x) = \sum_{k=0}^{m-d-1} a_k(x)$, \mbox{where~$a_k\in\RR[x]$} is the $k$-homogeneous part of the adjoint polynomial. Then
    % \[
    %     \adj_P(x_0,x) = x_0^{m-d-1} \,\sum_{k=0}^{\mathclap{m-d-1}} a_k(x/x_0) =\; \sum_{\ell=0}^{\mathclap{m-d-1}} a_{m-d-1-\ell}(x)x_0^\ell. 
    % \]
    % Multiplying both sides of \eqref{eq:omegasdef} with $\prod_F (h_Fx_0 - \langle u_F,x\rangle)$ we find
    % %
    % \begin{equation}
    % \label{eq:powerseries_adjoint}
    %     \sum_{\ell=0}^{\mathclap{m-d-1}} a_{m-d-1-\ell}(x)x_0^\ell = \prod_F (h_Fx_0 - \langle u_F,x\rangle) \sum_{s\geq 0} \Omega_s(P;x)x_0^s.
    % \end{equation}
    % %
    % Suppose first that $\drop(P) > t$. This is, by definition, equivalent to $a_{m-d-1-\ell}(x) = 0$, for $0\leq \ell \le t$. 
    % Hence, the corresponding coefficients of $x_0^s$ on the right hand side of \eqref{eq:powerseries_adjoint} must vanish as well. 
    % For $s=0$, this coefficient is $\prod_F (-\langle u_F,x\rangle) \Omega_0(P)$,\nls that is, $\Omega_0(P) = 0$. 
    % Proceeding by induction on $s$, we see that for $0\leq s \le t$ the coefficient is $\prod_F (-\langle u_F,x\rangle) \Omega_s(P)$.
    % Thus, $\Omega_s(P) = 0$ whenever $0\leq s \le t$.
    %
    % Conversely, suppose that $\Omega_s(P)=0$ for all $0\leq s < t$.
    % Then performing the~previous argument backwards, the corresponding coefficients of $x_0^s$ on the right hand side of \eqref{eq:powerseries_adjoint} vanish, and we conclude $a_{m-d-1-\ell}(x) = 0$ for all $0\leq \ell \le t$.
\end{proof}

Next, we observe that $\Omega_s$ is a valuation:
in view of \eqref{eq:omegasdef} we have \[ \Omega_s(P) = \frac{1}{s!}\left[  \frac{\partial^s}{\partial x_0^s}\Omega(P;x_0,x) \right]_{x_0=0}, \]
so $\Omega_s$ arises as the composition of the valuation $P\mapsto \Omega(P;x_0,x)$ with the linear map $f\mapsto \tfrac1{s!}[\tfrac{\partial^s}{\partial x_0^s} f]_{x_0=0}$ and thus inherits the valuation property. For the same reason, $\Omega_s$ is a \emph{simple} valuation, \ie\ it vanishes if $P$ has empty interior. 
It moreover exhibits the same linear equivariance as $\Omega_0$ (\cf\ \cref{res:Omega0_trafo}):

\begin{lemma}
    \label{lemma:equivariance_omegas}
    Let $P\subset\RR^d$ be a convex polytope and let $S\in\GL(\RR^d)$. We have
    \[
    \Omega_s(SP,Sx) = |\det(P)|^{-1}\,\Omega_s(P;x).
    \]
\end{lemma}

\begin{proof}
    The proof is analogous to \cref{res:Omega0_trafo}, except that in the last step we do not set $x_0=0$, but expand both sides at $x_0=0$ and compare coefficients.
    % %
    % Let $\overline{S}$ be the linear transformation $(x_0,x)\mapsto (x_0,Sx)$ on $\RR^{d+1}$. We have $\det\overline S=\det S$ and $(SP)^{\hom} = \overline S P^{\hom}$ and so we obtain from \eqref{eq:Omegahom_geom} and \cref{res:Omega_trafo} that
    % \[
    % \Omega(SP;x_0,Sx) = |\det(S)|^{-1}\, \Omega(P;x_0,x).
    % \]
    % Expanding both sides at $x_0=0$ and comparing coefficients yields the claim.
\end{proof}

\begin{lemma}
    \label{lemma:hom_rational_functions}
    For any convex polytope $P\subset\RR^d$, the rational function $x\mapsto \Omega_s(P;x)$ is homogeneous of degree $-(d+1+s)$.
\end{lemma}

\begin{proof}
    Using the geometric definition of $\Omega(P;x_0,x)$ from \eqref{eq:Omegahom_geom} together with the linear equivariance  \cref{res:Omega_trafo}, we find
    \[\begin{split}
        \smash{\sum_{s\geq 0}} \Omega_s(P;\lambda x) x_0^s &= \Omega(P^{\hom} ; (x_0,\lambda x)) = \lambda^{-d-1}\,\Omega(\lambda^{-1} P^{\hom}; (\lambda^{-1}x_0,x)) \\ 
        &=  \lambda^{-d-1} \Omega(P^{\hom} ; (\lambda^{-1}x_0,x)) = \lambda^{-d-1} \smash{\sum_{s\geq 0}} \Omega_s(P;x) \lambda^{-s} x_0^s.
    \end{split}\]
    Comparing coefficients concludes the proof.
\end{proof}

This provides a generalization of \cref{res:1_homogeneous}:

\begin{corollary}
\label{thm:omegas_homogogeneity}
    $\Omega_s$ is $(s+1)$-homogeneous (in the sense of \cref{sec:homogeneous}).
    %The valuation $\Omega_s$ is $(s+1)$-homogeneous, \ie\ we have $\Omega_s(\lambda P) = \lambda^{s+1}\Omega_s(P)$ for all convex polytopes $P\subset\RR^d$ and $\lambda \geq 0$.
\end{corollary}

\begin{proof}
    Combining \cref{lemma:equivariance_omegas,lemma:hom_rational_functions} we obtain
    \[
        \Omega_s(\lambda P;x) = \lambda^{-d}\,\Omega_s(P;\lambda^{-1}x) = \lambda^{-d}\lambda^{d+1+s}\,\Omega_s(P;x). \qedhere
    \]
\end{proof}

A shortcoming of the $\Omega_s,s>0$ as compared to $\Omega_0$ is that they are not translation-invariant.
Nevertheless some statements about their behavior under translation~can be made.

%To conclude the section, we describe the behaviour of $\Omega_s$ with respect to translation.

\begin{theorem}
\label{thm:transl_polynomial}
    For any convex polytope $P\subset \RR^d$, the function
    \[
    \RR^d \to \RR(x),\quad t\mapsto \Omega_s(P+t;x)
    \]
    is a polynomial of degree at most $s$ with coefficients in $\RR(x)$.
\end{theorem}

\begin{proof}
    It suffices to show the statement for a simplex $\Delta = \{ x\in\RR^d \colon \langle x,u_i\rangle \leq h_i\}$, $u_i\in\RR^d\setminus\{0\}$, $h_i\in\RR$, and then use the valuation property. In this case, the adjoint of $\Delta$ is a constant. We have
    \[\begin{split}
        \Omega(\Delta;x_0,x) 
        %&= \Omega(\Delta^{\hom};(x_0,x)) 
        &= \frac{\adj_\Delta}{\prod_i (h_i x_0 - \langle x,u_i\rangle)} 
        = \frac{\adj_\Delta}{\prod_i (-\langle x,u_i\rangle)}\cdot \frac{1}{\prod_i \big(1- \tfrac{h_i}{\langle x,u_i\rangle}x_0\big)}
        \\&=\Omega_0(\Delta;x) \prod_{i=0}^d \sum_{s\geq 0}\left(\frac{h_i}{\langle x,u_i\rangle}\right)^s\!x_0^s
        \\&= \Omega_0(\Delta;x) \sum_{s\geq 0}\left[ \sum_{s_0+\cdots +s_d=s} \,\prod_{i=0}^d \left(\frac{h_i}{\langle x,u_i\rangle}\right)^{s_i} \right]\! x_0^s.
    \end{split}\]
    Comparing coefficients yields,
    \[\Omega_s(\Delta;x) = \Omega_0(\Delta;x) \!\!\!\!\!\!\! \sum_{s_0+\cdots +s_d=s} \,\prod_{i=0}^d \left(\frac{h_i}{\langle x,u_i\rangle}\right)^{s_i}\!\!\!.\]
    Since $\Delta+t = \{ x\in\RR^d \colon \langle x,u_i\rangle \leq h_i + \langle u_i,t\rangle\}$ we obtain
    \[
        \Omega_s(\Delta+t;x) = \!\!\!\!\!\!\!\sum_{s_0+\cdots +s_d=s} \,\prod_{i=0}^d \left(\frac{h_i + \langle t,u_i\rangle}{\langle x,u_i\rangle}\right)^{s_i}
    \]
    which is a polynomial in $t$.
\end{proof}

%\notodo{\msays{Better move the next paragraph to after the proof of \cref{thm:transl_polynomial}?}}%
%
Valuations that exhibit a polynomial behavior under translation have been studied in great detail by Khovanski\u{i} and Pukhlikov \cite{khovanskiipukhlikov}, see also \cite{mcmullen_euler}. They appear~naturally in the context of tensor valuations \cite{alesker, bergjochemkosilverstein, kiderlen, ludwigsilverstein}.

Finally, we show that $\Omega_s$ exhibits translation-invariance at least on polytopes of a sufficiently high degree drop:

\begin{theorem}
    \label{thm:omegas_transl_inv}
    \label{res:omegas_transl_inv}
    %Let $P\subset\RR^d$ be a convex polytope for which $\Omega_s(P)=0$ for all $s<s_0$. Then, $\Omega_{s_0}(P+t) =\Omega_{s_0}(P)$ for all $t\in\RR^d$.
    If $\drop(P)\ge s$, then $\Omega_s(P+t)=\Omega_s(P)$ for all $t\in\RR^d$.
\end{theorem}

\begin{proof}
    For $t\in\RR^d$, let $A_t \in\GL(\RR^{d+1})$ be the linear transformation given~by~$(x_0,x)\mapsto (x_0,x+x_0t)$. Clearly $\det A_t = 1$ and $A_t^{-1} = A_{-t}$. Moreover, we~have $(P+t)^{\hom} = A_tP^{\hom}$. 
    Using the geometric definition \eqref{eq:Omegahom_geom} of $\Omega(P;x_0,x)$, as~well~as~\cref{res:Omega_trafo}, we obtain
    %
    % \begin{align*}
    %     \smash{\sum_{s\geq 0}} \Omega_s(P+t;x)x_0^s 
    %       &= \Omega(P+t;x_0,x)
    %     \\&= \Omega(T;x_0,x-x_0t) 
    %        = \sum_{s\geq 0}\Omega_s(P;x-x_0t)x_0^s.        
    % \end{align*}
    %
    \[\begin{split}
        \sum_{r\geq 0} \Omega_r(P+t;x)x_0^r &= \Omega((P+t)^{\hom}; (x_0,x))
        \\[-2.3ex]
        &= \Omega(P^{\hom}; (x_0,x-x_0t)) = \sum_{s\geq 0}\Omega_r(P;x-x_0t)x_0^r.
    \end{split}\]
    Since $\drop(P)\ge s$, \cref{res:higher_drops} yields $\Omega_{r}(P)=0$ for all $r<s$.
    We can then divide both sides of the above identity by $x_0^{s}$ to arrive at %obtain the following identity of power series:
    \begin{equation*}
    \label{eq:fps_shenanigans}
        \sum_{r\geq s}\Omega_r(P+t;x)x_0^{r-s} = \sum_{r\geq s}\Omega_r(P;x-x_0t)x_0^{r-s}.
    \end{equation*}
    Setting $x_0=0$ yields the claim.
    %Comparing coefficients of the constant terms yields $\Omega_s(P+t;x)=\Omega_s(P;x-x_0t)$.
    %
    % If $x\in\RR^d$ is generic in the sense that $\Omega_s(P;x)$ exists as a real number, then the function $x_0\mapsto \Omega_s(P;x-x_0t)$ can be written as a Taylor series around $x_0=0$ whose constant term is $\Omega_s(P;x)$. After expanding all terms in \eqref{eq:fps_shenanigans} of the form $\Omega_s(P;x-x_0t)$ in this manner and comparing the constant term, we obtain $\Omega_s(P+t,x) = \Omega_s(P)$ as desired. 
\end{proof}

We close this section with an application to polytope decompositions.
Recall~that by \cref{res:drop_tiling} we have $\drop(P_1\cupdot\cdots\cupdot P_n)\ge \min_i\drop(P_i)$.
We can now make a statement about the equality case: if for given pieces $P_i\subset\RR^d$ it is possible to built at least one polytope whose drop attains the lower bound $\min_i\drop(P_i)$, then all polytopes built from these pieces will attain the bound.

\begin{lemma}
    Given $P=P_1\cupdot\cdots\cupdot P_n$ and $Q=(P_1+t_1)\cupdot\cdots\cupdot(P_n+t_n)$, then
    $\drop(P)=\min_i\drop(P_i)$ if and only if $\drop(Q)=\min_i\drop(P_i)$.
\end{lemma}
\begin{proof}
    Assume $s:=\drop(P)=\min_i\drop(P_i)$.
    By \cref{res:higher_drops} we have $\Omega_s(P)\not=0$, and by \cref{res:omegas_transl_inv} we have that $\Omega_s$ is translation-invariant on the $P_i$.
    Hence
    \begin{align*}
        \Omega_s(Q) 
        &= \Omega_s((P_1+t_1)\cupdot\cdots\cupdot(P_n+t_n))
        %&= \Omega_s(Q_1\cupdot\cdots\cupdot Q_n)
        %\\&= \Omega_s(Q_1)+\cdots+\Omega_s(Q_n)
        \\&= \Omega_s(P_1+t_1)+\cdots+\Omega_s(P_n+t_n)
        \\&= \Omega_s(P_1)+\cdots+\Omega_s(P_n)
        = \Omega_s(P_1\cupdot\cdots\cupdot P_n)
        = \Omega_s(P) \not=0
    \end{align*}
    By \cref{res:higher_drops} we can conclude $\drop(Q)\le s$.
    But by \cref{res:drop_tiling} we also have $\drop(Q)\ge \min_i\drop(P_i)\ge s$, and hence $\drop(Q)=s$.
\end{proof}

One can read this as follows: the drop is a translation scissors invariant as long~as we only decompose into pieces of sufficiently high drop.

%% file: sec/appendix.tex
\appendix

\section{Computing $\Omega$ for simplices and simplicial cones}
\label{sec:appendix_simplices}

In this section we explicitly compute the formulae for the canonical forms of~simplicial cones and simplices as presented in \cref{ex:simplices}. 
%Here, we are working with the definition of $\Omega$ as the polar volume function.

\begin{proposition}
\label{prop:omega_cone}
    Let $C := \{y\in\RR^d \mid \langle u_i, y \rangle \leq 0\}$ be a simplicial cone with facet normals $u_1,...,u_d\in\RR^d\setminus\{0\}$. Its canonical form is given by 
    \[
        \Omega(C;x) = (-1)^d \frac{|\det(u_1,...,u_d)|}
        {\prod_i\langle u_i,x\rangle}.
        %{\prod_{i=1}^d\langle u_i,x\rangle}.
    \]
\end{proposition}

\begin{proof}
    Let $x\in \Int(P)$, then $(C-x)^\circ$ is bounded. In this case we have
    \[
        C-x 
        = \big\{y\in\RR^d \mid \langle u_i,y\rangle \leq -\langle u_i,x\rangle\big\} 
        = \left\{ y\in\RR^d \,\middle\vert\, \left\langle\frac{u_i}{-\langle u_i,x\rangle}, y \right\rangle\leq 1\right\}.
    \]
    Its dual $(C-x)^\circ = \conv(\{0\}\cup\{u_i/(-\langle u_i,x\rangle) \mid 1\le i\le d\})$ is a simplex, the volume of which we can express via a determinant:
    \[
        \Omega(C;x) 
        = d! \vol(C-x)^\circ 
        = \left|\,\det\! \left(\frac{u_i}{-\langle u_i,x\rangle} \,\middle\vert\, 1\le i\le d\right) \right| 
        = \frac{|\det(u_1,...,u_d)|}{\prod_i (-\langle u_i,x\rangle)}.
    \]
    Note that in the last step we once again used $x\in\Int(C)$, which guarantees that~the factors $-\langle u_i,x\rangle$ are positive and can be pulled out of the absolute value. % preserving the sign.
\end{proof}

\begin{proposition}
\label{prop:omega_simplex}
    Let $\Delta := \{ y\in\RR^d \mid \langle u_i,x\rangle \leq h_i\}$ be a simplex with facet normals $u_0,...,u_d\in\RR^d\setminus\{0\}$ and facet heights $h_0,...,h_d\in\RR$. Its canonical form is given by
    \[
        \Omega(\Delta; x) = \frac{\left | \det\! \left( \!\begin{pmatrix}
            u_i\\
            h_i
        \end{pmatrix} \,\middle|\, 0\leq i\leq d\right)\right|}{\prod_i (h_i - \langle u_i,x\rangle) }.
    \]
\end{proposition}

\begin{proof}
    Let $x\in\Int(\Delta)$. 
    Similar to the proof of \cref{prop:omega_cone} we have 
    \[
    (\Delta-x)^\circ 
    %= \big\{y\in\RR^d \mid \langle u_i,y\rangle \leq h_i - \langle u_i,x\rangle\big\} 
    = \conv \left\{ \frac{u_i}{h_i-\langle u_i,x\rangle} \,\middle|\, 0\leq i\leq d \right\}.
    \]
    A convenient way to compute the volume of the dual simplex is by computing the volume of a pyramid of height 1 with $(\Delta-x)^\circ$ as a base. 
    More precisely,\nls let $\overline{T} := \conv(\{0\} \cup (\{1\}\times (\Delta-x)^\circ)\subset\RR^{d+1}$. Then, $\vol\overline T = \tfrac{1}{d+1}\vol(\Delta-x)^\circ$ and~we~obtain
    \[\begin{split}
        \Omega(\Delta;x) 
        &= 
        d!\vol(\Delta-x)^\circ 
        = 
        (d+1)! \vol(\overline T) 
        = \left|\det\!\left(\!\begin{pmatrix}
            \tfrac{u_i}{h_i-\langle u_i,x\rangle}
            \\
            1
        \end{pmatrix} \,\middle|\, 0\leq i\leq d \right)\right|\\
        &=
        \frac{\left|\det\!\left(\! \begin{pmatrix}
            u_i
            \\
            h_i-\langle u_i,x\rangle 
        \end{pmatrix} \,\middle|\, 0\leq i\leq d \right)\right|}{\prod_i (h_i - \langle u_i,x\rangle)} = \frac{\left|\det\!\left(\! \begin{pmatrix}
            u_i
            \\
            h_i 
        \end{pmatrix} \,\middle|\, 0\leq i\leq d \right)\right|}{\prod_i (h_i - \langle u_i,x\rangle)}.
    \end{split}\]
    The last equality is obtained by a row operation: 
    if the first $d$ rows of the matrix are labeled $r_1,...,r_d$, we obtain the final expression by adding $x_ir_i$ to the last row.
\end{proof}

%\newpage

\section{Restricting the canonical form to facets}
\label{sec:appendix_facet_restriction}

We prove the correctness of the facet restriction formula \eqref{eq:Omega_on_face} under the assumption that the normal vector $u_F$ is normalized to length 1: if $F\subset P$ is a facet~and $x\in\aff(F)$, then
$$
    \Omega^{d-1}(F;x) 
    %= \operatorname{Res}_F(\Omega^d(P)) 
    = 
    %\|u_F\| 
    \frac{\adj_P(x)}{\prod_{\substack{G\subset P\\G\not=F}} L_G(x)}
    =
    \lim_{\mathclap{\tilde x\to x}} \big(L_F(\tilde x)\,\Omega^d(P;\tilde x)\big)
    .    
$$

W.l.o.g.\ we may assume $\aff(F)=\RR^{d-1}\times\{0\}\subset\RR^d$. 
We write $\hat F$ if we consider $F$ as a full-dimensional polytope in $\aff(F)\simeq\RR^{d-1}$.
Fix~$\hat x\in\Int(\hat F)$ and set $x_t:=(\hat x,t)$. 
We may assume that $t>0$ and small enough, so that $x_t\in\Int(P)$.\nls
Since~$u_F$~is~normalized, $L_F(x):=h_F-\<x,u_F\>$ measures the Euclidean distance of $x$ from $\aff(F)$.
In particular, $L_F(x_t)=t$.
If $S_t$ denotes the diagonal matrix with diagonal entries $(1,...,1,t)$, that is, $S_t$ represents non-uniform scaling in the last coordinate, then $|\det(S_t)|=t$ and $S_t^{-1}=S_{1/t}$.
If $x:=x_0=(\hat x,0)$, using \cref{res:Omega_trafo} we can rewrite
\begin{align*}
    \frac{\adj_P(x)}{\smash{\prod_{\substack{G\subset P\\G\not=F}}} L_G(x)}
      &=\, \lim_{t\to0} \big( L_F(x_t) \Omega^d(P;x_t) \big)
       = \lim_{t\to0} \big( |\det(S_t)| \,\Omega^d(P;x) \big)
    \\&%
       \overset{\smash{\mathclap{\text{\ref{res:Omega_trafo}}}}}%
       =\, \lim_{t\to0} \Omega^d(S_{1/t} P; S_{1/t}x)
       = \lim_{t\to0} \Omega^d(S_{1/t} P; (\hat x,1))
\end{align*}

The effect of $S_{1/t}$ on $P$ as $t\to 0$ is that $P$ becomes increasingly stretched~along~the last coordinate axis, while $F$ and $(\hat x,1)$ stay put (see \cref{fig:stretching_polytope}). %The point $(\hat x,1)$ stays put as well.
\begin{figure}[h!]
    \centering
    \includegraphics[width=0.9\linewidth]{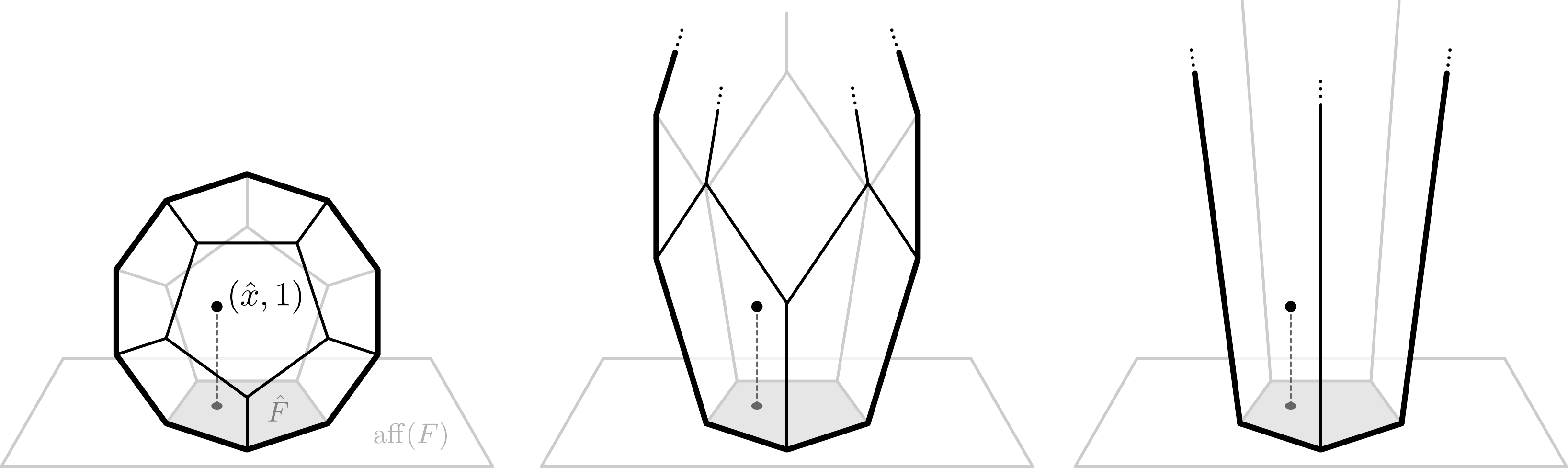}
    \caption{$S_{1/t}P$ for decreasing values of $t$.}
    \label{fig:stretching_polytope}
\end{figure}

In the limit $t\to 0$, $S_{1/t}P$ approaches a half-infinite cylinder over $\hat F$.
Since $(\hat x,1)$ has distance 1 from the base face, the dual of $S_{1/t}P$, polarized at $(\hat x,1)$, approaches a pyramid of height one over $\smash{(\hat F-\hat x)^\circ}$ (see \cref{fig:stretching_polytope_lim_polar}).
\begin{figure}[h!]
    \centering
    \includegraphics[width=0.6\linewidth]{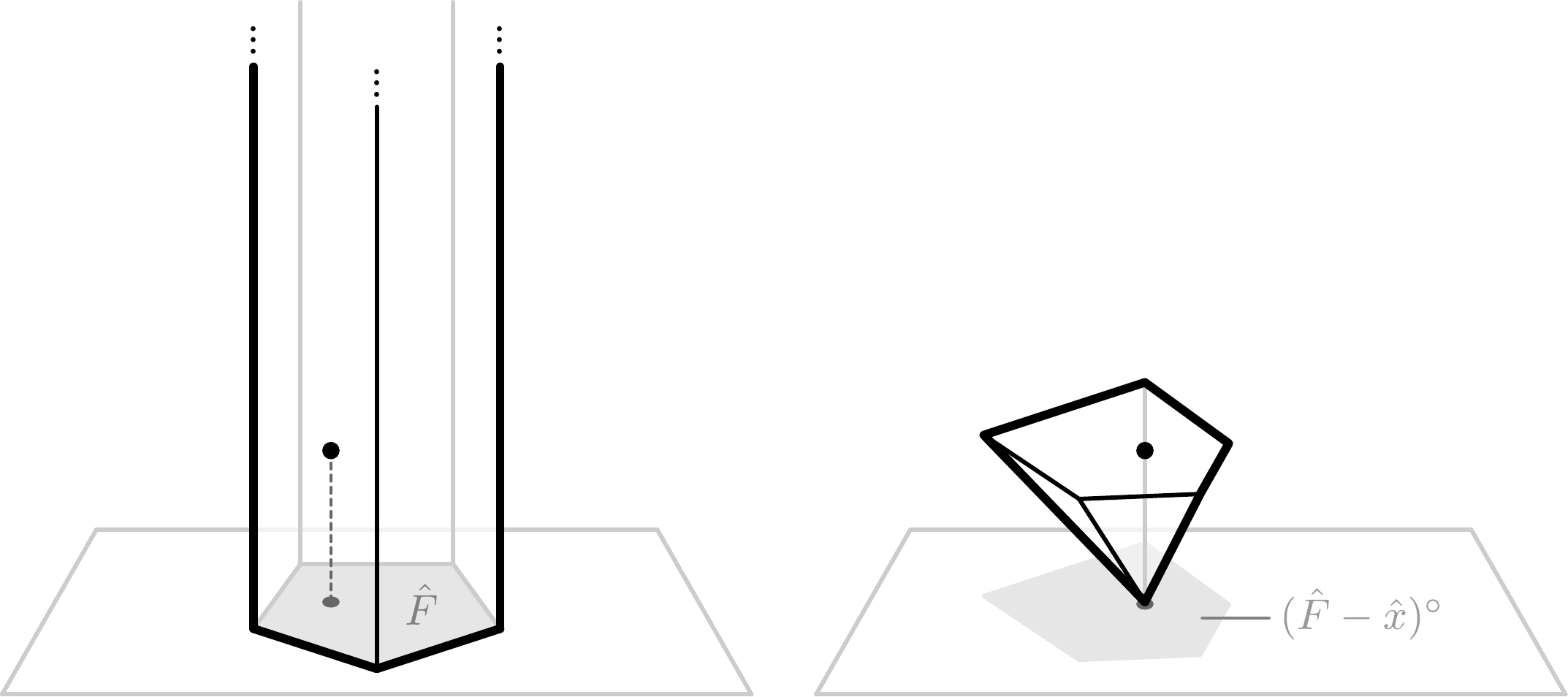}
    \caption{Left: limit of $S_{1/t}P$ as $t\to0$. Right: dual of the limit.}
    \label{fig:stretching_polytope_lim_polar}
\end{figure}

\noindent
The proof now concludes as follows:
\begin{align*}
    \frac{\adj_P(x)}{\smash{\prod_{\substack{G\subset P\\G\not=F}}} L_G(x)}
         \,&\overset{\smash{\mathclap{\text{\eqref{eq:omegasdef}}}}}=\, \lim_{t\to0} \big(d!\vol(S_{1/t} P-(\hat x,1))^\circ\big)
       \\&=\, d!\cdot \underbrace{\tfrac1d\vol(\hat F-\hat x)^\circ}_{\mathclap{\text{volume of pyramid over $(\hat F-\hat x)^\circ$ of height 1}}}
          = \Omega^{d-1}(\hat F,\hat x)
          = \Omega^{d-1}(F,x).
\end{align*}

\section{Adjoint degrees for polytope products}
\label{sec:appendix_adjoint_product}

The following is needed in the proof of \cref{res:drop_properties} \ref{it:product}: given polytopes $P_1\subset \RR^{d_1}$ and $P_2\subset\RR^{d_2}$, it holds
$$\deg(\adj_{P_1\times P_2})=\deg(\adj_{P_1})+\deg(\adj_{P_2}).$$
To show this, we first compute the canonical form of $P_1\times P_2$.
Since this \mbox{involves~com}\-puting the volume of the dual of $P_1\times P_2$, we recall the following facts:
$$(P_1\times P_2)^\circ=P_1^\circ\oplus P_2^\circ:=\conv\big((P_1^\circ\times\{0\}^{d_2})\cup (\{0\}^{d_1}\times P_2^\circ)\big)\subset\RR^{d_1+d_2}.$$
where $\oplus$ denotes the \emph{direct sum} of convex polytopes. It is also well-known that
$$\vol(P_1\oplus P_2)=\frac{d_1!d_2!}{(d_1+d_2)!}\cdot\vol(P_1)\vol(P_2).$$
If $(x_1,x_2)\in\RR^{d_1+d_2}$, then
%
% \begin{lemma}
%     Given polytopes $P_1\subset\RR^{d_1}$ and $P_2\subset\RR^{d_2}$, we have
%     %
%     $$\Omega(P_1\times P_2;(x_1,x_2))=\frac{d_1!d_2!}{(d_1+d_2)!}\Omega(P_1;x)\Omega(P_2;x_2).$$
% \end{lemma}
% %
% \begin{proof}
    \begin{align*}
        \Omega(P_1\times P_2;(x_1,x_2))
        &=
        (d_1+d_2)! \cdot \vol((P_1\times P_2)-(x_1,x_2))^\circ
        \\ &=
        (d_1+d_2)! \cdot\vol((P_1-x_1)\times(P_2-x_2))^\circ
        \\&=
        (d_1+d_2)! \cdot\vol((P_1-x_1)^\circ\oplus(P_2-x_2)^\circ)
        \\&=
        d_1!d_2! \cdot \vol(P_1-x_1)^\circ \vol(P_2-x_2)^\circ
        \\&=
        \Omega(P_1;x_1)\,\Omega(P_2;x_2).
        % \\&=
        % \frac{d_1!d_2!}{(d_1+d_2)!} \vol(P_1-1)^\circ \vol(P_2-x_2)^\circ
        % \\&=
        % \frac{d_1!d_2!}{(d_1+d_2)!} \Omega(P_1;x_1)\Omega(P_2;x_2).
    \end{align*}
% \end{proof}

Recall that if $F_1\subset P_1$ is a facet of $P_1$ with linear form $L_{F_1}(x_1):= h_{F_1}-\<x_1,u_{F_1}\>$, then $F:=F_1\times P_2$ is a facet of $P_1\times P_2$ which has the essentially same linear form $L_F(x_1,x_2):=h_{F_1}-\<(x_1,0),(u_{F_1},0)\>=L_{F_1}(x_1)$. Analogous statements hold if we start from a facet $F_2\subset P_2$.
Moreover, all facets of $P_1\times P_2$ are of this form.
Hence
\begin{align*}
\frac{\adj_{P_1\times P_2}(x_1,x_2)}{\prod_F L_F(x_1,x_2)} 
&=
\Omega(P_1\times P_2;(x_1,x_2))
=
\Omega(P_1;x_1)\,\Omega(P_2;x_2)
\\[-0.5ex]&=
\frac{\adj_{P_1}(x_1)}{\prod_{F_1}L_{F_1}(x_1)}\frac{\adj_{P_2}(x_2)}{\prod_{F_2}L_{F_2}(x_2)}
\\[1ex]&=
\frac{\adj_{P_1}(x_1)\adj_{P_2}(x_2)}{\prod_{F_1}L_{F_1\times P_2}(x_1,x_2)\prod_{F_2}L_{P_1\times F_2}(x_1,x_2)}
\\[1ex]&=
\frac{\adj_{P_1}(x_1)\adj_{P_2}(x_2)}{\prod_F L_F(x_1,x_2)}.
\end{align*}
Comparing numerators, we indeed obtain $\adj_{P_1\times P_2}(x_1,x_2)=\adj_{P_1}(x_1)\adj_{P_2}(x_2)$, from which follows the additivity of degrees.

\section{Equivalence of the definitions of homogenized canonical form}
\label{sec:appendix_homogenized_Omega}

%\notodo{\asays{Shall we move this Lemma to Appendix A, or at least swap Appendix B and C for the sake of the storyline?}\msays{Absolutely. I swapped the sections. I would prefer not to merge them because they are referenced from distinct parts of the main text.}}

We prove that the two definitions of homogenized canonical $\Omega(P;x_0,x)$ given in \cref{sec:homogenized_Omega} are equivalent.

\begin{lemma}
    \label{res:homogenized_Omega}
    $\Omega(P^{\hom};(x_0,x))=x_0^{-d-1} \Omega(P;x/x_0).$
\end{lemma}
\begin{proof}
    Since both sides of the equation are valuations, by \cref{res:Omega_valuation} it suffices~to triangulate $P$ and prove the statement for simplices.
    Given a simplex $\Delta:=\{x\in\RR^d$ $\mid \<x,u_i\> \le h_i\}$, its homogenization is a simplicial cone
    $$\Delta^{\hom}=\{(x_0,x)\in\RR^{d+1}\mid \<(x_0,x),(-h_i,u_i)\>\le 0\}$$ 
    Using the explicit expressions given in \cref{ex:simplices} (see also \cref{sec:appendix_simplices}), we verify
    \begin{align*}
        \Omega(\Delta^{\hom};(x_0,x)) 
        &= \frac{
            %\big|\det\!\big((u_i,-h_i)\T\mid i\in\{0,...,d\}\big)\big|
            \left|\,\det\!\begin{pmatrix}
                -h_0 & \ndots & -h_d
                \\[0.3ex]
                | & & | \\[-0.5ex]
                u_0 & \ndots & u_d \\
                | & & |
            \end{pmatrix}\right|
        }{\prod_i (h_i x_0 - \<x,u_i\>)} 
        = \frac{
            %\big|\det\!\big((u_i,h_i)\T\mid i\in\{0,...,d\}\big)\big|
            \left|\,\det\!\begin{pmatrix}
                | & & | \\[-0.5ex]
                u_0 & \ndots & u_d \\
                | & & | \\[1ex]
                h_0 & \ndots & h_d
            \end{pmatrix}\right|
        }{x_0^{d+1}\prod_i (h_i - \<x/x_0,u_i\>)}
        \\[1ex]
        &= x_0^{-d-1} \Omega(\Delta;x/x_0).
        \qedhere
    \end{align*}
\end{proof}

\section{Computing the adjoint of edge tangent cones}
\label{sec:appendix_lij_adj_Tij}

For the notation and setting, see \cref{sec:decomposition_on_simplices}.
We prove the following:

\begin{lemma}
    \label{res:lij_edj_Tij}
    If $u_0,...,u_d$ is a positive affine basis, and if $i<j$, then
    \begin{align}
    %\label{eq:Dij_def}
    \ell_{ij}\adj_{T_{ij}} = 
    -\det\!\begin{pmatrix}
        | &  & | &  & | & & |
        \\[-0.5ex]
        u_0 & \ndots & v_i & \ndots & v_j & \ndots & u_d
        \\
        | &  & | &  & | & & |
        \\[0.5ex]
        0 & \ndots & 1 & \ndots & 1 & \ndots & 0
    \end{pmatrix}.
    \end{align}
\end{lemma}
\begin{proof}
    We use
    \begin{align*}
        U:=\begin{pmatrix}
            | &  & |
            \\[-0.5ex]
            u_0 & \ndots & u_d
            \\
            | &  & |
        \end{pmatrix}
        \!,\quad\!\!
        A:=\begin{pmatrix}
            | & & | \\[-0.5ex]
            u_0 & \ndots & u_d \\
            | & & | \\[0.5ex]
            -h_0 & \ndots & -h_d
        \end{pmatrix}
        \!,\quad\!\!
        D_{ij}:=\begin{pmatrix}
            | &  & | &  & | & & |
            \\[-0.5ex]
            u_0 & \ndots & v_i & \ndots & v_j & \ndots & u_d
            \\
            | &  & | &  & | & & |
            \\[0.5ex]
            0 & \ndots & 1 & \ndots & 1 & \ndots & 0
        \end{pmatrix}
        \!.
    \end{align*}
    We shall assume a translation of $\Delta$ so that $h_i>0$. %, and so the last row of $A$ is negative.
    In particular, since $u_0,...,u_d$~is~a positive affine basis, $\det(A)=-\adj_{\Delta}<0$.
    In order to prove $\ell_{ij}\smash{\adj_{T_{ij}}}=-\det(D_{ij})$ we shall show in two steps that the terms~agree in both sign and absolute value.
    
    Since $\ell_{ij}\adj_{T_{ij}}>0$, for the correct sign we need to show that $\det(D_{ij})<0$.\nls
    Since $\det(A)<0$, it suffices to verify that $\det(AD_{ij})=\det(A)\det(D_{ij})>0$.
    We inspect the entries of $AD_{ij}$:
    $$
    (AD_{ij})_{k\ell}=\begin{cases}
        \<u_k,u_\ell\> & \text{if $\ell\not\in\{i,j\}$}
        \\
        \<u_k,v_\ell\> - h_k = 0 & \text{if $\ell\in\{i,j\}$ and $k\not=\ell$}
        \\
        \<u_k,v_k\> - h_k =:a_k < 0 & \text{if $\ell\in\{i,j\}$ and $k=\ell$}
    \end{cases}
    $$
    where we used the definition of $h_i:=\<u_i,v_j\>$ for any $j\not=i$.
    In other words, $AD_{ij}$ is the Gram matrix $U\T U$, except that the columns $i$ and $j$ are replaced with $a_ie_i$ and $a_j e_j$ respectively ($e_k$ being the $k$-th standard unit vector).
    To compute $\det(AD_{ij})$, we first develop by the $i$-th, and then by the $j$-th column, each of which has only~a single non-zero (actually, negative) entry.
    If $U_{(ij)}$ denotes the sub-matrix~of~$U$~ob\-tained by deleting column $i$ and $j$, then
    $$\det(AD_{ij})=\underbrace{a_ia_j}_{>0}\cdot \underbrace{\det(U_{(ij)}\T U_{(ij)})}_{>0}>0,$$
    where we used that Gram matrices have positive determinants.

    To see that the absolute values agree, consider the following computation, where $\equiv$ means equality up to sign:
    \begin{align}
    %\qquad\qquad\!\!
    % \det\!\begin{pmatrix}
    %     | &  & | &  & | & & |
    %     \\[-0.5ex]
    %     u_0 & \ndots & v_i & \ndots & v_j & \ndots & u_d
    %     \\
    %     | &  & | &  & | & & |
    %     \\[1ex]
    %     0 & \ndots & 1 & \ndots & 1 & \ndots & 0
    % \end{pmatrix}
    \det(D_{ij})
    &\equiv\notag
    \det\!\begin{pmatrix}
        & | &  & | & & | & |
        \\%[-0.5ex]
        %u_0 & \cdots & u_d & v_j-v_i
        \bdots & \hat u_i & \ndots & \hat u_j & \edots & v_i & v_j
        \\
        & | &  & | & & | & |
        \\[1ex]
        \bdots & 0 & \ndots & 0 & \ndots & 1 & 1
    \end{pmatrix}
    %\\&
    =\notag
    \det\!\begin{pmatrix}
        & | &  & | & & | & |
        \\%[-0.5ex]
        %u_0 & \cdots & u_d & v_j-v_i
        \bdots & \hat u_i & \ndots & \hat u_j & \edots & v_i & v_j - v_i
        \\
        & | &  & | & & | & |
        \\[1ex]
        \bdots & 0 & \ndots & 0 & \ndots & 1 & 0
    \end{pmatrix}
    \\&\equiv \label{eq:Dij_form}
    \det\!\begin{pmatrix}
        & | &  & | & & |
        \\%[-0.5ex]
        %u_0 & \cdots & u_d & v_j-v_i
        \bdots & \hat u_i & \ndots & \hat u_j & \edots & v_j-v_i
        \\
        & | &  & | & & |
    \end{pmatrix}
    %\\&
    =\notag
    \ell_{ij}\cdot \det\!\begin{pmatrix}
        & | &  & | & & |
        \\%[-0.5ex]
        %u_0 & \cdots & u_d & v_j-v_i
        \bdots & \hat u_i & \ndots & \hat u_j & \edots & \vec e_{ij}
        \\
        & | &  & | & & |
    \end{pmatrix}
    %= 
    %\ell_{ij} \adj_{T_{ij}} 
    \end{align}
    The last determinant agrees with $\ell_{ij}\adj_{T_{ij}}$ in absolute value by \eqref{eq:Dij_form}.
\end{proof}

\section{Computing $D_{ij}$ for ortho-simplices}
\label{sec:appendix_Dij}

Let $D_{ij}$ be the $(d+1)\times(d+1)$-matrix defined in \eqref{eq:Dij_def}.
Following the notation~of \cref{sec:decomposition_proof}, we show that for an ortho-simplex with parameters $\ell_0,...,\ell_{d+1}$ holds
\begin{equation}
    \label{eq:Dij_orthosimplex}
    \det(D_{ij}) = \frac{\ell_{i+1}^2+\cdots+\ell_j^2}{\ell_i\ell_{i+1}\ell_{j}\ell_{j+1}} \cdot \ell_0\cdots\ell_{d+1}.
\end{equation}

% Let $\ell_0,\ldots,\ell_{d+1}\in \RR_{>0}$ and write $e_1,\ldots,e_d$ for the standard basis vectors of $\RR^d$. For convenience we also write $e_0=e_{d+1}=0$.
% For $i\in\{0,\ldots,d\}$ we define \begin{align*}
%     v_i&:=\sum_{j=1}^i\ell_ie_i\in \RR^d \\
%     u_i&:=\ell_ie_{i+1}-\ell_{i+1}e_i \in \RR^d
% \end{align*}

Note that we already know that both sides of \eqref{eq:Dij_orthosimplex} are positive.
Hence we tacitly ignore the signs from now on.
We can then work with $D_{ij}$ in the \mbox{equivalent~form~\eqref{eq:Dij_form}}, but we move the last column $v_j-v_i$ to index $j+1$.
Plugging in the values for~$v_i$~and $u_i$, we are left with computing the determinant of the following block matrix:
\newcommand{\pp}{\phantom+}
\newcommand{\lm}{\mathllap-}
%
\iffalse % first form of matrix
$$
\left(\begin{array}{cccc|ccc|cccc|c}
    \ell_0 & \!\!-\ell_2 & & &
    & & & 
    & & & &
    \\[-1.2ex]
    & \!\!\pp\ell_1 & \!\!\ddots\!\! & &
    & & & 
    & & & &
    \\[-1.2ex]
    & & \!\!\ddots\!\! & \lm\ell_i & 
    & & & 
    & & & &
    \\
    & & & \ell_{i-1} &
    & & & 
    & & & &
    %
    \\[0.7ex] \hline % --------------------
    %
    & & & & 
    -\overset{\phantom.}\ell_{i+2} & & &
    & & & & 
    \ell_{i+1}
    \\[-1.2ex]
    & & & & 
    \pp\ell_{i+1} & \!\!\ddots\!\! & &
    & & & & 
    %\vdots
    \\[-1.2ex]
    & & & & 
    & \!\!\ddots\!\! & \lm\ell_j &
    & & & & 
    \raisebox{0.5em}{\smash{\vdots}}
    \\
    & & & & 
    & & \ell_{j-1} &
    & & & & 
    \ell_j
    %
    \\[0.7ex] \hline % --------------------
    %
    & & & &
    & & &
    -\overset{\phantom.}\ell_{j+2} & & &
    \\[-1.2ex]
    & & & &
    & & &
    \pp\ell_{j+1} & \!\!\ddots\!\! & &
    \\[-1.2ex]
    & & & &
    & & &
    & \!\!\ddots\!\! & \!\!\!\!-\ell_d & &
    \\
    & & & &
    & & &
    & & \!\!\!\!\pp\ell_{d-1} & \!\!-\ell_{d+1} &
\end{array}\right)
$$
%
By moving the last column to the end of the second column block, we obtain a block-diagonal matrix with quadratic blocks of size $i$, $j-i$ and $d-j$.
%
\fi
$$
\left(\begin{array}{cccc|cccc|cccc}
    \ell_0 & \!\!-\ell_2 & & &
    & & & &
    & & &
    \\[-1.2ex]
    & \!\!\pp\ell_1 & \!\!\ddots\!\! & &
    & & & &
    & & &
    \\[-1.2ex]
    & & \!\!\ddots\!\! & \lm\ell_i & 
    & & & &
    & & &
    \\
    & & & \ell_{i-1} &
    & & & &
    & & &
    \\[0.7ex] \hline % --------------------
    & & & & 
    -\overset{\phantom.}\ell_{i+2} & & & \ell_{i+1} &
    & & &
    \\[-1.2ex]
    & & & & 
    \pp\ell_{i+1} & \!\!\ddots\!\! & &
    & & &
    %\vdots
    \\[-1.2ex]
    & & & & 
    & \!\!\ddots\!\! & \lm\ell_j\!\!\!\! & \raisebox{0.5em}{\smash{\vdots}} &
    & & &
    \\
    & & & & 
    & & \ell_{j-1}\!\!\!\! & \ell_j &
    & & &
    \\[0.7ex] \hline % --------------------
    & & & &
    & & & &
    -\overset{\phantom.}\ell_{j+2} & &
    \\[-1.2ex]
    & & & &
    & & & &
    \pp\ell_{j+1} & \!\!\ddots\!\! &
    \\[-1.2ex]
    & & & &
    & & & &
    & \!\!\ddots\!\! & \!\!\!\!-\ell_d &
    \\
    & & & &
    & & & &
    & & \!\!\!\!\pp\ell_{d-1} & \!\!-\ell_{d+1}
\end{array}\right)
$$
Conveniently, this matrix is block-diagonal with square blocks of dimension $i$, $j-i$ and $d-j$ respectively.
Its determinant can therefore be computed by multiplying~the determinants of the diagonal blocks.
Evidently, the determinants of the top-left and bottom-right blocks multiply to $\ell_0\cdots\ell_{i-1}\ell_{j+2}\cdots\ell_{d+1}$ (up to a sign).
If we extract the middle block as
$$
M_{ij}:=
\begin{pmatrix}
    -\overset{\phantom.}\ell_{i+2} & & & \ell_{i+1}
    \\[-1.2ex]
    \pp\ell_{i+1} & \!\!\ddots\!\! &
    %\vdots
    \\[-1.2ex]
    & \!\!\ddots\!\! & \lm\ell_j\!\!\!\! & \raisebox{0.5em}{\smash{\vdots}}
    \\
    & & \ell_{j-1}\!\!\!\! & \ell_j
\end{pmatrix}
$$
it suffices to prove the following:

\begin{lemma}
    $$\det(M_{ij}) = (-1)^{j-i-1}\frac{\ell_{i+1}^2+\cdots+\ell_j^2}{\ell_i\ell_{i+1}\ell_j\ell_{j+1}}\cdot \ell_{i}\cdots \ell_{j+1}.$$
\end{lemma}
The sign is included here because it is needed for the inductive proof.
\begin{proof}
    The proof proceeds by induction on $j-i$. 
    For $j-i=1$ the statement~becomes $\det(\ell_{i+1}) = (-1)^0 \ell_{i+1}$ which clearly holds.
    For $j-i\ge 2$ we apply Laplace expansion to the last row of $M_{ij}$.
    When deleting the row and column of $\ell_{j-1}$ we are left with $M_{i(j-1)}$ whose determinant we know by induction hypothesis.
    When deleting the row and column of $\ell_j$, we are left with a lower-triangular matrix whose determinant is $(-1)^{j-i-1}\ell_{i+2}\cdots\ell_j$.
    The proof then finishes with the following computation:
    \begin{align*}
        \det(M_{ij}) 
            &= -\ell_{j-1} \det(M_{i(j-1)}) + (-1)^{j-i-1}\ell_j \cdot \ell_{i+2}\cdots\ell_j
            \\
            &= (-1)^{j-i-1}{\color{red}\ell_{j-1}} \frac{\ell_{i+1}^2+\cdots+\ell_{j-1}^2}{\ell_i\ell_{i+1}{\color{red}\ell_{j-1}}\ell_{j}}\cdot \ell_{i}\cdots \ell_{j} + (-1)^{j-i-1}\ell_j \cdot \ell_{i+2}\cdots\ell_j
            \\
            &= (-1)^{j-i-1}\Big(\frac{\ell_{i+1}^2+\cdots+\ell_{j-1}^2}{\ell_i\ell_{i+1}\ell_{j}{\color{blue}\ell_{j+1}}}\cdot \ell_{i}\cdots \ell_{j} {\color{blue}\ell_{j+1}} \\&\hspace{10em}+ \frac{\ell_j{\color{ForestGreen}\ell_j}}{\color{ForestGreen}\ell_i\ell_{i+1}\ell_{j}\ell_{j+1}} \cdot {\color{ForestGreen}\ell_i\ell_{i+1}}\ell_{i+2}\cdots\ell_j{\color{ForestGreen}\ell_{j+1}}\Big)
            \\
            &= (-1)^{j-i-1}\frac{\ell_{i+1}^2+\cdots+\ell_j^2}{\ell_i\ell_{i+1}\ell_j\ell_{j+1}}\cdot \ell_{i}\cdots \ell_{j+1}.
            \qedhere
    \end{align*}
\end{proof}

\section{Geometric interpretation of the adjoint of a triangle}
\label{sec:appendix_triangle_adjoint}

The adjoint polynomial of a triangle is of degree zero, that is, it is a number. This number is well-defined if we assume that normal vectors in \eqref{eq:Omega_represenation} are normalized.\nls
It is moreover invariant under translation of the triangle, as can be inferred, for example, from \eqref{eq:trangle_adjoint}.
Here we show that the adjoint of a triangle $\Delta$ can be expressed as
$$\adj_\Delta=\frac{\operatorname{Area}(\Delta)}{\operatorname{Circumradius}(\Delta)}.$$

\begin{figure}[h!]
    \centering
    \includegraphics[width=0.45\linewidth]{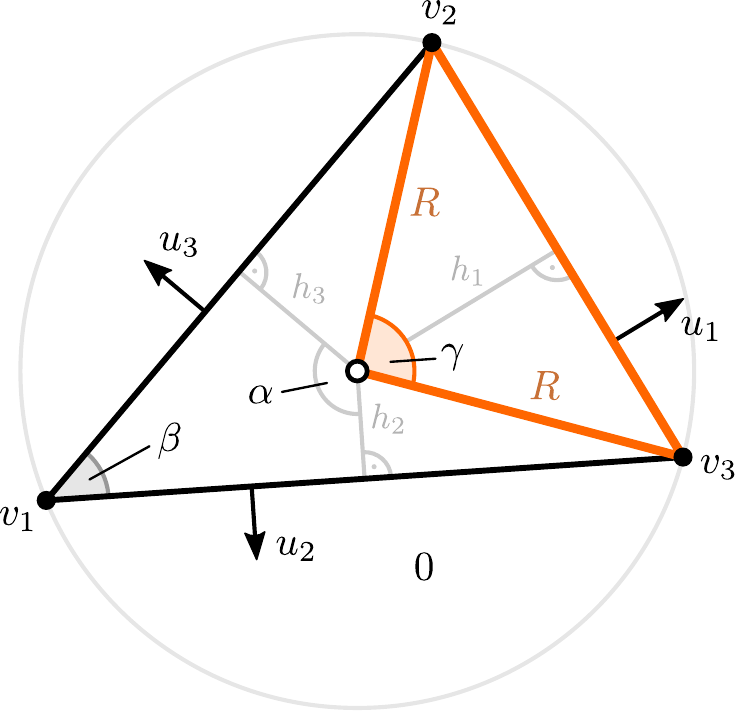}
    \caption{Visualization for the proof of \cref{sec:appendix_triangle_adjoint}.}
    \label{fig:triangle_proof}
\end{figure}

%\msays{To the figure: add origin, add $2\gamm$, add $A_1$, remove $x$. Add a second triangle with three highlighted triangles.}

Let $\Delta$ be a triangle with vertices $v_1,v_2$ and $v_3$.
Let $u_i$ and $h_i$ be the normal vector and height of the edge opposite to $v_i$.
By translation invariance we can assume that $\Delta$ is translated so that its circumcenter is the origin.
The circumradius we denote~by $R$.
The three line segments connecting the origin to the vertices dissect $\Delta$ into three triangles with areas $A_1$, $A_2$ and $A_3$, that sum up to $\operatorname{Area}(\Delta)$.
Below~we~compute~the area $A_1$ (see \cref{fig:triangle_proof}).
% We compute the area $A_1$ of the sub-triangle $T_1$ highlighted in \cref{fig:triangle_proof}. If $\gamma$ is the angle of
%
% By translation invariance we can assume that $\Delta$ is translated so that its circumcenter is the origin.
% Let $v_1,v_2,v_3$ the vertices of $\Delta$, and let $u_i$ and $h_i$ be the normal vector and height of the edge opposite to $v_i$.
% Also set $R:=\operatorname{Circumradius}(\Delta)$.
% Using the determinantal expression for the triangle adjoint (\cf\ \cref{ex:simplices}), we obtain
% %
% \begin{equation*}
%     %\label{eq:adjoint_expansion}
%     \adj_\Delta = \det\!\begin{pmatrix}
%         |\!\!\!\!   & |\!\!\!\!     & |     \\[-0.5ex]
%         u_0\!\!\!\! & u_1\!\!\!\!   & u_2   \\
%         |\!\!\!\!   & |\!\!\!\!     & |     \\[0.5ex]
%         h_0\!\!\!\! & h_1\!\!\!\!   & h_2
%     \end{pmatrix} 
%     %= \sum_{\{i,j,k\}=\{1,2,3\}} h_i\sin\angle(u_j,u_k).
%     = h_1\sin\angle(u_2,u_3)+h_2\sin\angle(u_3,u_1)+h_3\sin\angle(u_1,u_2).
% \end{equation*}
%
%
%Below, we prove that $h_1\sin\angle(u_2,u_3)$ is the area of the triangle spanned by $v_2$, $v_3$ and the origin, divided by $R$. Analogous identities hold for the other terms of the sum. Together they sum up to $\operatorname{Area}(\Delta)/R$.

Let $\alpha:=\angle(u_2,u_3)$ be the angle between the normal vectors $u_2$ and $u_3$.
As shown in \cref{fig:triangle_proof}, the angle $\beta$ at $v_1$ lies opposite to $\alpha$ in a quadrangle with two right angles.
Hence $\beta=\pi-\alpha$.
By the inscribed angle theorem, the angle $\gamma=\angle(v_20v_3)$~is~twice~as large as $\beta$.
Thus, $\gamma=2\pi-2\alpha$.
The highligted triangle $v_20v_3$ is isosceles, and so its area $A_1$ evaluates to
\begin{align*}
    A_1
    =Rh_1\sin(\nicefrac\gamma2)
    =Rh_1\sin(\pi-\alpha)
    =Rh_1\sin(\alpha).
    %\\
    %&
    %\hspace{3em}
    %\implies \frac{A_1}R = \tfrac12 h_1\sin\angle(u_2,u_3).
\end{align*}
%
% The area of $\Delta$ decomposes into six analogously defined triangles, for one of which we just computed the area. Hence $\operatorname{Area}(\Delta)=2(A_1+A_2+A_3)$.
% Analogously we define areas $A_2$ and $A_3$ for which we compute $$A_2/R=\tfrac12 h_2\sin\angle(u_1,u_3)\wideand A_3/R=\tfrac12 h_3\sin\angle(u_1,u_2).$$
% The claim then follows from $\operatorname{Area}(\Delta)=2(A_1+A_2+A_3)$.
Analogous formulas hold for $A_2$ and $A_3$.
The proof then concludes with the following computation:
\begin{align*}
    %\label{eq:adjoint_expansion}
    \adj_\Delta 
    &= \det\!\begin{pmatrix}
        |\!\!\!\!   & |\!\!\!\!     & |     \\[-0.5ex]
        u_0\!\!\!\! & u_1\!\!\!\!   & u_2   \\
        |\!\!\!\!   & |\!\!\!\!     & |     \\[0.5ex]
        h_0\!\!\!\! & h_1\!\!\!\!   & h_2
    \end{pmatrix} 
    %= \sum_{\{i,j,k\}=\{1,2,3\}} h_i\sin\angle(u_j,u_k).
    \\[1ex]&= h_1\sin\angle(u_2,u_3)+h_2\sin\angle(u_3,u_1)+h_3\sin\angle(u_1,u_2)
    %\\[-1.5em]
    \\&= \frac{A_1}R + \frac{A_2}R + \frac{A_3}R = \frac AR.
\end{align*}
For the first equality we used Laplace expansion in the last row and the fact that~the determinant of a $(2\times2)$-matrix is the sine of the angle between its columns.